\documentclass[a4paper,reqno]{amsart}

\usepackage{graphicx}
\usepackage{mathrsfs}
\usepackage{amsfonts}
\usepackage{amssymb}
\usepackage{amsmath}
\usepackage{amsthm}
\usepackage{amscd}
\usepackage[all,2cell]{xy}
\usepackage{color}

\usepackage[pagebackref,colorlinks]{hyperref}
\usepackage{enumerate}

\usepackage{tikz-cd}

\usepackage{comment}
\usepackage{hyperref}
\UseAllTwocells \SilentMatrices

\numberwithin{equation}{section}

\usepackage{comment}

\theoremstyle{plain}
\newtheorem{thm}{Theorem}[section]
\newtheorem{prop}[thm]{Proposition}
\newtheorem{cor}[thm]{Corollary}
\newtheorem{lem}[thm]{Lemma}
\newtheorem{conjecture}{Conjecture}

\theoremstyle{definition}
\newtheorem{defi}[thm]{Definition}
\newtheorem{exm}[thm]{Example}
\theoremstyle{remark}
\newtheorem{rmk}[thm]{\bf Remark}

\def\g{\gamma}
\def\G{\Gamma}

\def\xra{\xrightarrow[]{}}
\def\a{\alpha}
\def\b{\beta}

\def\mG{\mathcal{G}}

\def\VV{\mathcal{V}}
\def \Z{\mathbb Z}
\def\-{\text{-}}
\def\TT{\mathcal{T}}

\def\L{L}

\def\TT{\mathscr{T}}

\newcommand{\FKbar}{\operatorname{\mathrm{F}\ol{\mathrm{K}}}}

\def\VV{\mathcal{V}}
\def\LL{\mathscr{L}}

\newcommand{\ol}{\overline}

\def\sub{\subseteq}

\newcommand{\FK}{\operatorname{FK}}

\newcommand{\BF}{\operatorname{BF}}
\newcommand{\idd}{\operatorname{id}}

\newcommand{\Imm}{\operatorname{Im}}

\newcommand{\Ker}{\operatorname{Ker}}
\newcommand{\Coker}{\operatorname{Coker}}

\def\sub{\subseteq}

\newcommand{\gr}{\operatorname{gr}}

\newcommand{\Mod}{\operatorname{Mod}}
\newcommand{\Gr}{\operatorname{Gr}}

\newcommand{\Spec}{\operatorname{Spec}}
\newcommand{\Prime}{\operatorname{Prime}}

\begin{document}

\title[Graded $K$-Theory, Filtered $K$-theory and the classification of graph algebras]{Graded $K$-Theory, Filtered $K$-theory and \\ the classification of graph algebras}

\subjclass[2010]{16D70,18F30}

\keywords{Leavitt path algebra, graph $C^*$-algebra, graded $K$-theory, filtered $K$-theory, graded prime ideal, graded Grothendieck group}

\date{\today}

\author{Pere Ara}
\address{Pere Ara: Department of Mathematics\\
Universitat Aut\`onoma de Barcelona\\
 08193 Bellaterra (Barcelona)\\And}
\address{\qquad \qquad \quad Barcelona Graduate School of Mathematics (BGSM), Barcelona, Spain}
\email{para@mat.uab.cat}

\author{Roozbeh Hazrat}
\address{Roozbeh Hazrat: 
Centre for Research in Mathematics and Data Sceince\\
Western Sydney University\\
Australia} \email{r.hazrat@westernsydney.edu.au}

\author{Huanhuan Li}
\address{
Huanhuan Li: School of Mathematical Sciences\\
Anhui University, Hefei 230601, Anhui, PR China} \email{lihuanhuan2005@163.com}

\begin{abstract} 
We prove that an isomorphism of  graded Grothendieck groups $K^{\gr}_0$ of two Leavitt path algebras  induces an isomorphism of a certain quotient of algebraic filtered $K$-theory and consequently an isomorphism of filtered $K$-theory of their associated graph $C^*$-algebras. As an application, we show that, since for a finite graph $E$ with no sinks, $K^{\gr}_0\big(L(E)\big)$ of the Leavitt path algebra $L(E)$ coincides with Krieger's dimension group of its adjacency matrix $A_E$, our result relates the shift equivalence of graphs to the filtered $K$-theory and consequently gives that two arbitrary shift equivalent matrices give stably isomorphic graph $C^*$-algebras. This result was only known for irreducible graphs. 


 \end{abstract}

\maketitle


\section{introduction}

 One of the beauties of the theory of Leavitt path algebras is that one can obtain a substantial amount of information about the structure of the algebra from the geometry of its associated graph. The first theorem proved in this theory was that the simplicity of a Leavitt path algebra is equivalent to that in the associated graph every cycle has an exit and every vertex connects to every infinite path and every finite path ending in a sink (\cite{ap}, \cite[\S2.9]{lpabook}).  
 
 The theory of Leavitt path algebras is intrinsically related, via graphs, to the theory of symbolic dynamics and $C^*$-algebras where the major classification programs have been a domain of intense research in the last 50 years.  However, it is not yet clear what is the right invariant for the classification of Leavitt path algebras, and for that matter, graph $C^*$-algebras~\cite{tomforde2}. In the case of simple graph $C^*$-algebras (i.e., algebras with no nontrivial ideals), it is now established that $K$-theory functors $K_0$ and $K_1$ can classify these algebras completely~\cite{ror2,ror3}. Following the early work of R\o rdam~\cite{ror}  and Restorff~\cite{restorff},  it became clear that one way to preserve enough information in the presence of ideals in a $C^*$-algebra, is to further consider the $K$-groups of the ideals, their subquotients and how they are related to each other via the six-term sequence. Over the next ten years since \cite{ror,restorff} this approach, which is now called \emph{filtered $K$-theory},  was subsequently investigated and further developed by Eilers, Restorff, Ruiz and S\o rensen~\cite{errs4,errs2}, where it was shown that the sublattice of gauge invariant prime ideals and their subquotient $K$-groups can be used as an invariant.  In a major work~\cite{errs3}  it was shown that filtered $K$-theory is a complete invariant for unital graph $C^*$-algebras. In~\cite{errs} the four authors introduced the filtered $K$-theory in the purely algebraic setting and showed that if two Leavitt path algebras with coefficients in complex numbers $\mathbb C$ have isomorphic filtered algebraic $K$-theory then the associated graph $C^*$-algebras have isomorphic filtered $K$-theory.
 
 This paper is devoted to  \emph{graded $K$-theory} as a capable invariant for the classification of graph algebras. This approach was initiated in~\cite{roozbehhazrat2013} and further studied in~\cite{arapardo,haz2013,haz3}. The main aim of this paper is to show that in the setting of graph algebras, graded $K$-theory determines a large portion of filtered $K$-theory. To be precise, we show that for two Leavitt path algebras over a field, if their graded Grothendieck groups $K^{\gr}_0$ are isomorphic, then certain precisely defined quotients of their filtered $K$-theories are also isomorphic. This shows the richness of a graded Grothendieck group as an invariant. Namely, the single group $K_0^{\gr}(L_k(E))$ of the Leavitt path algebra $L_k(E)$ associated to a graph $E$, with coefficients in a field $k$, contains all the information about the $K_0$ groups and the quotient groups $\ol{K}_1$ of $K_1$ defined below of the subquotients of graded ideals of $L_k(E)$, and how they are related via the long exact sequence of $K$-theory. For the sake of precision, let us point out that the graded Grothendieck group $K_0^{\gr}$ is a {\it graded invariant} of Leavitt path algebras, but not an algebra invariant: isomorphic Leavitt path algebras may have non-isomorphic graded Grothendieck groups.

A crucial first step in relating graded $K$-theory to the filtered $K$-theory is the realisation that there is a tight connection between the algebraic structure of $L_k(E)$ and the monoid structure of $\VV^{\gr}(L_k(E))$, i.e., the positive cone of $K_0^{\gr}(L_k(E))$. 
Namely, we observe that  the lattice of graded (prime) ideals of $L_k(E)$ is isomorphic to the lattice of order (prime) ideals of $\VV^{\gr}(L_k(E))$. This allows us to lift an 
order-preserving $\mathbb Z[x,x^{-1}]$-module isomorphism between the graded Grothendieck groups of two Leavitt path algebras
\begin{equation}\label{fdgfdrge1}
\varphi: K_0^{\gr}(L_k(E)) \longrightarrow  K_0^{\gr}(L_k(F)),
\end{equation}
to a natural homeomorphism  between the space of spectrums of their graded prime ideals,
\[\varphi: \Spec^{\gr}(L_k(E)) \longrightarrow  \Spec^{\gr}(L_k(F)).\]
   We then proceed to piece together $K_0$ and $K_1$ groups of ideals together from this correspondence. This is possible as we can establish a van den Bergh type exact sequence relating $K^{\gr}_0$ to $K_0$ and $K_1$ (Proposition~\ref{someother}),  
\begin{equation}
\xymatrix{
K_1(\L_k(E)) \ar[r]& K_0^{\gr}(\L_k(E)) \ar[r]& K_0^{\gr}(\L_k(E)) \ar[r]& K_0(\L_k(E))\ar[r]&0.}
\end{equation} 

The group $K_1(L_k(E))$ splits in a non-canonical way as
$$K_1(L_k(E)) \cong {\rm Coker} \bigg( A_E^t-I: (k^{\times})^{R}\longrightarrow (k^{\times})^{E^0}\bigg)\bigoplus \Ker \bigg( A^t-I : \Z^{R}\longrightarrow \Z^{E^0}\bigg),$$
where $A_E$ is the adjacency matrix of the graph $E$ (see Lemma~\ref{k0k1}(ii)).    
In order to get a canonical splitting, which is functorial with respect to the maps induced by graded subquotients, we consider the group 
$$\ol{K}_1(L_k(E))= K_1(L_k(E))/G_E,$$ (see \eqref{definition-quotient}).
Here $G_E$ is the subgroup of $K_1(L_k(E))$ generated by the elements  $[-v + (1-v)]_1$ for all $v\in E^0$, with $1$ the unit for the unitization of $L_k(E)$. Corresponding groups $\ol{K}_1(J/I)$ can also be defined for all graded subquotients $J/I$ of $L_k(E)$. Using these groups, and setting $\ol{K}_0(J/I)=K_0(J/I)$, we define a quotient $\FKbar_{0,1}(L_k(E))$ of algebraic filtered $K$-theory $\FK_{0,1}(L_k(E))$ (see Definition \ref{def:FKbar}).


We then show that the isomorphism $\varphi$ of (\ref{fdgfdrge1}) induces a commutative diagram involving the natural transformation of filtered $\ol{K}$-theory (Theorem~\ref{maintheorem}). Namely for graded  ideals $I \subseteq J\subseteq P$ of $L_k(E)$, one can obtain induced isomorphisms  $\alpha_{J/I, n}: \overline{K}_n(J/I)\xra \overline{K}_n(\widetilde{J}/\widetilde{I})$, $\alpha_{P/I, n}: \overline{K}_n(P/I)\xra \overline{K}_n(\widetilde{P}/\widetilde{I})$, and $\alpha_{P/J, n}: \overline{K}_n(P/J)\xra \overline{K}_n(\widetilde{P}/\widetilde{J})$, for $n=0,1$, where $\widetilde{I}=\varphi(I), \widetilde{J}=\varphi(J)$ and $\widetilde{P}=\varphi(P)$, such that the following diagram commutes:
\begin{equation*}
\xymatrix@C+.1pc@R+1.1pc{
\overline{K}_1(J/I)\ar[r]^{}\ar[d]^{\alpha_{J/I, 1} \hskip .2in{\txt{\tt } \hskip -.4in}} \ar[r]^{}& \overline{K}_1(P/I)\ar[r]^{} \ar[d]^{\alpha_{P/I, 1} \hskip .2in{\txt{ \tt } \hskip -.4in}}& \overline{K}_1(P/J)\ar[r]^{} \ar[d]^{\alpha_{P/J, 1} \hskip .2in{\txt{ \tt } \hskip -.4in}} & K_0(J/I)\ar[r]^{}\ar[d]^{\alpha_{J/I, 0} \hskip .2in{\txt{ \tt } \hskip -.4in}}&K_0(P/I)\ar[r]^{}\ar[d]^{\alpha_{P/I, 0} \hskip .2in{\txt{ \tt } \hskip -.4in}}&K_0(P/J)\ar[d]^{\alpha_{P/J, 0}}\\
\overline{K}_1(\widetilde{J}/\widetilde{I})\ar[r]^{}& \overline{K}_1(\widetilde{P}/\widetilde{I})\ar[r]^{} & \overline{K}_1(\widetilde{P}/\varphi(J))\ar[r]^{} & K_0(\widetilde{J}/\widetilde{I})\ar[r]^{}&K_0(\widetilde{P}/\widetilde{I})\ar[r]^{}&K_0(\widetilde{P}/\widetilde{J}).
}
\end{equation*}
Here the rows come from the long exact sequence of algebraic $K$-theory (see \cite[Theorem 2.4.1]{cortinas}). 

Since the isomorphism of graded $K$-theory of Leavitt path algebras with coefficients in complex numbers $\mathbb C$ implies the isomorphisms of filtered $K$-theory of their corresponding graph $C^*$-algebras (Theorem \ref{lateradd}), we relate the shift equivalent matrices via Krieger's dimension groups to graded $K$-theory and in return to algebraic and thus analytic filtered $K$-theory (\S\ref{gdtbryr777}). Then invoking  the Eilers, Restorff, Ruiz and S\o rensen recent result~\cite{errs3} on the classification of finite graph algebras via filtered $K$-theory,  we can conclude that shift equivalent matrices have stably isomorphic Cuntz-Krieger $C^*$-algebras (Proposition~\ref{bfg1998d}).
This was only known in the case of irreducible matrices by a combination of the Franks Theorem on the classification of flow equivalence via the  Bowen-Franks group~\cite{franks}, Parry and Sullivan's description of flow matrices in terms of moves~\cite{parrysullivan}  and Bates and Pask's paper~\cite{batespask} translating these moves into the setting of graph $C^*$-algebras. 

The following diagram summarizes the connections between the graphs $E$ and $F$, their adjacency matrices $A_E$ and $A_F$, their associated graph algebras and their graded and filtered $K$-theories.

\begin{equation}\label{tgftgrtgete3}
\xymatrix{
A_E \ar@{<->}[rrrr]^{\text{shift equivalence of matrices}}  && \ar@2{<->}[d]&& A_F\\
\Delta_E \ar@{<->}[rrrr]^{\text{\, iso. of Krieger's \, \, \, dimension groups}} && \ar@2{<->}[d] && \Delta_F\\
K_0^{\gr}(L_k(E)) \ar@{<->}[rrrr]^{\text{\, \, \, iso. of  \, \, \, \, $K_0^{\gr}$-group}} &&\ar@[blue]@2{->}[d]&& K_0^{\gr}(L_k(F))\\
\FKbar_{0,1}(L_k(E)) \ar@{<->}[rrrr]^{\text{\, iso. of alg. \,  \, fil. $K$-groups}} &&\ar@2{->}^{\text{ for } k=\mathbb C}[d]&& \FKbar_{0,1}(L_k(F)) \\
\FK_{0,1}(C^*(E)) \ar@{<->}[rrrr]^{\text{\, \, \, \, \, iso. of  \, \, \,fil.  $K$-groups}} &&\ar@2{<->}[d]&& \FK_{0,1}(C^*(F)) \\
C^*(E)  \ar@{<->}[rrrr]^{\text{\, \, Morita \, \, \, \, equivalent}}_{\text{}} &&&& C^*(F)
}
\end{equation}

This paper is devoted to row-finite graphs (graphs such that each vertex emits a finite number of edges). In the presence of infinite emitters, the monoid of a Leavitt path algebra is more involved (see Remark~\ref{hgftgftrgt22}).  Although the majority of the techniques we employ are valid for arbitrary graphs, so that consequently the statements we establish hold for arbitrary graphs (albeit with more complex and lengthier proofs), there are several instances where the techniques for such graphs need to be yet established, such as working with quotient monoids.
This paper is thus devoted to the row-finite case. Our major application is related to symbolic dynamics and it only requires working with finite graphs and thus we are equipped to employ our results in this setting.

We close the introduction by recalling a conjecture posed in \cite[Conjecture~1]{roozbehhazrat2013}, namely the graded Grothendieck group $K_0^{\gr}$ along with its ordering and its module structure is a complete invariant for graded Morita equivalence for the class of (finite) Leavitt path algebras (see also~\cite{arapardo,haz2013}, \cite[\S~7.3.4]{lpabook}).  The language of homology of groupoids allows us to propose a single invariant which would classify both Leavitt and graph $C^*$-algebras. For an arbitrary graph $E$ and its associated graph groupoid $\mG_E$, it was proved  in \cite{hl2} that 
there is an order-preserving $\mathbb Z[x,x^{-1}]$-module isomorphism 
\begin{align*}
K_0^{\gr}(L_k(E)) &\longrightarrow  H_0^{\gr}(\mG_E),\\
[L_k(E)] &\longmapsto [1_{\mG_E^{(0)}}]
\end{align*}
where $H_0^{\gr}(\mG_E)$ is the zeroth graded homology group of the \'etale graph groupoid $\mG_E$. We can then formulate the following conjecture. 

\begin{conjecture}{\cite[Conjecture~1]{roozbehhazrat2013}}  \label{goldenconjecture1}
Let $E$ and $F$ be finite graphs and $k$ a field. Then the following are equivalent.

\begin{enumerate}[\upshape(1)]

\item There is a gauge preserving isomorphism $\varphi: C^*(E) \rightarrow  C^*(F)$;

\smallskip

\item There is a graded ring isomorphism $\varphi:L_k(E) \rightarrow L_k(F)$;

\smallskip

\item There is an order-preserving $\Z[x,x^{-1}]$-module isomorphism 
$\phi: H_0^{\gr}(\mG_E) \rightarrow H_0^{\gr}(\mG_F)$, such that $\varphi([1_{\mG_E^{(0)}}])=[1_{\mG_F^{(0)}}]$. 

\end{enumerate}

\end{conjecture}

It is further conjectured that the ordered $\mathbb Z[x,x^{-1}]$-module isomorphism $H_0^{\gr}(\mG_E) \rightarrow H_0^{\gr}(\mG_F)$ (without preserving the order units) should give that these algebras are graded Morita equivalent (\cite[Remark~16]{haz3}).

Along the paper we will use basic concepts from algebraic $K$-theory, both in the graded and the non-graded settings. We refer the reader to \cite{cortinas}, \cite{haz} and \cite{rosenberg} for background on these concepts.

\section{The lattice of ideals of graph algebras}\label{monidgfhf}

\subsection{Lattices}\label{laikfgut6}

Throughout the paper we work with the lattice of graded ideals of a Leavitt path algebra and the corresponding lattice of graded order ideals of the associated graded Grothendieck group. They are linked via the lattice of admissible pairs of the graph which defines the algebra (see~\S\ref{lpabasis}). Recall that a lattice is a partially ordered set $M$ which any two elements $a,b\in M$ have \emph{meet} $a \wedge b $ (the greatest lower bound) and \emph{join} $a\vee b$ (the least upper bound).  A morphism between two lattices is a map which preserves meets and joins and thus preserves the order. An element $x\in M$ is called a \emph{prime} element if for any $a,b \in M$ with $a\wedge b \leq x$, we have $a\leq x$ or $b \leq x$. 
We will work with the prime elements in the lattice of graded ideals of a Leavitt path algebra $L(E)$ (\S\ref{lpabasis}) and the lattice of graded order ideals of $\VV^{\gr}(L(E))$ (Definition \ref{orfdregeret54}). Clearly an isomorphism between lattices preserves prime elements.

\subsection{Graphs}\label{}
Below we briefly recall the notions that we use throughout the paper. 

A (directed) graph $E$ is a tuple $(E^{0}, E^{1}, r, s)$, where $E^{0}$ and $E^{1}$ are
sets and $r,s$ are maps from $E^1$ to $E^0$. A graph $E$ is finite if $E^0$ and $E^1$ are both finite. We think of each $e \in E^1$ as an edge 
pointing from $s(e)$ to $r(e)$. 
We use the convention that a (finite) path $p$ in $E$ of length $n\geq 1$ is
a sequence $p=\a_{1}\a_{2}\cdots \a_{n}$ of edges $\a_{i}$ in $E$ such that
$r(\a_{i})=s(\a_{i+1})$ for $1\leq i\leq n-1$. We define $s(p) = s(\a_{1})$, and $r(p) =
r(\a_{n})$. A vertex is viewed as a path in $E$ of length $0$. If there is a path from a vertex $u$ to a vertex $v$, we write $u\ge v$. A subset $M$ of $E^0$ is \emph{downward directed}  if for any two $u,v\in M$ there exists $w\in M$ such that $u\geq w$ and $v\geq w$ (\cite[\S4.2]{lpabook}, \cite[\S2]{rangaswamy}).

A graph $E$ is said to be \emph{row-finite} if for each vertex $u\in E^{0}$,
there are at most finitely many edges in $s^{-1}(u)$. A vertex $u$ for which $s^{-1}(u)$
is empty is called a \emph{sink}. 

The \emph{covering graph}  $\overline{E} = E\times_1 \mathbb Z$  of $E$ is given by
\begin{gather*}
    \overline E^0 = \big\{v_n \mid v \in E^0 \text{ and } n \in \Z \big\},\qquad
    \overline E^1 = \big\{e_n \mid e\in E^1 \text{ and } n\in \Z \big\},\\
    s(e_n) = s(e)_n,\qquad\text{ and } \qquad  r(e_n) = r(e)_{n-1}.
\end{gather*}

As examples, consider the following graphs
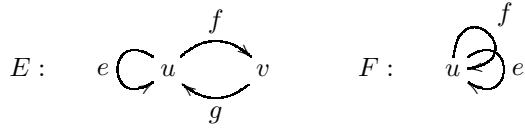
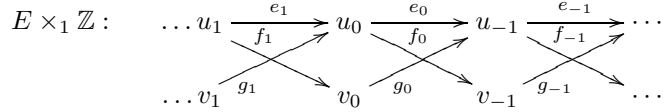
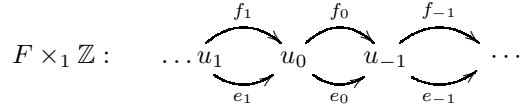
\begin{equation*}
{\def\labelstyle{\displaystyle}
E : \quad \,\, \xymatrix{
 u \ar@(lu,ld)_e\ar@/^0.9pc/[r]^f & v \ar@/^0.9pc/[l]^g
 }} \qquad \quad
{\def\labelstyle{\displaystyle}
F: \quad \,\, \xymatrix{
   u \ar@(ur,rd)^e  \ar@(u,r)^f
}}
\end{equation*}
Then
\begin{equation*}
E\times_1 \mathbb Z: \quad \,\,\xymatrix@=15pt{
\dots  {u_{1}} \ar[rr]^-{e_1} \ar[drr]^(0.4){f_1} &&  {u_{0}} \ar[rr]^-{e_0} \ar[drr]^(0.4){f_0} && {u_{-1}}  \ar[rr]^-{e_{-1}} \ar[drr]^(0.4){f_{-1}} && \cdots\\
\dots {v_{1}}   \ar[urr]_(0.3){g_1} && {v_{0}} \ar[urr]_(0.3){g_0}  && {v_{-1}}  \ar[urr]_(0.3){g_{-1}}&& \cdots
}
\end{equation*}
and
\begin{equation*}
F \times_1 {\mathbb Z}: \quad \,\,\xymatrix@=15pt{
\dots  {u_{1}} \ar@/^0.9pc/[r]^{f_1} \ar@/_0.9pc/[r]_{e_1}  &  {u_{0}} \ar@/^0.9pc/[r]^{f_0} \ar@/_0.9pc/[r]_{e_0} & {u_{-1}}  \ar@/^0.9pc/[r]^{f_{-1}}  \ar@/_0.9pc/[r]_{e_{-1}} & \quad \cdots
}
\end{equation*}

Throughout the note $E$ is a row-finite graph. Recall that a subset $H \subseteq E^0$ is said to be \emph{hereditary} if
for any $e \in E^1$ we have that $s(e)\in H$ implies $r(e)\in H$. A hereditary subset $H
\subseteq E^0$ is called \emph{saturated} if whenever $v$ is not a sink, $\{r(e):
e\in E^1 \text{~and~} s(e)=v\}\subseteq H$ implies $v\in H$. We let $\TT_E$ denote the set of hereditary saturated subsets of $E^0$, and order two hereditary saturated subsets $H$ and $H'$ by $H\leq H'$ if  $H\sub H'$. It has been established that the ordered set $\TT_E$ is actually a lattice (see \cite[Proposition~2.5.6]{lpabook}).


We denote by $E_H$ the \emph{restriction graph} with $H$ a hereditary saturated subset of $E^0$ such that
$$E^0_H= H,$$
$$E^1_H=\{e\in E^1\;|\;  s(e) \in H\}$$ and we restrict $r$ and $s$ to $E^1_{H}$. On the other hand, for a hereditary saturated subset of $E^0$, we denote by $E/H$ the \emph{quotient graph} such that 
$$(E/H)^0=E^0\setminus H,$$
$$(E/H)^1= \{e\in E^1\;|\; r(e)\notin H\}$$ and we restrict $r$ and $s$ to $(E/H)^1$.

For hereditary saturated subsets $H_1$ and $H_2$ of $E$ with  $H_1 \subseteq  H_2$, define the quotient graph $H_2 / H_1 $ as a graph such that 
$(H_2/ H_1)^0=H_2\setminus H_1$ and $(H_2/H_1)^1=\{e\in E^1\;|\; s(e)\in H_2, r(e)\notin H_1\}$. The source and range maps of $H_2/H_1$ are restricted from the graph $E$. If $H_2=E^0$, then $H_2/H_1$ is the \emph{quotient graph} $E/H_1$ (\cite[Definition~2.4.11]{lpabook}).

\subsection{Cohn algebras and Leavitt path algebras}\label{lpabasis}

We refer the reader to \cite{lpabook} for concepts of Cohn path algebras and Leavitt path algebras.

Let $E$ be a row-finite graph and $k$ a field. 
The \emph{Cohn path algebra} $C_k(E)$ of $E$ is the quotient of the free associative $k$-algebra generated by the set $E^0\cup E^1 \cup \{e^*\;|\;  e\in  E^1\}$, subject to the relations:
\begin{itemize}
\item[(0)] $v\cdot w = \delta_{v, w}$ for $v, w\in E^0$;
\item[(1)] $s(e)\cdot e = e = e \cdot r(e)$ for $e\in E^1$;
\item[(2)]  $r(e) \cdot e^* = e^* = e^*\cdot s(e)$ for $e\in E^1$;
\item[(3)]  $e^*\cdot f = \delta_{e, f} r(e)$ for $e, f\in E^1$.
\end{itemize} The algebra $C_k(E)$ is in fact a $*$-algebra; it is equipped with an involution $*: C_k(E)\xra C_k(E)^{\rm op}$ which fixes vertices and maps $e$ to $e^*$ for $e\in E^1$. Here $C_k(E)^{\rm op}$ is the opposite algebra of $C_k(E)$.

Denote by $\mathcal{K}(E)$ the ideal of the Cohn path algebra $C_k(E)$ generated by the set 
\begin{equation}\label{koutyut}
\big \{v-\sum_{v\in s^{-1}(v)}ee^*\;|\; v\in E^0 \text{~is not a sink} \big\}.
\end{equation}
 The \emph{Leavitt path algebra} $L_k(E)$ of $E$ over the field $k$ (see \cite[Definition~1.2.3]{lpabook}) is the quotient algebra $C_k(E)/\mathcal{K}(E)$. Then there is a short exact sequence of rings
\begin{align}
 \label{sescohn}
 \CD
 0@>>> \mathcal{K}(E)@>{l}>>C_k(E)@>{p}>>L_k(E)@>>>0.
\endCD
\end{align} 

Throughout this paper we simply write $C(E)$ instead of $C_k(E)$ and $L(E)$ instead of $L_k(E)$. These algebras are naturally $\mathbb Z$-graded and this graded structure plays an important role in this paper (see \cite[\S2.1]{lpabook}).

Denote by $\LL^{\gr}\big(\L(E)\big)$ the lattice of graded (two-sided) ideals of $L(E)$. There is a lattice isomorphism between the set $\TT_E$ of hereditary saturated subsets of $E^0$ and the set $\LL^{\gr}\big(\L(E)\big)$ (\cite[Theorem~2.5.8]{lpabook}). The correspondence is  
 \begin{align}\label{latticeisosecideal}
 \Phi: \TT_E&\longrightarrow \LL^{\gr}\big(\L(E)\big),\\ 
 H &\longmapsto \langle H\rangle, \notag 
 \end{align} 
 where $H$ is a hereditary saturated subsets of $E^0$, and $\langle H\rangle$ is the graded ideal generated by the set $H$.

For graded ideals $I,J$ of $L(E)$, since  $I J= I\cap J$, the prime elements of the lattice $\LL^{\gr}\big(\L(E)\big)$ coincide with the graded prime ideals of $L(E)$ (see \ref{laikfgut6}). We denote by $\Spec^{\gr}(\L(E))$ the set of graded prime ideals of $\L(E)$. 
The prime ideals of the Leavitt path algebra $\L (E)$ are completely characterised in terms of their generators (\cite[Theorem 3.12]{rangaswamy}).
Denote by $\TT'_{E}$ the set consisting of the hereditary saturated subsets $H$ of $E^0$ such that $E^0\setminus H$ is downward directed.
Then the correspondence $\Phi$ of (\ref{latticeisosecideal}) restricts to the one-to-one correspondence
\begin{align}\label{latticeisosecidealpri}
 \Phi: \TT'_E&\longrightarrow \Spec^{\gr}(\L(E)),\\ 
 H&\longmapsto \langle H\rangle. \notag 
 \end{align}


\section{The monoid $\VV^{\gr}$ and the graded Grothendieck group $K_0^{\gr}$ of graph algebras}

Let $M$ be a commutative monoid with a group $\Gamma$ acting on it.  Throughout we assume that the group $\Gamma$ is abelian. Indeed in our setting of graph algebras, this group is $\mathbb Z$. We define an ordering on the monoid $M$ by $a\leq b$ if $b=a+c$, for some $c\in M$. 
 A $\Gamma$-\emph{order ideal} of a monoid $M$ is a  subset $I$ of $M$ such that for any $\alpha,\beta \in \Gamma$, ${}^\alpha a+{}^\beta b \in I$ if and only if 
$a,b \in I$. Equivalently, a $\Gamma$-order ideal is a submonoid $I$ of $M$ which is closed under the action of $\Gamma$ and it  is
\emph{hereditary} in the sense that $a \le b$ and $b \in I$ implies $a \in I$. The set
$\LL(M)$ of $\Gamma$-order ideals of $M$ forms a (complete) lattice (see \cite[\S5]{amp}). 

Let $G$ be the \emph{group completion} $M^+$ of a commutative monoid $M$. The action of $\Gamma$ on $M$ lifts to an action on $G$. There is a natural monoid 
homomorphism $\phi: M \rightarrow G$ and by $M_{+}$ we denote the image of $M$ under this homomorphism. The monoid $M_{+}$ is called the \emph{positive cone} of $G$, 
and induces a pre-ordering on $G$. We say that $I\subseteq G$ is a $\Gamma$-\emph{order ideal} of $G$ if $I=I_{+}-I_{+}$, where $I_{+}=I\cap M_{+}$ and $I_{+}$ is a $\Gamma$-order ideal of $M_{+}$.  It is not difficult to see that there is a lattice isomorphism between the $\Gamma$-order  ideals of $M_{+}$ and $\Gamma$-order ideals of $G$.  In our setting (i.e., the monoid of graded finitely generated projective modules) the monoid homomorphism $\phi: M\rightarrow G$ is injective and thus we can work with the lattice of $\Gamma$-order ideals of $M$. 

Let $I$ be submonoid of the monoid $M$. Define an equivalence relation $\sim_{I}$ on $M$ as follows: For $a, b\in M$,  $a\sim_{I} b$ if there exist $i,j\in I$ such that  $a+i=b+j$ in $M$. The quotient monoid $M/I$ is defined as $M/\sim$. Observe that $a\sim_{I} 0$ in $M$ for any $a\in I$. If $I$ is an order-ideal then $a\sim_I 0$ if and only if $a\in I$.

\subsection{The monoid $\VV^{\gr}$}

For a $\Gamma$-graded ring $A$ with identity, the isomorphism classes of graded finitely generated projective modules with the direct sum $[P] + [Q] = [P\oplus Q]$ as the addition operation constitute a monoid denoted by  $\VV^{\gr}(A)$. There is an action of the group $\Gamma$ on $\VV^{\gr}(A)$ via the shifting of the modules
$${}^\alpha [P] \longmapsto [P(\alpha)].$$
Here for a $\Gamma$-graded module $P=\bigoplus_{\gamma \in \Gamma}P_\gamma$, the $\alpha$-\emph{shifted} graded module $P(\alpha)$ is defined as $P(\alpha):=\bigoplus_{\gamma\in \Gamma}P(\alpha)_{\gamma}$, 
where  $P(\alpha)_{\gamma}=P_{\gamma+\alpha}$. Denote by $ A\-\Gr$ the category of graded left $A$-modules. For $\alpha\in\G$, the \emph{shift functor}
\begin{equation}
\label{shift}
\mathcal{T}_{\alpha}: A\text{-}{\Gr}\longrightarrow A\text{-}{\Gr},\quad M\mapsto M(\alpha)
\end{equation}
is an isomorphism with the property $\mathcal{T}_{\alpha}\mathcal{T}_{\beta}=\mathcal{T}_{\alpha+\beta}$
for $\alpha,\beta\in\Gamma$.

The group completion of $\VV^{\gr}(A)$ is called the \emph{graded Grothendieck group} $K_0^{\gr}(A)$. It naturally inherits the action of $\Gamma$. 
This is a pre-ordered abelian group and as above the monoid $K_0^{\gr}(A)_{+}$ consisting of isomorphism classes of graded finitely generated projective $A$-modules is the cone 
of the ordering (see \cite[\S3.6]{haz}). 

Denote by $A\-\Mod$ the category of left $A$-modules. 
The forgetful functor $U: A\-\Gr \rightarrow A\-\Mod$ (forgetting the graded structure) induces a homomorphism, $U: \VV^{\gr}(A)\xra \VV(A)$. Since the (graded) Grothendieck groups are the group completion of these monoids, the homomorphism $U$ extends to the homomorphism of groups $U:K_0^{\gr}(A)\xra K_0(A)$.

Note that if $\Gamma$ is trivial, the theory reduces to the classical (non-graded) theory, and $K_0^{\gr}(A)$  becomes the Grothendieck group $K_0(A)$ (\cite{goodearlbook,ls}).

When the ring $A$ is not unital, one can define $\VV^{\gr}(A)$ via idempotent matrices over $A$. We refer the reader to \cite{haz} for a comprehensive introduction to graded ring theory and the graded Grothendieck groups. 

For the case of a Leavitt path algebra $L(E)$ which is a $\mathbb Z$-graded algebra, we can describe the $\mathbb Z$-monoid $\VV^{\gr}(L(E))$ directly from the graph $E$. 
We do this first for $\VV(L(E))$ and then proceed to give the graded version. 

A commutative monoid $M_E$ associated to a directed row-finite graph $E$ was constructed in \cite{amp}. The monoid $M_E$ is generated by vertices $v\in E^0$ subject to relations 
\begin{equation} 
v=\sum_{e\in s^{-1}(v)} r(e),
\end{equation} for each $v\in E^0$ which is not a sink. It was proved that (see \cite[Theorem 3.5]{amp}) the natural map
\begin{align*}
M_E &\longrightarrow \VV(\L(E)),\\
v &\longmapsto [L(E)v],
\end{align*}
induces a monoid isomorphism.

There is an explicit description \cite[\S 4]{amp} of the congruence on the free commutative
monoid given by the defining relations of $M_{E}$ with $E$ a row-finite graph. Let $F$ be the free commutative monoid on
the set $E^{0}$. The nonzero elements of $F$ can be written in a unique form up to
permutation as $\sum_{i=1}^{n}v_{i}$, where $v_{i}\in E^{0}$. Define a binary relation
$\xra_{1}$ on $F\setminus\{0\}$ by $\sum_{i=1}^{n}v_{i}\xra_{1}\sum_{i\neq
j}v_{i}+\sum_{e\in s^{-1}(v_{j})}r(e)$ whenever $j\in \{1, \cdots, n\}$ is such that
$v_{j}$ is not a sink. Let $\xra$ be the transitive and reflexive closure of $\xra_{1}$
on $F\setminus\{0\}$ and $\sim$ the congruence on $F$ generated by the relation $\xra$.
Then $M_{E}=F/\sim$.

\begin{rmk}
Observe that for $\a, \b\in F\setminus\{0\}$, we have $\a \sim \b$ if and only if there is a finite string $\a = \a_0,\a_1,\cdots, \a_n =\b$, such that, for each $i = 0, \cdots , n-1$, either $\a_i\xra_1 \a_{i+1}$ or $\a_{i+1}\xra_1 \a_i$. The number $n$ above will be called the length of the string.
\end{rmk}

%

We state the following lemma, given in \cite[Lemma 4.3]{amp}, for later use.

\begin{lem}\label{monoidproperty}
Let $E$ be a row-finite graph, $F$ the free commutative monoid generated by $E^0$ and $M_E$ the monoid of the graph $E$.
For $\alpha, \beta\in F\setminus\{0\}$, $\alpha \sim \beta $ in
$F$ if and only if there is $\gamma \in F\setminus\{0\}$ such that $\alpha \to \gamma $ and $\beta \to
\gamma $.
\end{lem}

Next we recall the graded version of $M_E$ with $E$ a row-finite graph, which is a $\mathbb Z$-monoid~\cite[\S5.3]{ahls}. Let $M_{E}^{\gr}$ be an  
commutative monoid  generated by $\{v({i}) \mid  v\in E^{0}, i\in \Z\}$ subject to 
relations 
\begin{equation}
\label{rrr}
{v}({i})=\sum\limits_{e\in s^{-1}(v)}r(e)(i-1)
\end{equation} for $v\in E^{0}$ which is not a sink.

Throughout the paper, we simultaneously use $v\in E^0$ as a vertex of $E$, as an element of $L(E)$ and the element
$v = v(0)$ in $M^{\gr}_E$, as the meaning will be clear from the context.

The monoid $\mathcal{V}^{\gr}(\L(E))$ for a graph $E$ was studied in detail in \cite{ahls}. In the case that $E$ is a row-finite graph,  $\mathcal{V}^{\gr}(\L(E))$
is generated by graded finitely generated projective $L(E)$-modules $\big [\L(E)v(i)\big ]$. In  \cite[Proposition 5.7]{ahls}, $\VV^{\gr}(\L(E))$ and $M_E^{\gr}$ were related via the $\mathbb Z$-monoid isomorphism
%
\begin{align}\label{hfghfyft7}
M^{\gr}_E &\cong  M_{E\times_1 \Z} \cong \mathcal V(L(E\times_1 \Z))\cong\, \mathcal V^{\gr}(L(E)),
\\v(i) &\longmapsto v_i\longmapsto  [L(E\times_1 \Z) v_i ]   \longmapsto\big [(L(E)v\big) (-i)], \notag
\end{align} see \cite[Proposition~5.7]{ahls}. We correct here that the isomorphism $M^{\gr}_E\cong M_{E\times_1 \Z}$ should be given by $v(i)\mapsto v_i$ and that  the $\mathbb Z$-action on $M^{\gr}_E$ given by Equation (5-10) in \cite{ahls} should be ${}^n {v(i)}=v(i-n)$ for $n, i\in\mathbb Z$ and $v\in E^0$.

%

It was proved in \cite[Corollary 5.8]{ahls} that $\VV^{\gr}(\L(E))$ is a cancellative monoid and thus $\VV^{\gr}(\L(E))\xra K_0^{\gr}(\L(E))$ is injective and 
therefore the positive cone $K_0^{\gr}(\L(E))_{+}\cong \VV^{\gr}(\L(E))$. This shows that, in contrast with the non-graded setting, no information is lost going from the graded 
monoid to the graded Grothendieck group and as the morphisms involved are order-preserving, one can formulate the statements either on the level of $K_0^{\gr}$ or the monoid $\VV^{\gr}$.



Suppose that $H_1$ and $H_2$ are hereditary saturated subsets of $E^0$ with $H_1\subseteq H_2$. We denote by $M^{\gr}_{H_i}$ the monoid $M^{\gr}_{E_{H_i}}$ for $i=1,2$. 
We claim that $M^{\gr}_{H_1}$ is a submonoid of $M^{\gr}_{H_2}$. In fact, if $a=b$ in $M^{\gr}_{H_2}\cong M_{E_{H_2}\times_1\Z}$ with $a, b\in M^{\gr}_{H_1}$, by Lemma \ref{monoidproperty} 
there exists $\gamma$ in the free commutative monoid on the set $(E_{H_2}\times_1\Z)^0$ such that $a\xra \gamma$ and $b\xra \gamma$. Since $a, b\in M^{\gr}_{H_1}$, we have that $a$ and $b$ are sums of finitely 
many $v(i)$ with $v\in H_1$ and $i\in\Z$. Since $H_1$ is hereditary, all the binary relations $\xra_1$ appearing in $a\xra \gamma$ and $b\xra\gamma$ are in the free 
monoid generated by $(E_{H_1}\times_1\Z)^0$.  Thus $a=\g=b$ in $M^{\gr}_{H_1}\cong M_{E_{H_1}\times_1\Z}$.
 
 \begin{lem} \label{qmiso} 
 Suppose that $H_1\subseteq H_2$ with $H_1$ and $H_2$ two hereditary saturated subsets of $E^0$. There is a $\mathbb Z$-monoid isomorphism 
\begin{equation}
\label{quotientmonoid}
M_{H_2/H_1}^{\gr}\cong M_{H_2}^{\gr}/M_{H_1}^{\gr}.
\end{equation} In particular, for a hereditary saturated subset $H$ of $E^0$, and the order-ideal $I=\langle H \rangle \subseteq M_E^{\gr}$, we have a $\mathbb Z$-monoid isomorphism 
\[M_{E/H}^{\gr}\cong M_{E}^{\gr}/I.\]
\end{lem}
 
 \begin{proof} We define a homomorphism of monoids $f: M_{H_2}^{\gr}\xra M_{H_2/H_1}^{\gr}$ such that 
\begin{equation}
\label{mapsqutt}
f\big(v(i)\big)=
\begin{cases}
v(i), & \text{if~} v\in H_2\setminus H_1;\\
0, & \text{otherwise}.
\end{cases}
\end{equation} for $v\in H_2$ and $i\in \Z$. Take a vertex $v\in E_{H_2}^0$ which is not a sink. We claim that the relation \eqref{rrr} for the restriction 
graph $E_{H_2}$ is preserved by $f$. If $v\in H_1$ and $i\in \Z$, then $r(e)\in H_1$ for all $e\in s^{-1}(v)$, since $H_1$ is hereditary. Thus $f\big(\sum_{e\in s^{-1}(v)}r(e)(i-1)\big)=0=f\big(v(i)\big)$. If $v\in H_2\setminus H_1$ and $i\in \Z$, then 
\begin{equation*}
\begin{split}
f\big(\sum_{e\in s^{-1}(v)}r(e)(i-1)\big)
&=\sum_{\substack{e\in s^{-1}(v)\\ r(e)\in H_1}}f\big(r(e)(i-1)\big)+\sum_{\substack{e\in s^{-1}(v)\\r(e)\in H_2\setminus H_1}}f\big(r(e)(i-1)\big)\\
&=\sum_{\substack{e\in s^{-1}(v)\\r(e)\in H_2\setminus H_1}}r(e)(i-1)\\
&=v(i)\\
&=f\big(v(i)\big).
\end{split}
\end{equation*} 
Here, the second last equality follows from the relation for the monoid $M^{\gr}_{H_2/H_1}$. Thus $f$ is well-defined. Since $f\big(v(i)\big)=0$ for $v\in H_1$ and $i\in\Z$, we have the induced homomorphism $M^{\gr}_{H_2}/M^{\gr}_{H_1}\xra M^{\gr}_{H_2/H_1}$, still denoted by $f$. 

Now we define a homomorphism of monoids $g: M^{\gr}_{H_2/H_1}\xra M^{\gr}_{H_2}/M^{\gr}_{H_1}$ such that $g\big(v(i)\big)=[v(i)]$ for $v\in H_2\setminus H_1$ and $i\in \Z$. 
Observe that \[v(i)=\sum_{e\in s^{-1}(v), r(e)\in H_2\setminus H_1}r(e)(i-1),\] is in $M^{\gr}_{H_2/H_1}$ for $i\in\Z$ and any vertex $v\in H_2\setminus H_1$ which is not a sink. Thus $g$ is 
well-defined. We can directly check that $g\circ f={\rm id}_{M^{\gr}_{H_2}/M^{\gr}_{H_1}}$ and $f\circ g={\rm id}_{M^{\gr}_{H_2/H_1}}$. This finishes the proof. 
\end{proof}

The following consequence holds immediately.

\begin{cor} Suppose that $H_1$ and $H_2$ are hereditary saturated subsets of $E^0$ with $H_1\subseteq H_2$. 
 There is a short exact sequence of commutative monoids
 \begin{align}\CD
 \label{sesforgradedmonoid}
 0@>>> M^{\gr}_{H_1}@>{\iota}>>M^{\gr}_{H_2}@>{\pi}>>M^{\gr}_{H_2/H_1}@>>>0, 
\endCD
\end{align} where $\iota: M^{\gr}_{H_1}\xra M^{\gr}_{H_2}$ is the inclusion and $\pi: M^{\gr}_{H_2}\xra M^{\gr}_{H_2/H_1}$ is the map given by \eqref{mapsqutt}.
\end{cor}

For a hereditary saturated subset $H$ of $E^0$, 
set $e=\sum_{v\in H}v\in \mathcal M (I)$, where $I=\langle H\rangle$ is the ideal of $L(E)$ generated by $H$ and $\mathcal M (I)$ denotes 
the multiplier algebra of $I$. 
Observe that $e$ is a full idempotent in $\mathcal M(I)$, in the sense that $IeI= I$. Then $I$ is graded Morita equivalent to the Leavitt path algebra $L(E_H)$ (see \cite[Example 2]{haz3} 
and compare \cite[Theorem 5.7 (3)]{tomforde}, where the Morita equivalence is indeed a graded Morita equivalence). It follows that 
\begin{equation}
\label{idealmon}
\VV^{\rm gr}(I)\cong \VV^{\rm gr}(L(E_H))=M^{\rm gr}_H.
\end{equation}

Suppose that $I_1, I_2$ are two graded ideals of $L(E)$ with $I_1\subseteq I_2$. Let $H_{1}$ and $H_2$ be the two hereditary saturated subsets of $E^0$ with $I_1=\langle H_1\rangle$ and $I_2=\langle H_2\rangle$. 
The quotient ideal $I_2/I_1$ is a graded ideal of $L(E)/I_1= L(E/H_1)$. As before, $I_2/I_1$ is graded Morita equivalent to $L((E/H_1)_{H_2/H_1})$ and combining this with \eqref{quotientmonoid} we have 
\begin{equation}
\label{quotientiso}
\VV^{\gr}(I_2/I_1)\cong \VV^{\gr}(L((E/H_1)_{H_2/H_1}))\cong M^{\gr}_{H_2/H_1} \cong  M^{\gr}_{H_2}/M^{\gr}_{H_1} \cong \VV^{\gr}(I_2)/\VV^{\gr}(I_1).
\end{equation}

It is a fact that for a cancellative commutative monoid $M$, if there is a short exact sequence of commutative monoids 
$$\CD
 0@>>> N@>{l}>>M@>{t}>>P@>>>0
\endCD$$ 
with $N$ an order-ideal of $M$, then there is a short exact sequence of their group completions 
$$\CD
 0@>>> N^+@>{\overline{l}}>>M^+@>{\overline{t}}>>P^+@>>>0. 
\endCD$$  By \cite[Corollary 5.8]{ahls} $M^{\gr}_E$ is cancellative for any graph $E$. Combining this with the fact that $M^{\gr}_{H_1}$ is an order-ideal of $M^{\gr}_{H_2}$ if $H_1$ and $H_2$ are hereditary saturated 
subsets of $E^0$ with $H_1\subseteq H_2$, by \eqref{sesforgradedmonoid} we have a short exact sequence  \begin{align}\CD
 \label{shortexactfor}
 0@>>> K_0^{\gr}(L(E_{H_1}))@>{}>>K_0^{\gr}(L(E_{H_2}))@>{}>>K_0^{\gr}(L(H_2/H_1))@>>>0, 
\endCD
\end{align}  where the maps are induced by the homomorphisms of algebras $L(E_{H_1})\xrightarrow {} L(E_{H_2})$ and $L(E_{H_2})\xrightarrow {} L(E_{H_2})/I_1\cong L(H_2/H_1)$. We also have $\overline{\iota}$ and $\overline{\pi}$ to make  the following diagram commute
\begin{align}
\label{2020}
\xymatrix{
0\ar[r] &K_0^{\gr}(I_1) \ar[r]^{\overline{\iota}} \ar[d]^{}& K_0^{\gr}(I_2)\ar[r]^{\overline{\pi}}\ar[d]^{} & K_0^{\gr}(I_2/I_1)\ar[d]\ar[r]&0\\
0\ar[r] &K_0^{\gr}(L(E_{H_1})) \ar[r]^{} & K_0^{\gr}(L(E_{H_2}))\ar[r]&K_0^{\gr}(L(H_2/H_1))\ar[r]&0,}
\end{align} where the vertical isomorphisms are induced by the isomorphisms of the monoids in \eqref{idealmon} and \eqref{quotientiso}.

%


 Denote by $\LL\big(K_0^{\gr}(\L(E))\big)$ and $\LL\big(\VV^{\gr}(\L(E))\big)$ the lattice of $\Z$-order ideals of $K_0^{\gr}(\L(E))$ and $\VV^{\gr}(\L(E))$, respectively.  There is a lattice isomorphism between the set $\TT_E$ of hereditary saturated subsets of $E^0$ and the set $\LL\big(\VV^{\gr}(\L(E))\big)$  (\cite[Theorem 5.11]{ahls}). The correspondence is  
 \begin{align}\label{latticeisosec}
 \Phi: \TT_E&\longrightarrow \LL\big(\VV^{\gr}(\L(E))\big), \\ 
 H &\longmapsto \Phi(H), \notag 
 \end{align} 
 where $H$ is a hereditary saturated subset of $E^0$, and $\Phi(H)$ is the $\Z$-order ideal generated by the set
$$\big \{v(i)\;|\; v\in H, i\in\Z  \big \}.$$ 
This correspondence for finite graphs with no sinks was first established in \cite[Theorem 12]{roozbehhazrat2013}. 

%
Recall from \S\ref{lpabasis} that the lattice of graded ideals of $\L(E)$ is denoted by $\LL^{\gr}(\L(E))$. 
 Combining  the lattice isomorphisms (\ref{latticeisosecideal}) and (\ref{latticeisosec}), there are lattice isomorphisms between the following lattices: 
\begin{equation}\label{latticncnc}
\LL^{\gr}(\L(E)) \cong
\TT_E
    \cong \LL(\VV^{\gr}(\L(E))).\qedhere
\end{equation}

The isomorphisms (\ref{latticncnc}) allow us to relate the prime elements of these lattices. As mentioned in \S\ref{lpabasis}, the prime elements of $\LL^{\gr}(\L(E)) $ 
coincide with the graded prime ideals of $L(E)$.  We give an explicit 
definition of prime elements (\S\ref{laikfgut6}) in the setting of a lattice $\mathcal L (M)$ for a monoid $M$ with a $\Gamma$-action.

\begin{defi}\label{orfdregeret54} 
Let $M$ be a monoid with the group $\G$ acting on it. A $\Gamma$-order ideal $N$ is called \emph{$\Gamma$-prime} if for any $\Gamma $-order ideals $N_1, N_2 \subseteq M$, $N_1\cap N_2 \subseteq N$ implies 
that $N_1\subseteq N$ or $N_2\subseteq N$. 
%
%
%
%
%
%
\end{defi}

We are in a position to relate the spectrum of $L(E)$ with the set of $\Z$-prime order ideals of $\VV^{\gr}(L(E))$.

\begin{thm} 
\label{thmprime}
Let $E$ be a row-finite graph and $\VV^{\gr}(L(E))$ its associated monoid. Then there is a one-to-one order-preserving correspondence between the set of graded prime ideals of $L(E)$ and the set of $\Z$-prime order ideals 
of  $\VV^{\gr}(L(E))$. 
\end{thm}

\begin{proof}
Since there is a lattice isomorphism between the lattice of graded ideals of $L(E)$ and the lattice of $\Z$-order ideals of $\VV^{\gr}(L(E))$ via (\ref{latticncnc}) and the lattice isomorphisms preserve, 
in particular, prime elements, the statement follows. 
\end{proof}

\section{A short exact sequence of Grothendieck groups for a Leavitt path algebra} 

In order to relate the graded Grothendieck group $K_0^{\gr}$ to the non-graded $K_0$ and $K_1$ in the setting of Leavitt path algebras and ultimately to the filtered $K$-theory of these algebras, we establish a van den Bergh like exact sequence. For a right regular Noetherian $\mathbb Z$-graded ring $A$, van den Bergh established the long exact sequence 
\begin{equation}\label{gelatobal}
\cdots \longrightarrow K_{n+1}(A)  \longrightarrow K^{\gr}_n(A) \stackrel{\ol i}{\longrightarrow} K^{\gr}_n(A) \stackrel{U}{\longrightarrow} K_n(A) \longrightarrow \cdots.
\end{equation}
Here $\ol i= \overline{\mathcal T_1}  - \overline{\mathcal T_0}=\overline{\mathcal T_1}  - 1$, where $\mathcal T_1$ is the shift functor \eqref{shift} and $U$ is the forgetful functor \cite[\S 6.4]{haz}. Leavitt path algebras are in general  neither Noetherian nor regular. However for an arbitrary graph $E$, in Section~\ref{jdugt7htu6} we  establish an exact sequence 
\begin{align*}
K_1(L(E))  \longrightarrow K_0^{\gr}(\L(E))  \longrightarrow K_0^{\gr}(\L(E)) \longrightarrow K_0(\L(E)) \longrightarrow 0.
\end{align*}
This exact sequence involving only $K_0$-groups and for finite graphs with no sinks was first established in  \cite[Theorem 3]{haz2013}.

\subsection{A short exact sequence of Grothendieck groups} \label{jdugt7htu6}

Let $E$ be a row-finite graph and $\L(E)$ the Leavitt path algebra of $E$ over a field $k$. Recall from \cite[Corollary 5.8]{ahls} that $\VV^{\gr}(\L(E))\xra K_0^{\gr}(\L(E))$ is injective and there is a 
$\mathbb Z$-monoid isomorphism  $\VV^{\gr}(\L(E))\cong M_E^{\gr}$ (see~\eqref{hfghfyft7}).

We define a group homomorphism $\phi: K_0^{\gr}(\L(E))\xra K_0^{\gr}(\L(E))$ which takes the role of $\ol i$ in the sequence (\ref{gelatobal}). Let $F_E^{\gr}$ be the free commutative monoid generated by $\{v({i}) \mid  v\in
E^{0}, i\in \Z\}$. We first define a monoid homomorphism $\phi: F_E^{\gr}\xra K_0^{\gr}(\L(E))$ such that $\phi(v(i))=v(i+1)-v(i)$. Observe that $\phi$ preserves the relations of $M_E^{\gr}$. Thus there is an induced monoid homomorphism $M_E^{\gr}\xra K_0^{\gr}(\L(E))$, still denoted by $\phi$. Since $K_0^{\gr}(\L(E))$ is the group completion of $M_E^{\gr}$,  by universality of group completion, $\phi$ extends to a group homomorphism 
\begin{align}\label{mendhfuti6}
\phi: K_0^{\gr}(\L(E)) &\longrightarrow K_0^{\gr}(\L(E))\\
v(i) &\longmapsto v(i+1)-v(i), \notag 
\end{align}
where $i,j \in  \mathbb Z$. 

We are in a position to give a  short exact sequence relating the graded Grothendieck group to the nongraded one of a Leavitt path algebra over a row-finite graph.

\begin{prop}\label{3items} Let $E$ be a row-finite graph. We have the following short exact sequence 
\begin{align}\label{sesgroup}
K_0^{\gr}(\L(E)) \stackrel{\phi}{\longrightarrow} K_0^{\gr}(\L(E))   \stackrel{U}{\longrightarrow} K_0(\L(E)) \longrightarrow 0,
\end{align} 
where $\phi$ is the homomorphism   {\upshape (\ref{mendhfuti6})} and $U$ is the homomorphism induced by the forgetful functor. 
\end{prop}

\begin{proof}
Since each element $a\in K^{\gr}_0(\L(E))$ may be written as $a=x-y$ with $x=\sum_{i}u_i(s_i), y=\sum_{j}v_j(t_j)$ in $M_E^{\gr}$, 
we have $(U\circ \phi)(a)=U\big(\sum_{i}u_i(s_i+1)-\sum_{i}u_i(s_i)+\sum_{j}v_j(t_j+1)-\sum_{j}v_j(t_j)\big)=0$ by the definition of $U$ and $\phi$. 
Since the canonical homomorphism $M_E^{\gr}\xra M_E$ is surjective, the map $U: K^{\gr}_0(\L(E))\xra K_0(\L(E))$ is surjective.

Now we show that $\Ker U\sub {\rm Im}\, \phi$. Observe that \begin{equation}\label{ij}
u(i)-u(j)=\sum_{k=j}^{i-1}\phi(u(k))
\end{equation} with $u\in E^0, i, j\in\Z$ and $i>j$. Suppose that $b\in \Ker U$. We may write $b=x-y$ such that $$x=\sum_{s}u_s(k_s)$$ and $$y=\sum_{n}v_n(a_n)$$ in $M_E^{\gr}$. Since $$U(b)=\sum_{s}u_s-\sum_{n}v_n=0$$ in $K_0(\L(E))$, we have $$\sum_{s}u_s=\sum_{n}v_n$$ in $K_0(\L(E))$ and thus $$\sum_{s}u_s+z=\sum_{n}v_n+z$$ in $M_E$ for some $z\in M_E$ (see \cite[Corollary 1.5]{ls}). By Lemma \ref{monoidproperty}, there exists $\g\in F$ such that 
$$\sum_{s}u_s+z\longrightarrow \g$$ and 
$$\sum_{n}v_n+z \longrightarrow \g.$$
We will show that $\sum_{s}u_s+z - \g\in {\rm Im}\, \phi$ and $\sum_{n}v_n+z- \g\in {\rm Im}\, \phi$. Then it follows that $$\sum_{s}u_s+z-\g-(\sum_{n}v_n+z-\g)=\sum_{s}u_s-\sum_{n}v_n\in {\rm Im}\, \phi$$ and thus 
\begin{equation}
\begin{split}
b&=x-y\\
&=\sum_{s}u_s(k_s)-\sum_{n}v_n(a_n)\\
&= \big(\sum_{s}u_s(k_s)-\sum_{s}u_s\big)-\big(\sum_{n}v_n(a_n)-\sum_{n}v_n\big)+\big(\sum_{s}u_s-\sum_{n}v_n\big)\in {\rm Im}\, \phi,
\end{split}
\end{equation} since by \eqref{ij} $\sum_{s}u_s(k_s)-\sum_{s}u_s\in  {\rm Im}\, \phi$ and $\sum_{n}v_n(a_n)-\sum_{n}v_n\in {\rm Im}\, \phi$. We proceed by induction to show that $\sum_{s}u_s+z - \g\in {\rm Im}\, \phi$. Suppose that $\sum_{s}u_s+z\xra_1\g.$ We may write $\sum_{s}u_s+z=\sum_kw_k$ with $w_k\in E^0$.  If $w_1\xra_1 \sum_{e\in s^{-1}(w_1)}r(e)$ for $w_1\in E^0$ which is not a sink, then $$\sum_{s}u_s+z -\g=\sum_kw_k-\big(\sum_{k\neq 1}w_k+\sum_{e\in s^{-1}(w_1)}r(e)\big)=w_1-w_1(1)\in {\rm Im}\, \phi.$$ 
 By induction we get $\sum_{s}u_s+z -\g\in {\rm Im}\, \phi.$ Similarly, we have \[\sum_{n}v_n+z- \g\in {\rm Im}\, \phi.\] Therefore, we proved that the sequence \eqref{sesgroup} is exact.
\end{proof}

\begin{rmk}\label{hgftgftrgt22}

For an arbitrary graph $E$, $M_{E}^{\gr}$ is defined as a commutative monoid \cite [\S 5C]{ahls} such that the generators $\{v({i}) \mid  v\in
E^{0}, i\in \Z\}$ are supplemented  by generators $q_{Z}(i)$ as $i\in\Z$ and $Z$ runs
through all nonempty finite subsets of $s^{-1}(u)$ for infinite emitters $u\in E^{0}$.
The relations are
\begin{enumerate}
\item[(1)] $v({i})=\sum\limits_{e\in s^{-1}(v)}r(e)(i-1)$  for all
    regular vertices $v\in E^{0}$ and $i\in\Z$;
\item[(2)] $u({i})=\sum\limits_{e\in Z}r(e)(i-1)+q_{Z}({i})$ for
    all $i\in\Z$, infinite emitters $u\in E^{0}$ and nonempty finite subsets
    $Z\subseteq s^{-1}(u)$;
\item[(3)] $q_{Z_{1}}({i})=\sum\limits_{e\in Z_{2}\setminus
    Z_{1}}r(e)(i-1)+q_{Z_{2}}({i})$ for all $i\in\Z$, infinite
    emitters $u\in E^{0}$ and nonempty finite subsets $Z_{1}\subseteq Z_{2}\subseteq
    s^{-1}(u)$.
\end{enumerate}

Ara and Goodearl defined analogous monoids $M(E, C, S)$ and constructed natural
isomorphisms $M(E, C, S)\cong \mathcal{V}(CL_{K}(E, C, S))$ for arbitrary separated
graphs (\cite[Theorem~4.3]{ag}). The non-separated case reduces to that of ordinary
Leavitt path algebras, and extends the result of \cite{amp} to non-row-finite graphs.

For an arbitrary graph $E$, we do have the short exact sequence $$K_0^{\gr}(\L(E)) \stackrel{\phi}{\longrightarrow} K_0^{\gr}(\L(E))   \stackrel{U}{\longrightarrow} K_0(\L(E)) \longrightarrow 0$$ such that $\phi(q_{Z}(j))=q_Z(j+1)-q_{Z}(j)$ for generators $q_{Z}(j)$. The proof is similar to that for the row-finite case. This paper is devoted to row-finite graphs. In order to extend the results of this paper to arbitrary graphs, one needs to establish an exact sequence between $K_0^{\gr}$ of an arbitrary Leavitt path algebra $L(E)$ with the quotient $L(E)/I$, for any graded ideal $I$ as in Lemma~\ref{2dsnakelemma} (see Remark~\ref{hfgftrhyrrr}). 
\end{rmk}

%
%

Decompose the vertices of a row-finite graph $E$ as $E^0=R\sqcup S$, where $S$ is the set of sinks and $R=E^0\setminus S$. With respect to this decomposition we write the adjacency matrix of $E$ as
\[A_E=\tiny\begin{pmatrix}
B_E& C_E\\
0&0
\end{pmatrix}.\] 
Here, we list the vertices in $R$ first and then the vertices in $S$. 

Recall the following computation of $K$-theory for Leavitt path algebra from \cite[Corollary 7.7]{abc} and \cite[Proposition 3.4(ii)-(iii)]{grtw}. 

\begin{lem} \label{k0k1}Let $E$ be a row-finite graph and $L(E)$ the Leavitt path algebra over a field $k$. We have 
\begin{itemize}
\item [(i)] $K_0(L(E))\cong {\rm Coker} \bigg({\tiny{\begin{pmatrix}B_E^t-I\\ C_E^t\end{pmatrix}}}: \Z^{R}\longrightarrow \Z^{E^0}\bigg)$;

\item[(ii)] $K_1(L(E))$ is isomorphic to a direct sum: 

$${\rm Coker} \bigg({\tiny{\begin{pmatrix}B_E^t-I\\ C_E^t \end{pmatrix}}}: (k^{\times})^{R}\longrightarrow (k^{\times})^{E^0}\bigg)\bigoplus \Ker \bigg({\tiny{\begin{pmatrix}B_E^t-I\\ C_E^t \end{pmatrix}}}: \Z^{R}\longrightarrow \Z^{E^0}\bigg).$$ 
 \end{itemize}
\end{lem}



Throughout the paper, for an abelian group $G$ and a set $S$ (possibly infinite), we denote the direct sum $\oplus_{S}G$ by $G^S$.

In order to extend the exact sequence \eqref{sesgroup} to the $K_1$-group, we need the following proposition whose proof follows the ideas in the proof of \cite[Theorem 3.2]{rs}. 

\begin{prop}\label{4itemses} Let $E$ be a row-finite graph. Define $K: \Z^{R}\xra \Z^{R}\oplus \Z^S$ by $K(x)=\big((B_E^t-I)(x), C_E^tx\big)$ and $\psi: \Z^{R}\oplus \Z^S\xra K_0^{\gr}(\L(E))$ by $\psi(y)=y(0)$ where an element $y=\sum_{i}n_i$ in $\Z^R\oplus\Z^S$ is identified as $y=\sum_{i}n_iv_i$ with $v_i\in E^0$ and $y(0)=\sum_{i}n_iv_i(0)$. Then $\psi$ restricts to an isomorphism $\psi|$ from $\Ker K$ onto $\Ker \phi$ and an isomorphism $\overline{\psi}$ from ${\rm Coker} K$ onto ${\rm Coker} \phi$ such that the following diagram commutes:
\begin{align}
\label{aim}
\xymatrix{
\Ker K \ar[r]^{} \ar[d]^{\psi|}& {\Z}^R \ar[r]^{K} \ar[d]^{\psi} & {\Z}^{R}\oplus {\Z}^S \ar[r]^{\psi}\ar[d]^{\psi}& {\rm Coker} K\ar[d]^{\overline{\psi}}\\
\Ker \phi \ar[r]^{} & K_0^{\gr}(\L(E))\ar[r]^{\phi}& K_0^{\gr}(\L(E)) \ar[r]^{}&{\rm Coker} \phi,}
\end{align} where $\phi$ is given in \eqref{mendhfuti6}.
\end{prop}

\begin{proof} By typographical reasons, we denote $(v,n):=v_n$ and $(e,n):= e_n$ in the course of the proof.

For integers $m\leq n$, we denote by $E\times_1[m, n]$ the subgraph of $E\times_1 \Z$ with vertices $\{(v, k)\;|\; m\leq k\leq n, v\in E^0\}$ and edges $\{(e, k)\;|\; m< k\leq n, e\in E^1\}$. We allow $m=-\infty$ with the obvious modification of the definition. 
Recall that the Leavitt path algebra of an acyclic finite graph is a direct sum of copies of matrix algebras, indexed by the sinks in the graph (see \cite[Theorem 2]{roozbehhazrat2013}). Similarly, since $E\times_1[m, n]$ is an acyclic row-finite graph with paths of length at most $n-m$, $L(E\times_1[m, n])$ is a direct sum of matrix algebras indexed by the set of sinks in $E\times_1[m, n]$. 
The set of sinks in $E\times_1[m, n]$ is $$\{(v, m)\;|\; v\in E^0\}\cup \{(v, k)\;|\; v\in S, m<k\leq n\}.$$  We have that $K_0(L(E\times_1 [m,n]))$ is the free abelian group with generators $$\{[p_{(v, m)}]\;|\; v\in R\}\cup \{[p_{(v, k)}]\;|\; v\in S, m\leq k\leq n\}. $$ By continuity of $K$-theory we can let $n\to \infty$ and deduce that 
\begin{align*}K_0\big(L(E\times_1 [m, \infty))\big)
&=\big(\bigoplus_{v\in R}\Z[p_{(v, m)}]\big)\bigoplus \big(\bigoplus_{k=0}^{\infty}\bigoplus_{v\in S}\Z[p_{(v, m+k)}]\big)\\
&\cong \Z^R\oplus \Z^{S_m}\oplus\Z^{S_{m+1}}\oplus \cdots, 
\end{align*} where each $S_j$ is a copy of $S$ labelled to indicate its place in the direct sum. Denote by $\iota_m$ the inclusion map from $L(E\times_1 [m+1, \infty))$ to $L(E\times_1 [m, \infty))$ which induces a map from $K_0\big(L(E\times_1 [m+1, \infty))\big)
$ to $K_0\big(L(E\times_1 [m, \infty))\big)$. If $v\in R$, then in  $K_0\big(L(E\times_1 [m, \infty))\big)$ we have 
\begin{align*}[p_{(v, m+1)}]
&=\sum_{e\in s_E^{-1}(v)}[(e, m+1)(e, m+1)^*]\\
&=\sum_{e\in s_E^{-1}(v)}[(e, m+1)^*(e, m+1)]\\
&=\sum_{e\in s_E^{-1}(v)}[p_{(r(e), m)}]\\
&=\sum_{w\in E^0}A(v, w)[p_{(w, m)}].
\end{align*} If $v\in S$ and $k\geq m+1$, $[p_{(v, k)}]$ is a generator in $K_0\big(L(E\times_1 [m, \infty))\big)$. 
Thus the induced map from $\Z^{R}\oplus \Z^{S_{m+1}}\oplus\Z^{S_{m+2}}\oplus \cdots$ to $\Z^{R}\oplus \Z^{S_{m}}\oplus\Z^{S_{m+1}}\oplus \cdots$ is given by the matrix
\[ D = \tiny\begin{pmatrix}B_E^t&0& 0&0& \cdot\\
C_E^t&0&0&0 &\cdot\\
0&1&0&0&\cdot\\
0&0&1&0&\cdot\\
\cdot&\cdot&\cdot&\cdot&\ddots
 \end{pmatrix}\] 
and $K_0^{\gr}(\L(E))\cong K_0(L(E\times_1\Z))$ is the limit of the system 
$$\xymatrix{{\Z}^{R}\oplus {\Z}^S \oplus {\Z}^S\oplus \cdots \ar[r]^{D} & {\Z}^{R}\oplus {\Z}^S\oplus {\Z}^S \oplus \cdots   \ar[r]^<<<<{D}& \cdots.\\}$$ 

Recall that there is a canonical action of $\Z$ on $E\times_1 \Z$ which induces an action $\gamma: \Z\xra {\rm Aut}(L(E\times_1\Z))$ characterised by $\g_{1}((v, n))=(v, n-1)$ and $\g_1((e, n))=(e, n-1)$. Set $\beta=\g_1$. Then 
\[\beta^{-1}: L(E\times_1[m, \infty))\longrightarrow L(E\times_1[m, \infty)),\] and that 
\[\beta_*^{-1}: K_0(L(E\times_1[m, \infty))) \longrightarrow K_0(L(E\times_1[m, \infty))),\] viewed as a map on $\Z^R\oplus \Z^{S}\oplus\Z^{S}\oplus \cdots$ is multiplication by $D$. We have the following commutative diagram 
$$
\xymatrix{
L(E\times_1 [m+1, \infty)) \ar[r]^{\iota_m} \ar[d]^{\beta^{-1}-1}& L(E\times_1 [m, \infty)) \ar[r]^{\iota^m} \ar[d]^{\beta^{-1}-1} & L(E\times_1\Z)\ar[d]^{\beta^{-1}-1}\\
L(E\times_1 [m+1, \infty)) \ar[r]^{\iota_m} & L(E\times_1 [m, \infty)) \ar[r]^{\iota^m} & L(E\times_1 \Z)}$$ and thus the following diagram 
\begin{equation}
\label{com}
\xymatrix{
\Z^{R}\oplus \Z^{S_{m+1}}\oplus\Z^{S_{m+2}}\oplus \cdots \ar[r]^{D} \ar[d]^{D-1}& \Z^{R}\oplus \Z^{S_{m}}\oplus\Z^{S_{m+1}}\oplus \cdots \ar[r]^{\iota^m_*} \ar[d]^{D-1} & K_0(L(E\times_1\Z))\ar[d]^{\beta_*^{-1}-1}\\
\Z^{R}\oplus \Z^{S_{m+1}}\oplus\Z^{S_{m+2}}\oplus \cdots \ar[r]^{D} & \Z^{R}\oplus \Z^{S_{m}}\oplus\Z^{S_{m+1}}\oplus \cdots \ar[r]^{\iota^m_*} & K_0(L(E\times_1 \Z))}
\end{equation}
commutes. 

We can realise $K_0(L(E\times_1 \Z))$ as the group of equivalence classes $[(x_i)]$ of sequences in $\prod_{i=0}^{-\infty}(\Z^{R}\oplus \Z^{S}\oplus\Z^{S}\oplus \cdots)$ which 
eventually satisfy $Dx_i=x_{i-1}$ and two sequences are equivalent if they eventually coincide. 
The natural map $\iota_*^m$ sends $x\in \Z^{R}\oplus \Z^{S_{m}}\oplus\Z^{S_{m+1}}\oplus \cdots$ to the class of the sequence $(x_i)$ with 
$$
x_i
=\begin{cases} D^{m-i}(x), & \text{if~} i\leq m;\\
 0, &\text{if~} i>m.
 \end{cases}
$$

Observe that the homomorphism $\iota_*^0$ restricts to an isomorphism of $\Ker (D-1)$ onto $\Ker(\beta_*^{-1}-1)$ and induces an isomorphism $\overline{\iota}_*^0$ of $\Coker (D-1)$ 
onto $\Coker(\beta_*^{-1}-1)$. To show that $\iota_*^0$ is injective on $\Ker(D-1)$, we check directly that $\iota_*^0(x)$ is the constant sequence $(x)$ as $x=D(x)$. If $z=\iota_*^m(y)\in \Ker(\beta_*^{-1}-1)$, then we have 
\[(\beta^{-1}_*-1)(\iota_*^m(y))=0=\iota_*^m((D-1)(y))=\iota_*^m(D(y)-y).\] Thus there exists a large $N$ such that $D^N(D(y)-y)=0$. 
Take $y'=D^N(y)$. It follows that $z=\iota_*^0(y')$ with $y'=D(y')$. Therefore $\iota_*^0$ restricts to an isomorphism of $\Ker(D-1)$ onto $\Ker(\beta_*^{-1}-1)$. 
To show that $\overline{\iota}_*^0$ from $\Coker(D-1)$ to $\Coker(\beta_*^{-1}-1)$ is injective, suppose $\iota_*^0(x)=[(z_i)]\in \rm Im (\beta_*^{-1}-1)$. Then $[(z_i)]=[(Dy_i-y_i)]$ 
for some $(y_i)\in K_0\big(L(E\times_1\Z)\big)$.  For a large enough $k$, we have $D^k(x)=z_{-k}=Dy_{-k}-y_{-k}$. On the other hand, we have 
\[x=x-D^kx+D^kx=(1-D)(1+D+D^2+\cdots+D^{k-1})x-(D-1)y_{-k}\in \Imm(D-1),\] which implies that $\overline{\iota}_*^0$ is injective. To show that $\overline{\iota}_*^0$ is surjective, let $\iota_*^m(y)\in K_0\big(L(E\times_1\Z)\big)$
for some $m\le 0$. By \eqref{com}, we have 
\[\iota_*^m(Dy)-\iota_*^m(y)=\iota_*^m((D-1)y)=(\beta_*^{-1}-1)(\iota_*^m(y)).\] 
Thus $\iota_*^m(y)$ and $\iota_*^m(Dy)$ define the same equivalence class in $\Coker(\beta_*^{-1}-1)$. So $\iota_*^0(y)=\iota_*^m(D^{-m}y)$ define the same equivalence class as $\iota_*^m(y)$. Hence $\iota_*^0(y)=\iota_*^m(y)$ in $\Coker(\beta_*^{-1}-1)$ and $\overline{\iota}_*^0$ is surjective.

Let $i$ and $j$ be the inclusion of $\Z^R$ and $\Z^R\oplus\Z^S$ as the first coordinates of $\Z^R\oplus\Z^S\oplus\Z^S\oplus\cdots$. Then the following diagram commutes
\begin{align}
\label{a}
\xymatrix{
 {\Z}^R \ar[r]^{K} \ar[d]^{i} & {\Z}^{R}\oplus {\Z}^S \ar[d]^{j}\\
\Z^R\oplus\Z^S\oplus\Z^S\oplus\cdots\ar[r]^{D-1} & \Z^R\oplus\Z^S\oplus\Z^S\oplus\cdots,}
\end{align} where $i$ restricts to an isomorphism of $\Ker K$ onto $\Ker(D-1)$, and $j$ induces an isomorphism $\overline{j}$ of $\Coker K$ onto $\Coker (D-1)$ (compare \cite[Lemma 3.4]{rs}).

Denote by $\vartheta: K_0(L(E\times_1 \Z))\xra K_0^{\gr}(L(E))$ the isomorphism induced by the isomorphism of monoids $M_{E\times_1 \Z}\xra M^{\gr}_{E}, v_i\mapsto v(i)$ for $v\in E^0$ and $i\in\Z$ in \cite[Proposition 5.7]{ahls}. Combining \eqref{com} and \eqref{a}, we have the following commutative diagram
\begin{equation}
\label{final}
\xymatrix{
\Ker K\ar[r]^{}\ar[d]^{}&\Z^{R}\ar[r]^{K} \ar[d]^{i}& \Z^{R}\oplus \Z^{S} \ar[r]^{} \ar[d]^{j} & \Coker K\ar[d]^{}\\
\Ker(D-1)\ar[r]^{}\ar[d]^{}&\Z^{R}\oplus \Z^{S}\oplus\Z^{S}\oplus \cdots \ar[r]^{D-1}\ar[d]^{\iota_*^0} & \Z^{R}\oplus \Z^{S}\oplus\Z^{S}\oplus \cdots \ar[r]^{} \ar[d]^{\iota_*^0}& \Coker(D-1)\ar[d]^{}\\
\Ker (\beta_*^{-1}-1) \ar[r]^{}\ar[d]^{}& K_0(L(E\times_1\Z))\ar[r]^{\beta_*^{-1}-1} \ar[d]^{\vartheta}&  K_0(L(E\times_1\Z)) \ar[r]^{} \ar[d]^{\vartheta} & \Coker(\beta_*^{-1}-1) \ar[d]^{}\\
\Ker \phi\ar[r]^{}& K_0^{\gr}(L(E)) \ar[r]^{\phi} &K_0^{\gr}(L(E)) \ar[r]^{} & \Coker\phi,}
\end{equation} implying the commutative diagram in \eqref{aim} follows. The map $\psi:\Z^{R}\oplus\Z^{S}\xra K_0^{\gr}(L(E))$ is the composition of the maps in the third column and thus is given by $\psi(y)=y(0)$.
\end{proof}

\begin{prop} \label{someother}
Let $E$ be a row-finite graph and $k$ a field. We have the following short exact sequence 
\begin{equation}\label{sesl}
\xymatrix{
K_1(\L(E)) \ar[r]^{T}& K_0^{\gr}(\L(E)) \ar[r]^{\phi} & K_0^{\gr}(\L(E)) \ar[r]^{U}& K_0(\L(E))\ar[r]&0,}
\end{equation} 
where $T$ is the homomorphism given by 
$$K_1(L(E))\cong {\rm Coker} \bigg({\tiny{\begin{pmatrix}B_E^t-I\\ C_E^t \end{pmatrix}}}: (k^{\times})^{R}\longrightarrow (k^{\times})^{E^0}\bigg)\oplus \Ker\phi\stackrel{{\tiny{\begin{pmatrix}0 & 1\end{pmatrix}}}}{\longrightarrow} K_0^{\gr}(\L(E)).$$
\end{prop}

\begin{proof} The result follows directly from Proposition \ref{3items}, Lemma \ref{k0k1} (ii) and Proposition~\ref{4itemses}. \end{proof}


\section{$K$-theory of a Leavitt path algebra}

In this section, we first recall the long exact sequence for the computation of $K$-theory of Leavitt path algebras in \S\ref{gfhftfhtr2}. The group homomorphism $K_1(L(E))\xra \Ker \begin{pmatrix}A_E^t-I\end{pmatrix}$ appears in the long exact sequence where $A_E$ is the adjacency matrix of the graph $E$. We describe a map $\chi_1':\Ker \begin{pmatrix}A_E^t-I\end{pmatrix}\xrightarrow{} K_1(L(E))$ which is a section for the map $K_1(L(E))\xra \Ker \begin{pmatrix}A_E^t-I\end{pmatrix}$. However, the map $\chi_1'$ depends on the choice of certain bijections (see Example \ref{exm:badbehaviour}) and is not in general a group homomorphism. By using a certain quotient $K_1(L(E))/G_E$ of $K_1(L(E))$ (see \eqref{definition-quotient}), we get a group homomorphism $\chi_1:\Ker \begin{pmatrix}A_E^t-I\end{pmatrix}\xrightarrow{} K_1(L(E))/G_E$ which does not depend on the choices of bijections. This will be used in the proof of the main theorem of this paper, Theorem \ref{maintheorem}. 


\subsection{$K$-theory of Leavitt path algebras}\label{gfhftfhtr2}

The computation of $K$-theory for Leavitt path algebras was considered in \cite{abc,grtw}. In \cite{abc} the authors consider the Leavitt path algebra of a row-finite graph. The authors of \cite{grtw} extend the $K$-theory computation of \cite[Theorem 7.6]{abc} to Leavitt path algebras of graphs that may contain infinite emitters.

Recall that we decompose the vertices of a row-finite graph $E$ as $E^0=R\sqcup S$, where $S$ is the set of sinks and $R=E^0\setminus S$. With respect to this decomposition we write the adjacency matrix of $E$ as
\[A_E=\tiny\begin{pmatrix}
B_E& C_E\\
0&0
\end{pmatrix}.\] 
Here, we list the vertices in $R$ first and then the vertices in $S$. 

The following result is taken from \cite[Theorem 3.1]{grtw} (see also \cite[Theorem 7.6]{abc}).

\begin{lem} \label{longexact}
	Let $E$ be a row-finite graph and $k$ a field.  There is a long exact sequence
	$$\xymatrix{\cdots\ar[r]^{} &K_n(k)^{R} \ar[r]^{\tiny{\begin{pmatrix}B_E^t-I\\ C_E^t\end{pmatrix}}} & K_n(k)^{E^0} \ar[r]^{}& K_n(L_k(E))\ar[r]^{}&K_{n-1}(k)^{R}\ar[r]^>>>{\tiny{\begin{pmatrix}B_E^t-I\\ C_E^t\end{pmatrix}}}& \cdots\\}$$ for any $n\in\Z$.
\end{lem}

By Lemma \ref{longexact}, setting $n=1$, we have the following exact sequence.
\begin{equation}\label{01}
\xymatrix{
	{k^{\times}}^{R} \ar[r]^>>>>{\tiny{\begin{pmatrix}B_E^t-I\\ C_E^t\end{pmatrix}}} & {k^{\times}}^{E^0} \ar[r]^<<<{\lambda'}& K_1(L(E))\ar[r]^<<<{\xi'}& {\Z}^{R}\ar[r]^>>>>{\tiny{\begin{pmatrix}B_E^t-I\\ C_E^t\end{pmatrix}}}& \Z^{E^0}
}
\end{equation}
The definition of $\xi':K_1(L(E))\longrightarrow{}  {\Z}^{R}$ is given by the composition 
map $K_1(L(E))\xrightarrow[]{\partial} K_0(\mathcal{K}(E))\cong \Z^{R}$, where $\partial$ is the connecting map of $K$-theory with respect to the short exact sequence 
given in \eqref{sescohn}. Observe that by \eqref{01} $\xi'$ satisfies 
\[\rm Im \, \xi' = \Ker \bigg({\tiny{\begin{pmatrix}B_E^t-I\\ C_E^t\end{pmatrix}}}: \Z^{R}\longrightarrow \Z^{E^0}\bigg).\]
So we also have the map
\begin{equation}
\label{defxi}
\xi': K_1(L(E))\longrightarrow{} \Ker \bigg({\tiny{\begin{pmatrix}B_E^t-I\\ C_E^t\end{pmatrix}}}: \Z^{R}\longrightarrow\Z^{E^0}\bigg). 
\end{equation}

The map $\lambda':{k^{\times}}^{E^0} \xrightarrow[]{} K_1(L(E))$ in \eqref{01} is induced by the homomorphism of algebras 
\begin{align*}
\bigoplus_{v\in E^0}kv & \longrightarrow L(E),\\
v & \longmapsto v,
\end{align*}
for each $v\in E^0$. By $\eqref{01}$ we also have the map 
\[\lambda': \Coker\bigg({\tiny{\begin{pmatrix}B_E^t-I\\ C_E^t\end{pmatrix}}}: {k^{\times}}^{R} \longrightarrow {k^{\times}}^{E^0}\bigg)\longrightarrow K_1(L(E)),\]
such that  for $(\mu_v)_{v\in E^0}\in {k^{\times}}^{E^0},$
\[\lambda'((\mu_v)_{v\in E^0})= [\sum _{v\in E^0} \mu_v v]_1 \in K_1(L(E))
\] (see \cite[Proposition 3.4.3]{bort}). 


Recall that a group isomorphism from $\Ker \tiny{\begin{pmatrix}B_E^t-I\\ C_E^t\end{pmatrix}}$ to $K_1\big(C^*(E)\big)$ is described in \cite[Proposition 3.8]{cet}. Now we consider the Leavitt path algebra and define the map 
\[\chi_1': \Ker{\tiny{\begin{pmatrix}B_E^t-I\\ C_E^t\end{pmatrix}}}\longrightarrow K_1\big(L(E)\big),\] and show that $\xi'\circ \chi_1'=\rm id$.

Take $x$ in $\Ker\tiny{\begin{pmatrix}B_E^t-I\\ C_E^t\end{pmatrix}}$. Note that by definition $x$ has only finitely many nonzero entries $x_{v_1}, \cdots, x_{v_k}$. We define
\begin{equation}
\label{twosets}
\begin{split}&{L_x^+}=\Big \{(e, i)\;|\; e\in E^1, 1\leq i\leq -x_{s(e)}\Big\}\bigcup \Big\{(v, i)\;|\; v\in E^0, 1\leq i\leq x_v \Big \}, \text{~and~}\\ &{L_x^-}=\Big \{(e, i)\;|\; e\in E^1, 1\leq i\leq x_{s(e)}\Big\}\bigcup \Big\{(v, i)\;|\; v\in E^0, 1\leq i\leq -x_v\Big \}.
\end{split}\end{equation} By \cite[Lemma 3.1]{cet} and \cite[Lemma 3.2]{cet}, $L_x^+$ and $L_x^{-}$ have the same number of elements and there are bijections
$$[\cdot]: L_x^+\longrightarrow \{1, \cdots, h\}\qquad \text{ and} \qquad  \langle\cdot\rangle: L_x^-\longrightarrow \{1, \cdots, h\}$$ with the property that $[x, i]=\langle y, j\rangle$ implies $r(x)=r(y)$. Here $h$ is the common number of elements in $L^+_x$ and $L^-_x$.

With notation as above, define two matrices $V$ and $P$ by 
$$V=\sum_{\substack{1\leq i\leq x_w\\ s(e)=w}}eE_{[w, i],\langle e, i\rangle}+\sum_{\substack{1\leq i\leq -x_w\\ s(e)=w}}e^*E_{[e, i], \langle w, i\rangle}$$ and $$P=\sum_{1\leq i\leq x_w}wE_{[w, i],[w, i]}+\sum_{\substack{1\leq i\leq -x_w\\ s(e)=w, r(e)=v}}vE_{[e, i], [e, i]},$$ where $E_{\bullet, \bullet}$ denote the standard matrix units in $h\times h$-matrix algebra $\mathbb{M}_h(\mathcal{M}(L(E)))$ (see \cite[\S 2]{jackspielberg} and \cite[Definition 3.3]{cet}). 
Here $\mathcal{M}(L(E))$ is the multiplier algebra of $L(E)$. Recall that the multiplier of an algebra without unit has a unit.

We write $V_x$ and $P_x$ for the corresponding elements $V$ and $P$, using the added subscript to emphasise the dependence of each of $V$ and $P$ on $x \in \Ker \tiny{\begin{pmatrix}B_E^t-I \\ C_E^t\end{pmatrix}}$. In addition, we define \begin{equation} \label{u}
U_x=V_x+(1-P_x).\end{equation} The matrix $U_x$ is invertible and $U_x^{-1}=U_x^*=V_x^*+(1-P_x)$ (compare \cite[Lemma 3.4]{cet} and \cite[Fact 3.6]{cet}). 

The map $\chi_1':\Ker{\tiny{\begin{pmatrix}B_E^t-I\\ C_E^t\end{pmatrix}}}\longrightarrow K_1\big(L(E)\big)$ is defined as $$\chi_1'(x)=[U_x]_1$$ for $x\in \Ker{\tiny{\begin{pmatrix}B_E^t-I\\ C_E^t\end{pmatrix}}}.$



  The map $\chi_1'$ may depend on the choice of bijections $[\,\,]$ and $\langle \,\, \rangle$ chosen in its definition above.
  Indeed we have the following concrete example which shows that the element $\chi_1'(x)$ depends of the choices of these maps: 

\begin{exm}
	\label{exm:badbehaviour}
	The simplest example in which $\chi_1'$ depends on the choice of bijections is when $E$ has only one vertex $v$ and one edge $e$. In this case, $L_k(E)=k[t,t^{-1}]$ and $K_1(L_K(E))= K_0(k)\oplus K_1(k)= \Z\oplus k^{\times}$.
Consider $x=2\in \Ker (A_E^t-I) =\Z$. Then $L_x^+= \{(v,1), (v,2)\}$ and 
$L_x^-=\{(e,1),(e,2)\}$. Now taking $[v,i]= i$ and $\langle e,i \rangle = i$ for $i=1,2$, we obtain $\chi_1'(2)= \Big[\begin{pmatrix} e & 0 \\ 0 & e\end{pmatrix}\Big]_1$. Now taking $[v,i]= i$ for $i=1,2$ and $\langle e,1\rangle ' = 2$, $\langle e,2 \rangle '= 1$, we get $\tilde{\chi}_1'(2)=  \Big[ \begin{pmatrix} 0 & e \\ e & 0\end{pmatrix}\Big]_1\ne \chi_1'(2)$ (if the characteristic of $k$ is different from $2$).  
 \end{exm}

One can follow \cite{cortinasmontero2} and define a group homomorphism $\Ker\tiny{\begin{pmatrix}B_E^t-I\\ C_E^t\end{pmatrix}}\to K_1(L(E))$ by defining it on a basis of the free abelian group $\Ker\tiny{\begin{pmatrix}B_E^t-I\\ C_E^t\end{pmatrix}}$, and then (uniquely) extending to all of $\Ker\tiny{\begin{pmatrix}B_E^t-I\\ C_E^t\end{pmatrix}}$. However this procedure does not seem to be enough for our purposes, because we need functoriality with respect to graded subquotients of $L(E)$. We will instead show below  that we get a well-defined functorial group homomorphism $\chi_1\colon \Ker\tiny{\begin{pmatrix}B_E^t-I\\ C_E^t\end{pmatrix}}\to
\ol{K}_1(L(E))$, where $\ol{K}_1(L(E))$ is the quotient of $K_1(L(E))$ described in the following paragraph.

Let $E$ be a row-finite graph and $k$ a field. For $v\in E^0$, set $[-v]_1$ to be the image of the element $(1,\cdots, -1,\cdots, 1)^t\in {k^{\times}}^{E^0}$ under the map $\lambda' $ given by \eqref{01}. Here $-1$ corresponds to the component of $v$ and $1$ corresponds to all the other components.  We define $G=G_E$ to be the subgroup of $K_1(L_k(E))$ generated by $[-v]_1$, for $v\in E^0$. We write \begin{equation}
\label{definition-quotient}
\overline{K}_1(L(E)):=K_1(L(E))/G_E.
\end{equation}
We refer to $\ol{K}_1(L(E))$ as the {\it reduced $K_1$-group} of $L(E)$.

\begin{defi}
	\label{def:mapchi1}
	Let $E$ be a row finite graph. With the above notation we set $\chi_1\colon \Ker\tiny{\begin{pmatrix}B_E^t-I\\ C_E^t\end{pmatrix}}\to \ol{K}_1(L(E))$ to be the composition $\chi_1= p\circ \chi_1'$, where $p \colon K_1(L(E))\to \ol{K}_1(L(E))=K_1(L(E))/G_E$ is the natural projection.
\end{defi}

In order to show the properties of the map $\chi_1$, we need an easy lemma.

\begin{lem} 
	\label{lem:sign-of-a-permutation} Let $h$ be a positive integer, and suppose we have a decomposition $\{1,\dots , h\}= \bigsqcup_{v\in E^0} Z(v)$ for some subsets $Z(v)$ of $\{1,\dots , h\}$. Let $\sigma = \bigsqcup_{v\in E^0} \sigma_v$ be a permutation of $\{1,\dots ,h\}$, where each $\sigma_v$ is a permutation of $Z(v)$. Consider the element
	$$S=\sum_{v\in E^0} \sum_{t\in Z(v)}  v E_{t,\sigma_v(t)} \in M_h(L(E)).$$ 
	Then we have that $SS^*= S^*S= \sum_{v\in E^0} \sum_{t\in Z(v)} vE_{t,t}$ is a projection in $M_h(L(E))$ and we have
	$$[S+(1_h-SS^*)]_1 = \sum_{v\in E^0} [\mathrm{sign}(\sigma_v)v]_1\in K_1(L(E)).$$
\end{lem} 

\begin{proof}
	Observe that $[S+(1_h-SS^*)]_1$ is the image under $\lambda'$ of the element
	$([P_v]_1)_{v\in E^0}\in K_1(k)^{E^0}=K_1(k^{E^0})$, where, for each $v\in E^0$, $P_v$ is the permutation matrix $P_v=\sum_{t\in Z(v)} E_{t,\sigma (t)}$.
\end{proof}

We can now show that $\chi_1$ is a well-defined map.

\begin{lem}
	\label{lem:chi1welldefined} With the above notation, the map $\chi_1\colon \Ker\tiny{\begin{pmatrix}B_E^t-I\\ C_E^t\end{pmatrix}}\to \ol{K}_1(L(E))$ is a well-defined map, that is, it does not depend on the choices of the bijections $[\,\,]$ and $\langle \,\, \rangle$.
\end{lem}

\begin{proof} Take $x\in \Ker\tiny{\begin{pmatrix}B_E^t-I\\ C_E^t\end{pmatrix}}$. 
	Let $[\,\,]'\colon L_x^+ \to \{1,\dots , h_x\}$ and $\langle \,\, \rangle '\colon L_x^-\to \{ 1,\dots , h_x \}$ be different bijections satisfying the required properties. Let $\tau\in S_{h_x}$ be a permutation sending $[L_x^+(v)]'=\langle L_x^-(v) \rangle' $ to $[L_x^+(v)]=\langle L_x^-(v) \rangle $ for all $v\in E^0$. Set $[\,\,]''=\tau\circ [\,\,]'$ and $\langle \,\, \rangle ''= \tau \circ \langle \,\, \rangle '$. Then $[L_x^+(v)]''=[L_x^+(v)]$ and $\langle L_x^-(v)\rangle ''= \langle L_x^-(v)\rangle $ for each $v\in E^0$. Moreover we have $U_x''=T^{-1}U_x'T$, where $T=\sum_j E_{j,\tau(j)}$ is the permutation matrix associated to $\tau$. It follows that $[U_x']_1=[U_x'']_1$ and thus we can assume that $$\langle L_x^-(v) \rangle = [L_x^+(v)]=[L_x^+(v)]'= \langle L_x^-(v)\rangle '$$ 
	for all $v\in E^0$. With this assumption we see that there are two permutations $\sigma^{\pm} \in S_{h_x}$ of the form $\sigma^{\pm} = \bigsqcup_v \sigma^{\pm} _v$ such that $\sigma ^+_v([t])= [t]'$ for all $t\in L_x^+(v)$ and
	$\sigma ^-_v(\langle s\rangle )= \langle s\rangle' $ for all $s\in L_x^-(v)$. 
	Set 
	$$S^+= \sum_{v\in E^0} \sum_{t\in L_x^+(v)} vE_{\sigma^+_v([t]), [t]}, \qquad S^-= \sum_{v\in E^0} \sum_{t\in L_x^-(v)} vE_{\langle t\rangle, \sigma^-_v(\langle t\rangle )}.$$
	Then we have $S^+V_xS^-= V_x'$ and
	$$(S^+)^*S^+= S^-(S^-)^* = P_x=P_x'= (S^-)^*S^-= S^+(S^+)^*.$$ It follows that $S^+(1-P_x)=0= (1-P_x)S^-$. Therefore we get $(S^++(1-P_x))U_x(S^-+(1-P_x))=U_x'$. Using Lemma \ref{lem:sign-of-a-permutation}, we get
	$$[U_x']_1 -[U_x]_1 = \sum_{v\in E^0}[\text{sign}(\sigma^+_v)v]+ \sum_{v\in E^0}[\text{sign}(\sigma_v^-)v] \in G_E ,$$
	which shows that $ p([U_x])= p([U_x'])$ in $\ol{K}_1(L(E))$, as desired. 
	\end{proof}


We now show that $\chi_1$ is a group homomorphism:

\begin{lem}
	\label{lem:group-homo}
	The map $\chi_1\colon \Ker\tiny{\begin{pmatrix}B_E^t-I\\ C_E^t\end{pmatrix}}\to \ol{K}_1(L(E))$ is a group homomorphism.
\end{lem}

\begin{proof}
	Let $x,y\in \Ker\tiny{\begin{pmatrix}B_E^t-I\\ C_E^t\end{pmatrix}}$. We will show that  $\chi_1(x+y)= \chi_1(x) + \chi_1(y)$.

	We need some special bijections for each of $L_x^+,L_x^-,L_y^+,L_y^-$. We first introduce the following integers $x_v',y_v'$ for $v\in E^0$. 
	We write $x_v'=y_v'=0$ if $x_vy_v \ge 0$. Assume that $x_v>0$, $y_v<0$. Then we set
	$$x_v' = \text{min} \{ -y_v,x_v \} ,\quad y_v'= -x_v' = \text{max}\{ y_v, -x_v \}.$$
	Note that $0<x_v'\le x_v$ and $y_v \le y_v' <0$ in this situation. Also $x_v-x_v'= x_v+y_v$ if $x_v'<x_v $ and  $y_v-y_v'=x_v+y_v$ if $y_v<y_v'$.
	If $x_v=x_v'$ and $y_v=y_v'$ then $x_v+y_v=0$ (and conversely).
	
	Similarly, if $x_v<0$ and $y_v>0$, write 
	$$x_v' = \text{max} \{ -y_v,x_v \} ,\quad y_v'= -x_v' = \text{min}\{ y_v, -x_v \}.$$ 
	Similar remarks apply here.
	
	Now we can choose the bijections $[\,\,]_x, \langle \,\,\rangle_x, [\,\,]_y,\langle\,\, \rangle_y$ satisfying the following additional  properties:
	\begin{enumerate}
		\item For each $v\in E^0$ such that $x_v>0$ and $y_v<0$ and each $e\in E^1$ such that $s(e)=v$, we have $[v,i]_x= \langle v,i\rangle_y$ and $\langle e,i\rangle_x= [e,i]_y$ for $1\le i\le x_v'$.
		\item For each $v\in E^0$ such that $x_v<0$ and $y_v>0$ and each $e\in E^1$ such that $s(e)=v$, we have $\langle v,i\rangle _x=[v,i]_y$ and $[e,i]_x=\langle e,i\rangle_y$ for $1\le i\le y_v'$. 
	\end{enumerate}
	
	Note that Lemma \ref{lem:chi1welldefined} guarantees that $\chi_1 $ does not depend on the choice of bijections. We now explain how we can obtain bijections $[\,\,]$ and $\langle \,\, \rangle $ satisfying the above conditions (1) and (2) and also the required conditions that $[L_x^+(w)]_x=\langle L_x^-(w)\rangle_x$ and $[L_y^+(w)]_y=\langle L_y^-(w)\rangle _y$ for each $w\in E^0$. The key idea is to book space for the ``common" terms.  
	
	We first introduce the following sets:
	$$S_1^+= \{ (v,i)_x\mid x_v>0,y_v<0, 1\le i\le x_v' \},\qquad S_2^+=\{ (v,i)_y\mid x_v<0, y_v>0, 1 \le i\le y_v' \},$$
	and
	$$T_1^+=\{ (e,i)_x\mid s(e)=v, x_v<0, y_v>0\mid 1\le i\le -x_v' \},\qquad T_2^+=\{ (e,i)_y\mid s(e)= v, x_v>0, y_v<0, 1\le i\le -y_v' \}.$$ 
	Similarly, we write:
	$$S_1^-= \{ (v,i)_x \mid x_v <0 ,y_v>0,  1\le i\le - x_v'\},\qquad S_2^- = \{ (v,i)_y \mid x_v>0,y_v<0, 1\le i\le -y_v' \},$$
	and
	$$T_1^-= \{ (e,i)_x \mid s(e)=v, x_v >0 ,y_v<0,  1\le i\le x_v'\},\qquad T_2^- = \{ (e,i)_y \mid s(e)=v, x_v<0, y_v>0, 1\le i\le y_v' \}.$$
	Note that we have $|S_1^+(w)|=|S_2^-(w)|$ and  $|T_1^+(w)|=|T_2^-(w)|$,   where we indicate by $S(w)$ the set of elements $s\in S$ such that $r(s)=w$ for each of these sets $S$ and each $w\in E^0$. Simlarly we have   $|S_1^-(w)|=|S_2^+(w)|$ and  $|T_1^-(w)|=|T_2^+(w)|$ for each $w\in E^0$. 
	The sets $S_1^+$ and $T_1^+$ are identified with subsets of $L_x^+$ by the obvious identification $(a,i)_x\leftrightarrow (a,i)$. Similarly $S_1^-\sqcup T_1^-$ is identified with a subset of $L_x^-$,. Analogous considerations apply to the sets $S_2^{\pm}$ and $T_2^{\pm}$.  
	
	We first build the bijections $[\,\,]_x$ and $\langle \,\,\rangle_x$.
	We enumerate as $w_1, w_2, \dots ,w_N$ the vertices $w\in E^0$ such that
	$\text{max}\{ |S_1^+(w)|+|T_1^+(w)|, |S_1^-(w)|+|T_1^-(w)| \} >0$, and then we write 
	$$r_j= \text{max}\{ |S_1^+(w_j)|+|T_1^+(w_j)|, |S_1^-(w_j)|+|T_1^-(w_j)| \} \qquad (1\le j \le N).$$
	Observe that $r_j\le |L_x^+(w_j)|=|L_x^-(w_j)|$ for each $j$, so that $\sum_{j=1}^N r_j\le \sum_{w\in E^0}  |L_x^+(w)|= |L_x^+|=h_x$. 
	
	We first define $[\,\,]_x$ on $S_1^+(w_j) \sqcup T_1^+(w_j)$  and $\langle \,\, \rangle _x$ on $S_1^-(w_j)\sqcup T_1^-(w_j)$ respectively, by fixing injective maps  
	$$ [\,\,]_x \colon S_1^+(w_j) \sqcup T_1^+(w_j)\to \{\sum_{l=1}^{j-1} r_{l}+1,\dots ,\sum_{l=1}^j r_l \}, \quad \langle \,\, \rangle _x \colon  S_1^-(w_j)\sqcup T_1^-(w_j) \to   \{\sum_{l=1}^{j-1} r_{l}+1,\dots ,\sum_{l=1}^j r_l \}.$$
	Observe that at least half of the above maps are bijections. We now complete the definitions of $[\,\,]_x$ and $\langle \,\, \rangle _x$ to all of $L_x^+$ and $L_x^-$ respectively, in such a way that $[L_x^+(w)]_x=\langle L_x^-(w)\rangle_x$ for each $w\in E^0$. 
	
	We now define $[\,\,]_y$ and $\langle \,\, \rangle _y$. We first define injective maps 
		$$ [\,\,]_y \colon S_2^+(w_j) \sqcup T_2^+(w_j)\to \{\sum_{l=1}^{j-1} r_{l} + 1,\dots ,\sum_{l=1}^j r_l \}, \quad \langle \,\, \rangle _y \colon  S_2^-(w_j)\sqcup T_2^-(w_j) \to   \{\sum_{l=1}^{j-1} r_{l}+1,\dots ,\sum_{l=1}^j r_l \}$$
		by setting $[t]_y=\langle t\rangle _x$ for $t\in S_2^+(w_j)\sqcup T_2^+(w_j)$ and $\langle t\rangle _y= [t]_x$ for $t\in S_2^-(w_j)\sqcup T_2^-(w_j)$.  With this, we already obtain that conditions (1) and (2) above are satisfied.

		We now complete the definitions of $[\,\,]_y$ and $\langle \,\, \rangle _y$ to all of $L_y^+$ and $L_y^-$ respectively, in such a way that $[L_y^+(w)]_y=\langle L_y^-(w)\rangle_y$ for each $w\in E^0$. 
	
	At this point we have constructed maps $[\,\,]_x$,
	$\langle \,\,\rangle_x$, $[\,\,]_y$ and $\langle \,\, \rangle _y$ satisfying (1) and (2).

	We now consider the sets $L^+:=L_x^+\bigsqcup L_y^+$ and $L^-:= L_x^-\bigsqcup L_y^-$. In order to indicate the copy in $L^+$ of an element $(a,i)\in L_x^+$ we will write
	it as $(a,i)_x$, and similarly for the other factors. There are bijections $[\, \, ]$  and $\langle \,\, \rangle $ of $L^+$ and $L^-$ respectively onto $\{ 1,\dots , h_x+h_y \}$ given by 
	$[t]=[t]_x$, $\langle t'\rangle = \langle t'\rangle _x$ for $t\in L_x^+$ and $t'\in L_x^-$, and $[s]= [s]_y + h_x$, $\langle s' \rangle = \langle s'\rangle _y +h_x$ for $s\in L_y^+$ and $s'\in L_y^-$.  
	
	Embed $L_{x+y}^+$ into $L^+$ as follows. Assume first that $v\in E^0$, $x_v+y_v>0$ and $1\le i \le x_v+y_v$. There are three different cases. If $x_v\ge 0$ and $y_v\ge 0$, then
	send $(v,i)\in L^+_{x+y}$ to $(v,i)_x\in L_x^+$ if $1\le i \le x_v$ and to $(v, i-x_v)_y\in L_y^+$ if $x_v+1\le i\le x_v+y_v$. If $x_v>0$, $y_v<0$ and $-y_v<x_v$, then send
	$(v,i)$ to $(v, x_v'+i)_x\in L_x^+$. Note that $x_v'+1\le x_v'+i\le x_v$ in this case. Finally if $x_v<0$, $y_v>0$ and $-x_v<y_v$ then send $(v,i)$ to $(v, y_v'+i)_y\in L_y^+$.
	Note that $y_v'+1\le y_v'+i\le y_v$ in this case. 
	
	Now assume that $e\in E^1$, with $v=s(e)$, and that $x_v+y_v<0$ and suppose that $1\le i\le -x_v-y_v$. There are also three cases to consider. If $x_v<0$ and $y_v<0$, then send $(e,i)$ to $(e,i)_x\in L_x^+$
	if $1\le i\le -x_v$ and to $(e,i+x_v)_y\in L_y^+$ if $-x_v+1\le i\le -x_v-y_v$. If $x_v<0$, $y_v>0$ and $x_v<-y_v$ then send $(e,i)$ to $(e, -x_v'+i)_x\in L_x^+$. Observe that
	$-x_v'+1\le -x_v'+i\le -x_v$ in this case. Finally if $x_v>0$, $y_v<0$ and $y_v< -x_v$, then send $(e,i)$ to $(e, -y_v'+i)_y\in L_y^+$. Note that in this case we get
	$-y_v'+1\le -y_v'+i\le -y_v$. 
	
	In a similar way, we obtain an embedding of $L_{x+y}^-$ into $L^-$.
	
		Note that 
	$$L_x^+= (L_{x+y}^+\cap L_x^+) \sqcup S_1^+ \sqcup T_1^+ ,\qquad L_y^+ = (L_{x+y}^+\cap L_y^+)\sqcup S_2^+\sqcup T_2^+.$$
	Similarly
	$$L_x^-= (L_{x+y}^-\cap L_x^-) \sqcup S_1^-\sqcup T_1^- , \qquad L_y^-= (L_{x+y}^-\cap L_y^-)\sqcup S_2^-\sqcup T_2^- .$$
	We have
	\begin{align*}
	|L_{x+y}^+(v)|  & =  |L_x^+(v)|+|L_y^+(v)| - |S_1^+(v)| - |S_2^+(v)|-|T_1^+(v)|-|T_2^+(v)| \\
	&     = | L_x^-(v)| + | L_y^-(v)| - |S_2^-(v)| - |S_1^-(v)|-|T_2^-(v)|-|T_1^-(v)| = | L_{x+y}^-(v) |.
	\end{align*}	
	Now set
	$$V_{x,1}= \sum_{\substack{s(e)=w\\ (w,i)_x\in S_1^+ }} eE_{[w,i]_x,\langle e,i\rangle_x},\qquad V_{x,2} = \sum_{\substack{s(e)=w \\ (e,i)_x\in T_1^+ }} e^*E_{[e,i]_x,\langle w,i \rangle_x},$$
	and
	$$V_{y,1}= \sum_{\substack{s(e)=w\\ (w,i)_y\in S_2^+}} eE_{[w,i]_y+h_x,\langle e,i\rangle_y+h_x},\qquad V_{y,2} = \sum_{\substack{s(e)=w \\ (e,i)_y\in T_2^+ }} e^*E_{[e,i]_y +h_x,\langle w,i \rangle_y+h_x}.$$
	Then we have, using our choice of the functions
	\begin{equation}
	\label{eq:nicedecompositionVxVy}
V_x\oplus V_y= V_{x+y}' +[(V_{x,1}+V_{x,2}+V_{y,1}+V_{y,2})] = V_{x+y}' + [(V_{x,1}+V_{x,2})\oplus (V_{x,2}^*+V_{x,1}^*)],
		\end{equation}
	where $V'_{x+y}$ is the part of $V_x\oplus V_y$ coming from the terms in $L_{x+y}^+= (L_{x+y}^+\cap L_x^+)\sqcup (L_{x+y}^+\cap L_y^+)$. However it is not generally true that 
	$[L_{x+y}^+(v)]=\langle L_{x+y}^-(v) \rangle$, so we need to apply a permutation to get the desired element $V_{x+y}$ to compute $U_{x+y}$. 
	
	Let us define a permutation $\sigma = \sigma_1\sqcup \sigma _2\sqcup \sigma _3$ on $\{1,\dots , h_x+h_y \}$ as follows. 
	First $\sigma _1 \colon \langle S_1^-\sqcup T_1^- \rangle \to [S_2^+\sqcup T_2^+]$ is defined by
	$$\sigma _1 (\langle t\rangle_x) = [t]_y+h_x.$$
	This is well-defined by our choice of functions since indeed $\langle t \rangle _x=[t]_y$ for $t\in S_1^-\sqcup T_1^-$. Similarly we define
	$\sigma _2\colon \langle S_2^-\sqcup T_2^- \rangle \to [S_1^+\sqcup T_1^+]$ by
	$$\sigma _2 (\langle t\rangle_y + h_x) = [t]_x.$$
	Obviously the maps $\sigma _1$ and $\sigma _2$ are range and source preserving. Therefore, for each $v\in E^0$, the restriction of $\sigma_1\sqcup \sigma_2$ 
	defines a bijection from $\langle S_1^-(v)\sqcup T_1^-(v)\sqcup S_2^-(v)\sqcup T_2^-(v)\rangle$ onto $[S_1^+(v)\sqcup T_1^+(v)\sqcup S_2^+(v) \sqcup T_2^+(v)]$.
	Therefore we can define a bijection from $\langle L_{x+y}^-(v)\rangle $ onto $[ L_{x+y}^+(v)]$ so that $\sigma _1\sqcup \sigma _2\sqcup \sigma _3$ defines a bijection from
	$\langle L_x^-(v)\sqcup L_y^-(v)\rangle $ onto $[L_x^+(v)\sqcup L_y^+(v)]= \langle L_x^-(v)\sqcup L_y^-(v)\rangle $. In this way we obtain the permutation $\sigma$ of $\{ 1,\dots h_x+h_y\}$. 
	Note that if we replace the function $\langle \,\, \rangle $ by the new function $\langle \,\, \rangle ' := \sigma _3 \circ \langle \,\, \rangle $ in the definition of $V'_{x+y}$, we obtain a new 
	element $V_{x+y}$ with the property that $V_{x+y}V_{x+y}^* = V_{x+y}^*V_{x+y}$ and so we can compute the element $U_{x+y}$ using the functions $[\,\,]$ and $\langle \,\, \rangle '$.
	
	Now we want to express the element $U_{x+y}$ in terms of the element $U_x\oplus U_y$. 
	
	Write
	$$ S=\sum _{t\in L_x^-\sqcup L_y^-} r(t) E_{\langle t\rangle,\sigma (\langle t \rangle )} +(1-(P_x\oplus P_y))= \sum_{v\in E^0} \sum_{t\in L_x^-(v)\sqcup L_y^-(v)} vE_{\langle t\rangle , \sigma (\langle t \rangle )}+(1-P_x\oplus P_y))$$
	From Lemma \ref{lem:sign-of-a-permutation} we get that $SS^*=1=S^*S$ and that $[S]_1\in G_E$. Moreover we have $V_{x+y}= V'_{x+y}S$. Using this, putting $V:=V_{x,1}+V_{x,2}$, and taking into account \eqref{eq:nicedecompositionVxVy}, we get  
	\begin{align*}
	\Big( U_x\oplus U_y\Big) S &   = \Big( (V_x\oplus V_y) + (1-(P_x\oplus P_y))\Big) S\\
	& =\Big( V_{x+y}' + (V\oplus V^*)+(1-(P_x\oplus P_y))\Big)  \Big(\sum _{t\in L_x^-\sqcup L_y^-}r(t) E_{\langle t\rangle,\sigma (\langle t \rangle )} +(1-(P_x\oplus P_y))\Big)\\
	& =  V_{x+y} + \begin{pmatrix} 0 & V \\ V^* & 0 \end{pmatrix} +   (1-(P_x\oplus P_y))  \\
	& = \Big( V_{x+y}+ (1-V_{x+y}V_{x+y}^*) \Big) \begin{pmatrix} 1-VV^* & V \\ V^* & 1-V^*V \end{pmatrix} \\
	& = \Big( U_{x+y} + (1-1_{[L_{x+y}^+]}) \Big) \begin{pmatrix} 1-VV^* & V \\ V^* & 1-V^*V \end{pmatrix} .
	\end{align*}
	
Taking classes in $K_1(L_K(E))$, we obtain:
	$$[U_{x+y}]_1 + \Big[  \begin{pmatrix} 1-VV^* & V \\ V^* & 1-V^*V \end{pmatrix} \Big]_1= [U_x]_1 +[U_y]_1 + [S]_1 .$$
	Since $[S]_1\in G_E$, we only have to show that $\Big[  \begin{pmatrix} 1-VV^* & V \\ V^* & 1-V^*V \end{pmatrix} \Big]_1\in G_E$. However we have
	 $$ \Big[  \begin{pmatrix} 1-VV^* & V \\ V^* & 1-V^*V \end{pmatrix} \Big]_1 = [1-2VV^*]_1=[1-2V^*V]_1.$$
     Write $A(v):=S_1^+(v)\sqcup T_1^+(v)$ and
	$B(v)=S_2^+(v)\sqcup T_2^+(v)$. Observe that 
	$$[1-2VV^*]_1 = \sum_{v\in E^0}[(-1)^{|A(v)|}v]_1\quad \text{ and }\quad  [1-2V^*V]_1=\sum_{v\in E^0}[(-1)^{|B(v)|}v]_1,$$ so that  we get
	$$ \Big[  \begin{pmatrix} 1-VV^* & V \\ V^* & 1-V^*V \end{pmatrix} \Big]_1=
	 \sum_{v\in E^0}[(-1)^{|A(v)|}v]_1= \sum_{v\in E^0}[(-1)^{|B(v)|}v]_1\in G_E,$$
	 so that we finally conclude that $\chi_1(x+y)= \chi_1(x)+\chi_1(y)$.
	 \end{proof}

\begin{prop}\label{prop-section}
	Let $E$ be a row-finite graph. We have the following statements.
	\begin{itemize}
		\item[(i)] There exists a group isomorphism $\chi_0: \Coker{\tiny{\begin{pmatrix}B_E^t-I \\ C_E^t\end{pmatrix}}}\longrightarrow K_0\big(L(E)\big)$ given by $$\chi_0\Big({ {\bf e}_{v}+{\rm Im} \tiny{\begin{pmatrix}B_E^t-I \\ C_E^t\end{pmatrix}} } \Big) =v$$ for any $v\in E^0$. 
		\item[(ii)] The map $\chi_1':\Ker{\tiny{\begin{pmatrix}B_E^t-I\\ C_E^t\end{pmatrix}}}\longrightarrow K_1\big(L(E)\big)$ given by $$\chi_1'(x)=[U_x]_1$$ satisfies $\xi' \circ \chi_1' = \rm id,$ where $\xi': K_1\big(L(E)\big)\longrightarrow \Ker{\tiny{\begin{pmatrix}B_E^t-I\\ C_E^t\end{pmatrix}}}$ is given in \eqref{defxi}. 
	\end{itemize}
\end{prop}

\begin{proof}
	(i) There is a short exact sequence  \eqref{sescohn}
	\begin{align*}\CD
	\label{}
	0@>>> \mathcal{K}(E)@>{l}>>C(E)@>{p}>>L(E)@>>>0,
	\endCD
	\end{align*} where $\mathcal{K}(E)$ is the ideal (\ref{koutyut}) of the Cohn path algebra $C(E)$. We then have the induced maps of $K$-groups
	\begin{align}
	\label{k}
	\xymatrix{
		K_0\big(\mathcal{K}(E)\big) \ar[r]^{l_*} \ar[d]^{\kappa} & K_0\big(C(E)\big)\ar[r]^{p_*} \ar[d]^{\lambda}&K_0\big(L(E)\big)\\
		\Z^{R}\ar[r]_{\tiny{\begin{pmatrix} I- B_E^t \\ -C_E^t\end{pmatrix}}} & \Z^{E^0}.}
	\end{align} By \cite[Theorem 4.2]{cortinasmontero} we have $\lambda(v)=\bf{e}_v$ for any $v\in E^0$. By \cite[(4.13)]{cortinasmontero} there is an isomorphism of algebras 
	\begin{align}\label{isoalgs}
	\bigoplus_{v\in R}\mathbb{M}_{P_v}&\stackrel{\sim}\longrightarrow \mathcal{K}(E),\\
	\epsilon_{\alpha, \beta}&\longmapsto \alpha\big(v-\sum_{e\in s^{-1}(v)}ee^*\big)\beta^*. \notag
	\end{align}
	Observe that the induced isomorphism map $\kappa$ in (\ref{k}) is the composition of 
	\begin{equation}
	\label{ktheorymaps}
	K_0\big(\mathcal{K}(E)\big)\stackrel{\cong}{\longrightarrow} K_0(\oplus_{v\in R}\mathbb{M}_{P_v})\stackrel{\cong}{\longrightarrow} \Z^{R}\end{equation} which sends $v-\sum_{e\in s^{-1}(v)}ee^*$ to $\bf{e}_{v}$ and that the square in \eqref{k} is commutative which can be checked directly using the definitions of involved maps. Therefore 
	$\chi_0: \Coker{\tiny{\begin{pmatrix}B_E^t-I \\ C_E^t\end{pmatrix}}}\longrightarrow K_0\big(L(E)\big)$ is given by 
	$$\chi_0\Big ({\bf e}_{v}+{\rm Im} \tiny{\begin{pmatrix}B_E^t-I \\ C_E^t\end{pmatrix}} \Big)=v,$$ 
	for any $v\in E^0$. Obviously $p_*$ is surjective.
	
	(ii) To show that $\xi'\circ \chi_1'=\rm id$, take $x\in \Ker\tiny{\begin{pmatrix}B_E^t-I \\ C_E^t\end{pmatrix}}$. We lift $U_x$ to 
	$\widetilde{U}_x=\widetilde{V}_x+(1-\widetilde{P}_x)\in \mathbb{M}_h(\mathcal{M}(C(E))),$ where $\widetilde{V}_x$ and $\widetilde{P}_x$ are the elements $V$ 
	and $P$ in $\mathbb{M}_h(\mathcal{M}(L(E)))$ which we obtain by using universal property of the Cohn path algebra $C(E)$ (\S\ref{lpabasis}). Note that $\widetilde{U}_x$ is a partial isometry.  
	
	By \cite[\S 2.4]{cortinas}, we have $\xi'(U_x)=[p_h]-[a p_h a^{-1}]$ with $$a=\tiny{\begin{pmatrix}
		2\widetilde{U}_x-\widetilde{U}_x\widetilde{U}_x^*\widetilde{U}_x& \widetilde{U}_x\widetilde{U}_x^*-1\\ 1-\widetilde{U}_x^*\widetilde{U}_x&\widetilde{U}_x^*
		\end{pmatrix}}.$$ 
	We further have
	\[ a p_h a^{-1}=\tiny{\begin{pmatrix}
		\widetilde{U}_x\widetilde{U}_x^* & 0\\
		0&1-\widetilde{U}_x^*\widetilde{U}_x\end{pmatrix}},\] 
	\[1-\widetilde{U}_x\widetilde{U}_x^*=\widetilde{P}_x-\widetilde{V}_x\widetilde{V}_x^*=\sum_{1\leq i\leq x_w}(w-\sum_{\substack{e\in E^1\\ s(e)=w}}ee^*)E_{[w, i], [w, i]},\] and 
	
	\[1-\widetilde{U}_x^*\widetilde{U}_x=\widetilde{P}_x-\widetilde{V}_x^*\widetilde{V}_x=\sum_{1\leq i\leq -x_w}(w-\sum_{\substack{e\in E^1\\ s(e)=w}}ee^*)E_{\langle w, i\rangle, \langle w, i\rangle}.\]  Since $\kappa$ sends $w-\sum_{\substack{e\in E^1\\ s(e)=w}}ee^*$ to $\bf{e}_{w}$ by \eqref{ktheorymaps}, we have $(\xi'\circ \chi_1')(x)=\sum_{x_w\neq 0}x_w {{\bf e}_w}=x$, implying that $\xi'\circ \chi_1'=\rm id$. This completes the proof. 
\end{proof}


\subsection{Properties of $\overline{K}_1(L(E))$} 
\label{quotientforkone}

Let $E$ be a row-finite graph and $k$ a field. For $v\in E^0$, set $[-v]_1$ to be the image of the element $(1,\cdots, -1,\cdots, 1)^t\in {k^{\times}}^{E^0}$ under the map $\lambda$ given by \eqref{01}. Recall that we have defined the quotient $\overline{K}_1(L(E)):=K_1(L(E))/G_E$, where $G_E$ is the subgroup of $K_1(L(E))$ generated by $[-v]_1$, for $v\in E^0$.

Following from \eqref{01}, $\xi':K_1\big(L(E)\big)\xrightarrow{} \Ker{\tiny{\begin{pmatrix}B_E^t-I\\ C_E^t\end{pmatrix}}} $ factors through $\overline{K}_1(L(E))$. Then we have the following commutative diagram with $p$ the natural projection.
\begin{equation}
\label{thenewxi}
\xymatrix{
K_1\big(L(E)\big) \ar[rr]^{\xi'}\ar[rd]^p & & \Ker{\tiny{\begin{pmatrix}B_E^t-I\\ C_E^t\end{pmatrix}}}\\
& \overline{K}_1\big(L(E)\big)  \ar[ru]^{\xi} &}
\end{equation}

By Lemma \ref{lem:group-homo}, we have also the {\it group homomorphism} $\chi_1: \Ker{\tiny{\begin{pmatrix}B_E^t-I\\ C_E^t\end{pmatrix}}}\xrightarrow{} \overline{K}_1\big(L(E)\big) $ given by $\chi_1=p\circ \chi_1'$ (Definition \ref{def:mapchi1}).  

We mention that the $\Z$-action of $k^{\times}$ is given by $$n\cdot \mu=\mu^{n}$$ for $n\in\Z$ and $\mu\in k^{\times}$. We will consider the subgroup $\{-1,1\}$ of the multiplicative group $k^{\times}$, and the quotient group $\ol{k}^{\times}:=k^{\times}/\{-1,1\}$.

For a row-finite graph $E$, recall that $R$ is the set of non-sink vertices of $E^0$. We denote by $G_1=\{-1,1\}^{R}$ and $G_2=\{-1,1\}^{E^0}$ which are subgroups of ${k^{\times}}^{R}$ and ${k^{\times}}^{E^0}$, respectively. Observe that $\tiny{\begin{pmatrix}B_E^t-I\\ C_E^t\end{pmatrix}} G_1\subseteq G_2$. So $\tiny{\begin{pmatrix}B_E^t-I\\ C_E^t\end{pmatrix}}: {k^{\times}}^{R}\xrightarrow{}{k^{\times}}^{E^0}$ induces the morphism ${\ol{k}^{\times}}^R={k^{\times}}^{R}/G_1\xrightarrow{} {\ol{k}^{\times}}^{E^0} = (k^{\times})^{E^0}/G_2 $, which is also denoted by $\tiny{\begin{pmatrix}B_E^t-I\\ C_E^t\end{pmatrix}}$. The map $\lambda': {k^{\times}}^{E^0}\xrightarrow{} K_1(L(E))$ in \eqref{01} induces a map $\lambda: (k^{\times})^{E^0}/G_2\xrightarrow{} \overline{K}_1(L(E))$ such that $\lambda(xG_2)=\lambda'(x)G_E$.

We have the following commutative diagram with the first three vertical morphisms the natural projection. 
\begin{equation}
	\label{comfortwosses}
\xymatrix@C+2pc@R-1pc{
{k^{\times}}^R\ar[r]^{\tiny{\begin{pmatrix}B_E^t-I\\ C_E^t\end{pmatrix}}} \ar[d]^{}& {k^{\times}}^{E^0}\ar[r]^{\lambda'}\ar[d]^{}& K_1(L(E))\ar[r]^{\xi'}\ar[d]^{p}
&\Z^R\ar[d]^{\rm id} \ar[r]^{\tiny{\begin{pmatrix}B_E^t-I\\ C_E^t\end{pmatrix}}}&\Z^{E^0}\ar[d]^{\rm id}\\
(k^{\times})^{R}/G_1\ar[r]^{\tiny{\begin{pmatrix}B_E^t-I\\ C_E^t\end{pmatrix}}} & (k^{\times})^{E^0}/G_2\ar[r]^{\lambda}& \overline{K}_1(L(E))\ar[r]^{\xi}
& \Z^R \ar[r]^{\tiny{\begin{pmatrix}B_E^t-I\\ C_E^t\end{pmatrix}}}&\Z^{E^0}\\
 }\end{equation} It follows that $\lambda\circ \tiny{\begin{pmatrix}B_E^t-I\\ C_E^t\end{pmatrix}} =0, \xi\circ\lambda=0$, and $\tiny{\begin{pmatrix}B_E^t-I\\ C_E^t\end{pmatrix}}\circ \xi=0$. Clearly $\Ker \tiny{\begin{pmatrix}B_E^t-I\\ C_E^t\end{pmatrix}} \subseteq {\rm Im}\xi$ and $\Ker\xi\subseteq {\rm Im}\lambda$. So we have $\Ker\xi={\rm Im}\lambda.$ And we have the induced map from the cokernel \begin{equation}
 \overline{\Coker} \begin{pmatrix} B_E^t-I\\ C_E^t\end{pmatrix}=\Coker \bigg({\begin{pmatrix}B_E^t-I\\ C_E^t\end{pmatrix}}:(k^{\times})^{R}/G_1\longrightarrow{}(k^{\times})^{E^0}/G_2\bigg)\end{equation} to $\overline{K}_1(L(E))$ by $\lambda$, which is still denoted by $\lambda$.

\begin{prop} \label{newprop}Let $E$ be a row-finite graph. Keep the notation as above. We have the following statements.
\begin{itemize}
\item[(1)] We have the following commutative diagram whose two rows are both split short exact sequences of abelian groups.
\begin{equation}
\label{diagram-fortwo}\xymatrix@C+1pc@R1pc{
 0\ar[r]&\Coker\tiny{\begin{pmatrix}B_E^t-I\\ C_E^t\end{pmatrix}}\ar[r]^{\lambda'}\ar[d]^p& K_1(L(E))\ar[r]^>>>>>>>>>{\xi'}\ar[d]^{p}
&\Ker(\tiny{\begin{pmatrix}B_E^t-I\\ C_E^t\end{pmatrix}}:\Z^R\longrightarrow{}\Z^{E^0})\ar[d]^{\rm id}\ar[r]&0\\
0\ar[r]&  \overline{\Coker} \tiny{\begin{pmatrix}B_E^t-I\\ C_E^t\end{pmatrix}}\ar[r]^{\lambda}& \overline{K}_1(L(E))\ar[r]^>>>>>>>>>{\xi}
& \Ker(\tiny{\begin{pmatrix}B_E^t-I\\ C_E^t\end{pmatrix}}:\Z^R\longrightarrow{}\Z^{E^0})\ar[r]&0 \\
 }\end{equation}
\item[(2)]The homomorphism $\chi_1:\Ker{\tiny{\begin{pmatrix}B_E^t-I\\ C_E^t\end{pmatrix}}}\longrightarrow \overline{K}_1\big(L(E)\big)$ from Definition \ref{def:mapchi1} satisfies $\xi \circ \chi_1 = \rm id,$ where $\xi: \overline{K}_1\big(L(E)\big)\longrightarrow \Ker{\tiny{\begin{pmatrix}B_E^t-I\\ C_E^t\end{pmatrix}}}$ is given in \eqref{thenewxi}.
\end{itemize}
\end{prop}

\begin{proof}(1) We can directly check that the diagram \eqref{diagram-fortwo} commutes. We only need to show that $\lambda:\overline{\Coker} \tiny{\begin{pmatrix}B_E^t-I\\ C_E^t\end{pmatrix}}\xrightarrow{} \overline{K}_1(L(E))$ is injective. Suppose that $\lambda(x)=0$ for $x\in \overline{\Coker} \tiny{\begin{pmatrix}B_E^t-I\\ C_E^t\end{pmatrix}}$. Then there exists $x_0\in \Coker \tiny{\begin{pmatrix}B_E^t-I\\ C_E^t\end{pmatrix}} $ such that $p(x_0)=x$, and we have $(\lambda \circ p)(x_0)=(p\circ \lambda')(x_0)=0$. Thus $\lambda'(x_0)\in G_E$. It follows that $x_0 \in G_2 \cdot {\rm Im}  \tiny{\begin{pmatrix}B_E^t-I\\ C_E^t\end{pmatrix}} \Big/ {\rm Im}\tiny{\begin{pmatrix}B_E^t-I\\ C_E^t\end{pmatrix}}$. Therefore $x=0$ in $\overline{\Coker} \tiny{\begin{pmatrix}B_E^t-I\\ C_E^t\end{pmatrix}}.$


(2) We observe that $\xi\circ\chi_1=\xi\circ p\circ \chi_1'=\xi'\circ\chi_1'={\rm id}$ by Proposition \ref{prop-section}(2).
\end{proof}

By Proposition \ref{newprop}(1), we have the following consequence.

\begin{cor}\label{cor-new} Let $E$ be a row-finite graph. Keep the notation as above. Then $\overline{K}_1(L(E))$ is isomorphic to a direct sum:
$${\rm Coker} \bigg({\tiny{\begin{pmatrix}B_E^t-I\\ C_E^t \end{pmatrix}}}: (k^{\times})^{R}/G_1\longrightarrow (k^{\times})^{E^0}/G_2\bigg)\bigoplus \Ker \bigg({\tiny{\begin{pmatrix}B_E^t-I\\ C_E^t \end{pmatrix}}}: \Z^{R}\longrightarrow \Z^{E^0}\bigg).$$ 
\end{cor}

Let $E$ be a row-finite graph, $H$ a hereditary saturated subset of $E^0$, $R$ the set of regular vertices  and $S=E^0\backslash R$ the set of all sinks in $E^0$. We write the adjacency matrix $A_E$ of $E$ with respect to the decomposition $$R\cap H, ~~~~S\cap H, ~~~~R\setminus H, \text{ and }~~~~S\setminus H .$$ Therefore
 \[A_E=\tiny\begin{pmatrix}A& \alpha& 0&0\\
0&0&0&0\\
X&\xi& B&\beta\\
0&0&0&0
\end{pmatrix}.\]

\begin{lem}\label{ssk} Let $E$ be a row-finite graph and $H$ a hereditary saturated subset of $E^0$. There are commutative diagrams

$$\xymatrix@C+1.5pc@R-1pc{
\Ker{\tiny{\begin{pmatrix}A^t-I\\ \alpha^t \end{pmatrix}}}\ar[r]^{} \ar[d]^{\sigma}& {k^{\times}}^{R\cap H}\ar[r]^{\tiny{\begin{pmatrix}A^t-I\\ \alpha^t\end{pmatrix}}}\ar[d]^{\sigma}& {k^{\times}}^{H}\ar[r]\ar[d]^{\sigma}
& \Coker\tiny{\begin{pmatrix}{A^t}-I\\ \alpha^t\end{pmatrix}} \ar[d]^{\overline{\sigma}} \\
\Ker{\tiny{\begin{pmatrix}
A^t-I& X^t\\
\alpha^t & \xi^t\\
0&B^t-I\\
0&\beta^t
\end{pmatrix}}}\ar[r]^{} \ar[d]^{\sigma'}& {k^{\times}}^{R}\ar[r]^{\tiny{\begin{pmatrix}A^t-I& X^t\\
\alpha^t & \xi^t\\
0&B^t-I\\
0&\beta^t
\end{pmatrix}}}\ar[d]^{\sigma'}& {k^{\times}}^{E^0}\ar[r]\ar[d]^{\sigma'}
& \Coker{\tiny{\begin{pmatrix}A^t-I& X^t\\
\alpha^t & \xi^t\\
0&B^t-I\\
0&\beta^t
\end{pmatrix}}} \ar[d]^{\overline{\sigma'}} \\
\Ker\tiny{\begin{pmatrix}B^t-I\\ \beta^t\end{pmatrix}} \ar[r]^{} & {k^{\times}}^{{R\setminus H}}\ar[r]^{\tiny{\begin{pmatrix}B^t-I\\ \beta^t\end{pmatrix}}}& {k^{\times}}^{{E^0\setminus H}}\ar[r]
& \Coker\tiny{\begin{pmatrix}B^t-I\\ \beta^t\end{pmatrix}} 
 }$$ where $\sigma$ is the natural inclusion sending an element $(m_1,\cdots,m_n)^t$ in ${k^{\times}}^{H}$ to $(m_1,\cdots, m_n,1\cdots, 1)^t\in {k^{\times}}^{E^0}$ and $\sigma'$ is the natural projection.
\end{lem}
\begin{proof} We can check directly that the middle two squares are commutative. 
\end{proof}

\begin{lem}\label{ssz} Let $E$ be a row-finite graph, $E^0=R\sqcup S$, where $S$ is the set of sinks and $R=E^0\setminus S$. Let $H$ be a hereditary saturated subset of $E^0$ and $A_E=\tiny{\begin{pmatrix}A_{E_H}&0\\ X&A_{E/H}\end{pmatrix}}$ the adjacency matrix of $E$ with respect to the decomposition $E^0=H\sqcup (E^0\setminus H)$.  Let $S'$ be the set of sink vertices in $E/H$. Then there is a commutative diagram

\begin{equation}\label{cccom}
\begin{tikzpicture}[descr/.style={fill=white,inner sep=1.4pt}] 
        \matrix (m) [
            matrix of math nodes,
            row sep=2.9em,
            column sep=5em,
            text height=1.5ex, text depth=0.25ex
        ]
        { & \Ker{\begin{pmatrix}A_{E_H}^t-I\end{pmatrix}} & \Ker{\begin{pmatrix}A_{E}^t-I\end{pmatrix}} & \Ker{\begin{pmatrix}A_{E/H}^t-I \end{pmatrix}} &\\
          0 & {\Z}^{H\setminus S} & {\Z}^{R} & {\Z}^{(E^0\setminus H)\setminus S'}&0 \\
           0& {\Z}^{H} & {\Z}^{E^0} & {\Z}^{E^0\setminus H}&0 \\
            & \Coker{\begin{pmatrix}A_{E_H}^t-I\end{pmatrix}}      &      \Coker{\begin{pmatrix}A_E^t-I\end{pmatrix}}            & \Coker{\begin{pmatrix}A_{E/H}^t-I\end{pmatrix}}    & \\
                  };

        \path[overlay,->, font=\scriptsize,>=latex]
        (m-1-2) edge node[auto] {\(\tau\)} (m-1-3)
        (m-1-2) edge node[auto] {\(\)} (m-2-2)
        (m-1-3) edge node[auto] {\(\tau'\)} (m-1-4)
        (m-1-3) edge  node[auto] {\(\)} (m-2-3)
        (m-1-4) edge (m-2-4)
        (m-1-4) edge[out=355,in=165,blue] node[descr,yshift=0.3ex] {$\delta_{E_H}, E$} (m-4-2)
        (m-2-1) edge (m-2-2)
        (m-2-4) edge (m-2-5)
         (m-3-1) edge (m-3-2)
        (m-3-4) edge (m-3-5)
        (m-2-2) edge node[auto] {\(\tau\)} (m-2-3)
        (m-2-2) edge node[auto] {\(A_{E_H}^t-I\)}  (m-3-2)
        (m-2-3) edge node[auto] {\(\tau'\)} (m-2-4)
        (m-2-3) edge node[auto] {\(A_{E}^t-I\)} (m-3-3)
        (m-2-4) edge  node[auto] {\(A_{E/H}^t-I\)}(m-3-4)
        (m-3-2) edge node[auto] {\(\tau\)} (m-3-3)
        (m-3-2) edge (m-4-2)
        (m-3-3) edge node[auto] {\(\tau'\)} (m-3-4)
        (m-3-3) edge node[auto] {\(\overline{\tau}\)} (m-4-3)
        (m-3-4) edge (m-4-4)
        (m-4-2) edge node[auto] {\(\overline{\tau}\)} (m-4-3)
        (m-4-3) edge node[auto] {\(\overline{\tau'}\)} (m-4-4);
\end{tikzpicture}
\end{equation} where $\tau$ is the natural inclusion and $\tau'$ is the natural projection. 
The connecting map $\delta$ is given by $\delta_{E_H, E}(x)=[X^tx]$ for $x\in \Ker({A_{E/H}^t-I})$. 
\end{lem}
\begin{proof} We observe that $(E^0\setminus S)\setminus (H\setminus S)=(E^0\setminus H)\setminus S'$. Then the second row of \eqref{cccom} is a short exact sequence. We check directly that the two squares in the middle of the above diagram are commutative. By the Snake Lemma we have the short exact sequence 
\begin{multline}
\Ker \big(A_{E_H}^t-I \big) \stackrel{\tau}{\longrightarrow}  \Ker \big(A_{E}^t-I\big) \stackrel{\tau'}{\longrightarrow} \Ker \big(A_{E/H}^t-I\big) \stackrel{\delta}{\longrightarrow} \\ \stackrel{\delta}{\longrightarrow} \Coker \big(A_{E_H}^t-I\big)  \stackrel{\overline \tau}{\longrightarrow}
 \Coker \big(A_E^t-I\big) \stackrel{\overline {\tau'}}{\longrightarrow}
 \Coker \big(A_{E/H}^t-I\big),
 \end{multline}
  with the connecting map $\delta_{E_H, E}$ defined by $\delta_{E_H, E}(x)=X^tx$. 
\end{proof}

%

\begin{lem}\label{2dsnakelemma}  Let $E$ be a row-finite graph, $S$ the set of sink vertices of $E^0$ and $E^0=R\sqcup S$ with $R=E^0\setminus S$. Let $H$ be a hereditary saturated subset of $E^0$, $I=\langle H\rangle$ the graded ideal of $L(E)$ 
generated by $H$, and $A_E$ the adjacency matrix of $E$ with respect to the decomposition $E^0=H\sqcup (E^0\setminus H)$, where the vertices in $H$ are listed first.  
Let $S'$ be the set of sink vertices in $E/H$. Then there is a commutative diagram
\begin{equation}\label{3dsnakel}
\begin{tikzpicture}[descr/.style={fill=white,inner sep=1.5pt}]
        \matrix (m) [
            matrix of math nodes,
            row sep=2.5em,
            column sep=.8em,
            text height=1.5ex, text depth=0.25ex
        ]
        {  &&&\Ker(A_{E_H}^t-I) &&\Ker(A_{E}^t-I)&& \Ker(A_{E/H}^t-I)&& \\
            & &\Ker(\phi_{E_{H}})&& \Ker(\phi_E) && \Ker(\phi_{E/H})&&& \\
           &0&&\Z^{H\setminus S}&&\Z^{R}&&\Z^{(E^0\setminus H)\setminus S'}&&0&\\
          0 &&K_0^{\gr}(I)&&K_0^{\gr}(L(E))&&K_0^{\gr}(L(E)/I)&&0&&\\            
             &0&&\Z^{H}&&\Z^{E^0}&&\Z^{E^0\setminus H}&&0&\\
            0 &&K_0^{\gr}(I)&&K_0^{\gr}(L(E))&&K_0^{\gr}(L(E)/I)&&0&&\\  
            &&&\Coker(A_{E_H}^t-I) &&\Coker(A_{E}^t-I)&& \Coker(A_{E/H}^t-I)&& \\
             &&K_0(I) &&K_0(L(E))&& K_0(L(E)/I)&&& \\
             };

        \path[overlay,->, font=\scriptsize,>=latex]
        (m-1-4) edge (m-1-6) edge node[descr,yshift=0.3ex] {} (m-3-4)  edge (m-2-3)
        (m-1-6) edge (m-1-8) edge node[descr,yshift=0.3ex] {} (m-3-6) edge  node[auto] {\(\)} (m-2-5)
        (m-1-8) edge[out=362,in=189,blue] node[descr,yshift=0.3ex] {$\delta$} (m-7-4)
        (m-1-8) edge  (m-3-8)
          (m-1-8) edge[blue]  node[auto] {\(\psi\)} (m-2-7)
        (m-2-3) edge  (m-2-5)
        (m-2-3) edge  (m-4-3)
          (m-2-3) edge  node[auto] {\(\)} (m-2-5)
        (m-2-5) edge node[auto] {\(\)} (m-2-7)
        (m-2-5) edge node[auto] {\(\)}(m-4-5)
        (m-2-7) edge[out=375,in=175,blue] node[descr,yshift=0.3ex] {$\rho$} (m-8-3)
         (m-2-7) edge node[auto] {\(\)}(m-4-7)
        (m-3-2) edge (m-3-4)
        (m-3-4) edge node[descr,yshift=0.3ex] {} (m-3-6) edge node[auto] {\(\psi\)}  (m-4-3) edge node[descr,yshift=0.3ex] {\(\)} (m-5-4)
        (m-3-6) edge node[descr,yshift=0.3ex] {} (m-3-8) edge node[auto] {\(\psi\)}   (m-4-5) edge node[descr,yshift=0.3ex] {\(A^t_E-I\)} (m-5-6)
        (m-3-8) edge node[auto] {\(\)} (m-3-10)  edge node[auto] {\(\psi\)}  (m-4-7) edge node[descr,yshift=0.3ex] {} (m-5-8)
        (m-4-1) edge (m-4-3)
        (m-4-3) edge (m-4-5) edge node[auto] {\(\)} (m-6-3)
        (m-4-5) edge (m-4-7) edge node[auto] {\(\phi\)} (m-6-5)
        (m-4-7) edge (m-4-9) edge node[auto] {\(\)} (m-6-7)       
        (m-5-2) edge (m-5-4) 
        (m-5-4) edge node[descr,yshift=0.3ex] {} (m-5-6) edge node[descr,yshift=0.3ex] {} (m-7-4) edge node[auto] {\(\psi\)}  (m-6-3) 
        (m-5-6) edge node[descr,yshift=0.3ex] {} (m-5-8) edge node[descr,yshift=0.3ex] {} (m-7-6) edge node[auto] {\(\psi\)}  (m-6-5) 
        (m-5-8) edge (m-5-10) edge node[descr,yshift=0.3ex] {} (m-7-8) edge node[auto] {\(\psi\)}  (m-6-7) 
        (m-6-1) edge (m-6-3)
        (m-6-3) edge (m-6-5) edge (m-8-3)
        (m-6-5) edge (m-6-7) edge (m-8-5)
        (m-6-7) edge (m-6-9) edge (m-8-7)
        (m-7-4) edge node[descr,yshift=0.3ex] {} (m-7-6) edge[blue] node[auto] {\(\psi\)} (m-8-3)
        (m-7-6) edge node[descr,yshift=0.3ex] {} (m-7-8) edge (m-8-5)
        (m-7-8) edge (m-8-7)
        (m-8-3) edge (m-8-5)
        (m-8-5) edge (m-8-7);
\end{tikzpicture}
\end{equation} such that $\rho\circ \psi=\psi\circ\delta$.
\end{lem}
\begin{proof} By \eqref{aim}, we have a commutative diagram 
\[\xymatrix{
 {\Z}^R \ar[r]^{A^t_E-I} \ar[d]^{\psi} & {\Z}^{R}\oplus {\Z}^S\ar[d]^{\psi}\\
 K_0^{\gr}(\L(E))\ar[r]^{\phi}& K_0^{\gr}(\L(E)).} 
 \]
 We can check directly that the other squares of \eqref{3dsnakel} are commutative. Combining with Lemma \ref{ssz} and \eqref{shortexactfor}, by the snake Lemma, there exist $\delta$ and $\rho$ such that  $\rho\circ \psi=\psi\circ\delta$.
\end{proof}

\begin{rmk}\label{hfgftrhyrrr} In order to obtain a similar result as Lemma \ref{2dsnakelemma} for an arbitrary graph $E$ (possibly with infinite emitters $v\in E^0$ with $|s^{-1}(v)|=\infty$), one needs to develop techniques for infinite emitters in order to have similar conditions as Lemma \ref{ssz} to use Snake Lemma and to have a short exact sequence \begin{align*}\CD
 0@>>> K_0^{\gr}(I)@>{}>>K_0^{\gr}(L(E))@>{}>>K_0^{\gr}(L(E)/I)@>>>0, 
\endCD
\end{align*} where $I$ is a graded ideal of $L(E)$ so that one can establish Lemma~\ref{2dsnakelemma}. 
\end{rmk}

In the following, we will use the notation $\ol{[u]}_1$ to indicate the class in $\ol{K}_1(L(E))$ of a $K_1$-class $[u]_1$.

For a hereditary saturated subset $H$ of $E^0$, we have natural homomorphisms of algebras $L(E_H)\xrightarrow{}L(E)\xrightarrow{}L(E/H)$. 
We have the induced morphisms of groups $K_1(L(E_H))\xrightarrow{}K_1(L(E))\xrightarrow{}K_1(L(E/H))$ and $\overline{K}_1(L(E_H))\xrightarrow{}\overline{K}_1(L(E))\xrightarrow{}\overline{K}_1(L(E/H))$. 

The following part of this subsection is devoted to show that $\chi_1$ is compatible with $\overline{K}_1(L(E))\xrightarrow{}\overline{K}_1(L(E_H))\xrightarrow{}\overline{K}_1(L(E/H))$; see Proposition \ref{newprop}.

Recall that we use $\overline{\Coker} \begin{pmatrix}B_E^t-I\\ C_E^t\end{pmatrix} $ to denote \begin{equation*}
 \Coker \bigg({\begin{pmatrix}B_E^t-I\\ C_E^t\end{pmatrix}}:(k^{\times})^{R}/G_1\longrightarrow{}(k^{\times})^{E^0}/G_2\bigg).\end{equation*}

\begin{lem}\label{Y} Let $E$ be a row-finite graph, $S$ the set of sink vertices of $E$, $E^0=R\sqcup S$ with $R=E^0\setminus S$, $H$ a hereditary saturated subset of $E^0$ and $S'$ the set of sink vertices in $E/ H$. We write the adjacency matrix $A_E$ of $E$ by arranging vertices in $H$ first such that $A_E=\tiny{\begin{pmatrix}A_{E_H}&0\\
X&A_{E/H}\end{pmatrix}}$. Then there is a commutative diagram
\[\xymatrix@C+.1pc@R+1.1pc{
\overline{\Coker}\left(A_{E_H}^t-I\right) \ar[r]^>>>>>>>>{\lambda_{E_H}} \ar[d]^{\overline{\sigma}}& \overline{K}_1\big(L(E_H)\big) \ar[r]^>>>>>>>>>{\xi_{E_H}} \ar[d]^{\alpha} & \Ker\left(A_{E_H}^t-I: \Z^{H\setminus S}\xra \Z^{H}\right)\ar@<1.5ex>[l]^<<<<<<<<{{\chi_1}_{E_H}}\ar[d]^{\tau}\\
\overline{\Coker}\left({\begin{matrix}A_E^t-I\end{matrix}}\right) \ar[r]^>>>>>>>>>{\lambda_E}\ar[d]^{\overline{\sigma'}} & \overline{K}_1\big(L(E)\big) \ar[r]^>>>>>>>>>{\xi_E} \ar[d]^{\alpha'}& \Ker\left({\begin{matrix}A_E^t-I\end{matrix}}: {\Z}^{R}\xra{\Z}^{E^0}\right)\ar[d]^{\tau'}\ar@<1.5ex>[l]^<<<<<<<<{{\chi_1}_E}\\
 \overline{\Coker}\left({A_{E/H}^t-I}\right) \ar[r]^>>>>>>{\lambda_{E/H}} & \overline{K}_1\big(L(E/H)\big) \ar[r]^>>>>>>{\xi_{E/ H}} & \Ker\left({A_{E/H}^t-I}: {\Z}^{(E^0\setminus H)\setminus S'}\xra{\Z}^{E^0\setminus H}\right)\ar@<1.5ex>[l]^<<<{{\chi_1}_{E/H}}
 }\]
  with $\alpha\circ {\chi_1}_{E_H}={\chi_1}_E\circ\tau$ and $\alpha'\circ {\chi_1}_E={\chi_1}_{E/H}\circ \tau'$, where $\xi_E$ is given by \eqref{defxi}. Here $\alpha$ and $\alpha'$ are induced by the canonical homomorphisms of algebras.
\end{lem}

\begin{proof} We can check directly that $\alpha\circ \lambda_{E_H}=\lambda_E\circ \overline{\sigma}$ and $\alpha'\circ \lambda_E=\lambda_{E/H}\circ\overline{{\sigma}'}$. In order to show that $\tau\circ \xi_{E_H}=\xi_E\circ \alpha$ and $\tau'\circ \xi_E=\xi_{E/H}\circ \alpha'$, we use the following commutative diagram 

\begin{equation}
\xymatrix@C+.1pc@R+1.1pc{0\ar[r]&\mathcal{K}(E_H)\ar[r]^{}\ar[d]^{\Delta}&C(E_H)\ar[r]^{}\ar[d]^{l}&L(E_H)\ar[r]\ar[d]^{}&0\\
0\ar[r]&\mathcal{K}(E)\ar[r]^{}\ar[d]^{\Delta'}&C(E) \ar[d]^{p}\ar[r]^{}&L(E)\ar[d]^{}\ar[r]^{}&0\\
0\ar[r]&\mathcal{K}(E/H)\ar[r]^{}&C(E/H) \ar[r]^{}&L(E/H)\ar[r]^{}&0,
}
\end{equation} where $l: C(E_H)\rightarrow C(E)$ and $p: C(E)\rightarrow C(E/H)$ are canonical homomorphisms of algebras. It follows that the following diagram is commutative and such that in each row the composition of the horizontal morphisms coincides with $\xi'_{E_H}, \xi'_E$ and $\xi'_{E/H}$ respectively.
\begin{equation*}
\xymatrix@C+.7pc@R+.7pc{K_1\big(L(E_H)\big)\ar[r]\ar[d]^{\alpha}\ar@{.>}@/^2pc/[rr]^>>>>>>>>>>{\xi_{E_H}'}&K_0\big(\mathcal{K}(E_H)\big)\ar[r]^{}\ar[d]^{}&\Z^{H\setminus S}\ar[d]^{\tau}\\
K_1\big(L(E)\big)\ar[r]\ar[d]^{\alpha'} \ar@{.>}@/^2pc/[rr]^ >>>>>>>>>>{\xi_{E}'}&K_0\big(\mathcal{K}(E)\big)\ar[r]^{}\ar[d]^{}&\Z^{E^0\setminus S}\ar[d]^{\tau'}\\K_1\big(L(E/ H)\big)\ar[r] \ar@{.>}@/^2pc/[rr]^ >>>>>>>>>>{\xi_{E/H}'}&K_0\big(\mathcal{K}(E/H)\big)\ar[r]^{}&\Z^{{(E^0\setminus H)\setminus S'}}.
}
\end{equation*} Combining with \eqref{thenewxi}, we have the following diagram with $\xi_{E_H}'=\xi_{E_H}\circ p$, $\xi_{E}'=\xi_{E}\circ p$ and $\xi_{E/H}'=\xi_{E/H}\circ p$ and the left two squares commute. 
\begin{equation*}
\xymatrix@C+.7pc@R+.7pc{K_1\big(L(E_H)\big)\ar[r]^{p}\ar[d]^{\alpha} \ar@{.>}@/^2pc/[rr]^<<<<<<<<<{\xi_{E_H}'}&\overline{K}_1\big(L(E_H)\big)\ar[r]^{\xi_{E_H}}\ar[d]^<<<<{\alpha}&\Z^{H\setminus S}\ar[d]^{\tau}\\
K_1\big(L(E)\big)\ar[r]^{p}\ar[d]^{\alpha'} \ar@{.>}@/^2pc/[rr]^ <<<<<<<<<{\xi_{E}'}& \overline{K}_1\big(L(E)\big)\ar[r]^{\xi_E}\ar[d]^<<<<{\alpha'}&\Z^{E\setminus S}\ar[d]^{\tau'}\\K_1\big(L(E/ H)\big)\ar[r]^{p} \ar@{.>}@/^2pc/[rr]^ <<<<<<<<<{\xi_{E/H}'}& \overline{K}_1\big(L(E/H)\big)\ar[r]^{\xi_{E/H}}&\Z^{{(E^0\setminus H)\setminus S'}}
}
\end{equation*} We deduce that $\tau\circ \xi_{E_H}=\xi_E\circ \alpha$ and $\tau'\circ \xi_E=\xi_{E/H}\circ \alpha'$.

Now we show that $\alpha\circ {\chi_1}_{E_H}={\chi_1}_{E}\circ \tau$. Take $x\in\Ker(A_{E_H}^t-I:\Z^{H\setminus S}\longrightarrow \Z^{H})$. Note that $x$ has finitely many nonzero entries 
$x_{v_1}, \cdots, x_{v_l}$ with $v_1, \cdots, v_l\in H\setminus S$. By \eqref{twosets}, 
\[
L_x^+(E_H)=\Big \{(e, i)\;|\; e\in E^1_H, 1\leq i\leq -x_{s(e)}\Big \}\bigcup \Big\{(v, i)\;|\; v\in E^0_H, 1\leq i\leq x_v \Big\},
\] and 
\[L_x^-(E_H)=\Big\{(e, i)\;|\; e\in E^1_H, 1\leq i\leq x_{s(e)}\Big\}\bigcup \Big\{(v, i)\;|\; v\in E^0_H, 1\leq i\leq -x_v\Big\}.
\]
 Let $h_1$ be the cardinality of the sets $L_x^+(E_H)$ and $L_x^-(E_H)$. By definition of ${\chi_1}_{E_H}$, ${\chi_1}_{E_H}(x)=\ol{[U_x]}_1$ with $U_x$ a matrix in $\mathbb{M}_{h_1}(\mathcal{M}(L(E_H)))$.  Let $$x'=\tau(x)=\tiny\begin{pmatrix}x\\ 0\\ \vdots\\ 0\end{pmatrix}.$$ The element $x'$ has the same nonzero entries $x_{v_1},\cdots, x_{v_l}$ as $x$. Since $v_1, \cdots, v_l$ belong to $H$, we have 
 \[L_{x'}^+=\Big \{(e, i)\;|\; e\in E^1, 1\leq i\leq -x_{s(e)} \Big \}\bigcup \Big\{(v, i)\;|\; v\in E^0, 1\leq i\leq x_v\Big\}=L_x^+(E_H),\] and similarly $L_{x'}^-=L_x^-(E_H)$. For the bijections $[\cdot]: L_{x'}^+\xra\{1, \cdots, h_1\}$ and $\langle \cdot \rangle: L_{x'}^-\xra \{1, \cdots, h_1\}$, we use the same bijections $[\cdot]: L_{x}^+(E_H)\xra\{1, \cdots, h_1\}$ and $\langle \cdot \rangle: L_{x}^-(E_H)\xra \{1, \cdots, h_1\}$. Then we have \[{\chi_1}_E(\tau(x))={\chi_1}_E(x')=\ol{[U_x]}_1={\chi_1}_{E_H}(x)=(\alpha\circ {\chi_1}_{E_H})(x).\] 
 
Now we show that $\alpha'\circ {\chi_1}_E={\chi_1}_{E/H}\circ \tau'$. Suppose that $x\in \Ker(A_E^t-I)$. 
The element $x$ has finitely many nonzero entries $x_{v_1}, \cdots, x_{v_l}, x_{v_{l+1}},\cdots, x_{v_n}$ with 
$v_1, \cdots, v_l\in H$ and $v_{l+1}, \cdots, v_n\notin H$. Then 
\[L_x^+=\Big \{(z, i)\in L_x^+, r(z)\in H\Big\}\bigcup \Big\{(z, i)\in L_x^+, r(z)\notin H\Big\},\] and 
\[L_x^-=\Big\{(z, i)\in L_x^-, r(z)\in H\Big\}\bigcup \Big\{(z, i)\in L_x^-, r(z)\notin H\Big\}.\] And $\{(z, i)\in L_x^+, r(z)\in H\}$ and $\{(z, i)\in L_x^-, r(z)\in H\}$ have the same cardinality $h'$; 
$\{(z, i)\in L_x^+, r(z)\notin H\}$ and $\{(z, i)\in L_x^-, r(z)\notin H\}$ have the same cardinality $h''$. We choose two bijections $[\cdot]: L_x^+\xra \{1, \cdots, h'+h''\}$ and $\langle \cdot \rangle: L_x^-\xra \{1, \cdots, h'+h''\}$ such that $[x, i]=\langle y, j\rangle$ implies $r(x)=r(y)$ and that $[\cdot]$ restricts to an isomorphism from $\{(x, i)\in L_x^+, r(x)\in H\}$ onto $\{1, \cdots, h'\}$. In this case $[\cdot]$ restricts to an isomorphism from $\{(x, i)\in L_x^+, r(x)\notin H\}$ onto $\{h'+1,\cdots, h'+h''\}$. 
 Then we have 
 \[{\chi_1}_E(x)=\ol{[V_x+1-P_x]}_1\in \mathbb{M}_{h'+h''}(\mathcal{M}(L(E))),\] where $V_x=\tiny{\begin{pmatrix}V_{11}&V_{12}\\ V_{21}&V_{22}\end{pmatrix}}$ with $V_{11}\in \mathbb{M}_{h'}(\mathcal{M}(L(E_H)))$, and $V_{12}$ and $V_{21}$ matrices with entries in the two sided ideal $\langle H\rangle$ of $L(E)$.  On the other hand, $P_x=\tiny{\begin{pmatrix}P_{11}&0\\ 0&P_{22}\end{pmatrix}}$ is a $(h'+h'')\times(h'+h'')$-matrix 
with $P_{11}$ a $h'\times h'$-matrix with entries in the two sided ideal $\langle H\rangle$ of $L(E)$. The element $\tau'(x)$ has finitely many nonzero entries $ x_{v_{l+1}},\cdots, x_{v_n}$ with $v_{l+1}, \cdots, v_n\notin H$. Then we have \[L_{\tau'(x)}^+(E/H)=\{(x, i)\in L_x^+, r(x)\notin H\},\] and \[L_{\tau'(x)}^-(E/H)=\{(x, i)\in L_x^-, r(x)\notin H\}.\] We use the restrictions of $[\cdot]: L_x^+\xra \{1, \cdots, h'+h''\}$ and $\langle \cdot \rangle: L_x^-\xra \{1, \cdots, h'+h''\}$. Then we have ${\chi_1}_{E/H}(\tau'(x))=\ol{[V_{22}+1-P_{22}]}_1$. Therefore we have 
\begin{equation*}
\begin{split}
(\alpha' \circ {\chi_1}_E)(x)
&=\alpha' ({\chi_1}_E(x))\\
&= \ol{\Big[ \tiny{\begin{pmatrix}0&0\\ 0& V_{22}\end{pmatrix}}+{\begin{pmatrix}I&0\\ 0& I\end{pmatrix}}-{\begin{pmatrix}0&0\\ 0& P_{22}\end{pmatrix}}\Big]}_1\\
&=\ol{ \Big[ \tiny{\begin{pmatrix}I&0\\ 0& V_{22}+1-P_{22}\end{pmatrix}} \Big]}_1\\
&= \ol{[ V_{22}+1-P_{22} ]}_1\\
&=({\chi_1}_{E/H}\circ \tau')(x).
\end{split}
\end{equation*}
This shows $\alpha'\circ {\chi_1}_E={\chi_1}_{E/H}\circ \tau'$ and thus the proof is complete. 
\end{proof}

We will need the following lemma whose proof is straightforward. 

\begin{lem} \label{fact} Consider the following commutative diagram consisting of two split short exact sequences of modules over a ring $R$  
\begin{equation}
\xymatrix{0\ar[r]&A\ar[r]^{i}\ar[d]^{f}&B\ar[r]^{j}\ar[d]^{g}&C\ar@{.>}@/^/[l]^{s}\ar[r]\ar[d]^{h}&0\\
0\ar[r]&A' \ar[r]^{i'}& B' \ar[r]^{j'}&C'\ar[r]\ar@{.>}@/^/[l]^{s'}&0,
}
 \end{equation}
 such that $js=\idd_C$,  $j's'=\idd_{C'}$ and  $gs=s'h$. Then we have the following commutative diagram \begin{equation*}
\xymatrix{B\ar[r]^<<<<<<{\scriptscriptstyle{\cong\ \ }}\ar[d]^{g}&A\oplus C\ar[d]^{\tiny{\begin{pmatrix}f&0\\ 0&h\end{pmatrix}}}\\
 B' \ar[r]^<<<<<<{\scriptscriptstyle{\cong\ \ }}&A'\oplus C'.
}
\end{equation*} 
The isomorphism in the top right arrow is $\begin{pmatrix} i^{-1}(1-sj)\\ j \end{pmatrix}$, with inverse $\begin{pmatrix} i & s \end{pmatrix}$.
\end{lem}

Combining Lemma \ref{Y} and Lemma \ref{fact}, we have the following immediate consequence.

\begin{prop}\label{propforquotientcom} Let $E$ be a row-finite graph and $H$ a hereditary saturated subset of $E^0$. We write the adjacency matrix $A_E$ of $E$ 
by arranging vertices in $H$ first such that $A_E=\tiny{\begin{pmatrix}A_{E_H}&0\\
X&A_{E/H}\end{pmatrix}}$. There is a commutative diagram
$$\xymatrix@C+.1pc@R+1.1pc{\overline{K}_1\big(L(E_H)\big)\ar[d]^{\alpha}\ar[r]^>>>>>>>{\scriptscriptstyle{\cong\ \ }}&
\overline{\Coker}{\begin{pmatrix}A_{E_H}^t-I\end{pmatrix}} \oplus \Ker{\begin{pmatrix}A_{E_H}^t-I\end{pmatrix}}\ar[d]^{\tiny{\begin{pmatrix}\overline{\sigma}&0\\ 0&\tau\end{pmatrix}}}\\\overline{K}_1\big(L(E)\big)\ar[d]^{\alpha'}\ar[r]^>>>>>>>>>>{\scriptscriptstyle{\cong\ \ }}& \overline{\Coker}{(\begin{matrix}A_E^t-I\end{matrix})} \oplus \Ker{(\begin{matrix}A_E^t-I\end{matrix})}\ar[d]^{\tiny{\begin{pmatrix}\overline{\sigma'}&0\\ 0&\tau'\end{pmatrix}}}\\\overline{K}_1\big(L(E/H)\big) \ar[r]^>>>>>{\scriptscriptstyle{\cong\ \ }}&\overline{\Coker}{\begin{pmatrix}A_{E/H}^t-I\end{pmatrix}}\oplus \Ker{\begin{pmatrix}A_{E/H}^t-I\end{pmatrix}}.
 }$$ 
\end{prop}

\subsection{Graded ideals of Leavitt path algebras}

Let $E$ be a row-finite graph and $H\subseteq E^0$ a hereditary saturated subset of $E^0$. We define
$$F(H)= \{\alpha \;|\; \alpha= e_1\cdots e_n \text{~is a finite path in~} E \text{ with } r(e_n)\in H \text{ and } s(e_n)\notin {H} \}.$$  
Note that the paths in $F(H)$ must be of length $1$ or greater. 
Let $\overline{F}(H)$ denote another copy of $F(H)$ and write $\overline{\alpha}$ for the copy of $\alpha$ in $F(H)$.

Define $\overline{E}_{H}$ to be the graph with
\begin{equation}
\begin{split}
 &\overline{E}^0_{H}= H \cup F(H)\\
 &\overline{E}^1_{H}= \{e\in E^1 : s(e)\in  H\} \cup \overline{F}(H)\end{split}
 \end{equation}
and we extend $s$ and $r$ to $\overline{E}^1_H$ by defining $s(\overline{\alpha}) = \alpha$ and $r(\overline{\alpha}) = r(\alpha)$ for $\overline{\alpha}\in \overline{F}(H)$.
 
Recall from \cite[Theorem 6.1]{rt2014} the graded ideal $I_H$ generated by $H$ in $L(E)$ is isomorphic to $L(\overline{E}_{H})$ with $H\subseteq E^0$ a hereditary saturated subset of $E^0$. More precisely, the isomorphism \begin{equation}\label{ideals}L(\overline{E}_{H})\longrightarrow{} I_{H}\end{equation} of algebras is defined by $v\mapsto \mathcal{Q}_v$ for each $v\in \overline{E}^0_{H}$ and $e\mapsto T_e$, $e^*\mapsto T_e^*$ for each $e\in \overline{E}^1_{H}$, where 
\begin{equation*}
\mathcal{Q}_v=\begin{cases}v, & v\in H;\\
\alpha\alpha^*, & v=\alpha\in F(H);
\end{cases}
\end{equation*}

\begin{equation*}
T_e=\begin{cases}e, & e\in E^1;\\
\alpha, & e=\overline{\alpha}\in \overline{F}(H);\\
\end{cases}
\end{equation*} and \begin{equation*}
T_e^*=\begin{cases}e^*, & e\in E^1;\\
\alpha^*, & e=\overline{\alpha}\in \overline{F}(H).\\
\end{cases}
\end{equation*}


\begin{lem}
\label{kone} Let $E$ be a row-finite graph and $H$ a hereditary saturated subset of $E^0$. Then the natural homomorphism of algebras 
\begin{equation}
\begin{split}
\label{nn}
L(E_H) &\longrightarrow L(\overline{E}_{H})\\
 v  &\longmapsto v\\
 e  &\longmapsto e\\
 e^*  &\longmapsto e^*, 
 \end{split}
 \end{equation}  
    for $v\in H$ and $e\in E_H^1$ induces isomorphisms $K_1\big(L(E_H)\big)\xra K_1\big(L(\overline{E}_{H})\big)$  and $\overline{K}_1\big(L(E_H)\big)\xra \overline{K}_1\big(L(\overline{E}_{H})\big)$ of $K$-groups. 
\end{lem}

\begin{proof} Denote by $\overline{A}$ the adjacency matrix for the graph $\overline{E}_{H}$. We arrange vertices in $H$ first such that $\overline{A}=\tiny{\begin{pmatrix}A_{E_H}&0\\
Z&0 \end{pmatrix}}$ where $Z$ is a matrix such that each row has exactly one entry with 1 and all the other entries are zero. 

Clearly the map 
\begin{equation}\begin{split}\label{kk}
\vartheta: \Ker \big (A_{E_H}^t-I: \Z^{H\setminus S}\longrightarrow \Z^{H} \big)&\longrightarrow \Ker\bigg({\tiny{\begin{pmatrix}A_{E_H}^t-I& Z^t\\ 0& -I\end{pmatrix}}}: \Z^{{\overline{E}_{H}^0}\setminus S}
\longrightarrow \Z^{\overline{E}_{H}^0}\bigg),\\
(n_v)^t &\longmapsto ((n_v), \mathbf{0})^t\end{split}\end{equation} is an isomorphism with $S$ the set of sinks in $E$. 
Now we show that the map 
\begin{equation}\label{cc}
\begin{split}
\varsigma : \Coker\big (A_{E_H}^t-I: {k^{\times}}^{H\setminus S}\longrightarrow {k^{\times}}^{H}\big )&\longrightarrow \Coker\bigg({\tiny{\begin{pmatrix}A_{E_H}^t-I& Z^t\\ 0& -I\end{pmatrix}}}: {k^{\times}}^{{\overline{E}_{H}^0}\setminus S} \longrightarrow {k^{\times}}^{\overline{E}_{H}^0}\bigg),\\
(m_1, \cdots, m_n)^t &\longmapsto (m_1, \cdots, m_n, 1, \cdots, 1)^t\end{split}\end{equation} is an isomorphism. Note that $\varsigma$ is well-defined as $(m_1, \cdots, m_n, 1, \cdots, 1)^t\in 
\Imm \tiny{\begin{pmatrix}A_{E_H}^t-I& Z^t\\ 0& -I\end{pmatrix}}$ if $(m_1, \cdots, m_n)^t\in {\rm Im} (A_{E_H}^t-I)$. Actually if $(m_1, \cdots, m_n)^t\in {\rm Im}(A^t_{E_H}-I)$, we may write $(A^t_{E_H}-I)x=(m_1,\cdots,m_n)^t$
for some $x\in {k^{\times}}^{H\setminus S}$. Then we have $\begin{pmatrix}A^t_{E_H}-I& Z^t\\ 0& -I\end{pmatrix}\begin{pmatrix}x\\1\\\cdots \\1\end{pmatrix}=(m_1,\cdots, m_n, 1,\cdots,1)^t$ as the 
$\mathbb Z$-action on $k^{\times}$ is given by $n\cdot \mu=\mu^n$ and $(n_1+n_2)\cdot \mu=\mu^{n_1}\mu^{n_2}$ for $n, n_1,n_2\in\mathbb Z$ and $\mu\in k^{\times}$. To show that $\varsigma$ is injective, suppose that \begin{equation}
\begin{pmatrix}m_1\\ \vdots\\ m_n\\ 1\\ \vdots \\1\end{pmatrix}=\begin{pmatrix}A_{E_H}^t-I& Z^t\\ 0& -I\end{pmatrix}\begin{pmatrix}x_1\\ \vdots\\  x_n\\  \vdots\\ \vdots\\  x_l\end{pmatrix}.\end{equation} 
Then $\tiny\begin{pmatrix}x_{n+1}^{-1}\\ x_{n+2}^{-1}\\ \vdots\\  x_l^{-1}\end{pmatrix}=\begin{pmatrix}1\\ \vdots\\  1\end{pmatrix}$ and so 
$\tiny\begin{pmatrix}x_{n+1}\\ x_{n+2}\\ \vdots\\  x_l\end{pmatrix}=\begin{pmatrix}1\\ \vdots\\  1\end{pmatrix}$, which gives $\tiny\begin{pmatrix}m_1\\ \vdots\\  m_n\end{pmatrix}\in \Imm (A_{E_H}^t-I)$. Hence $\varsigma$ is injective. 
To show that $\varsigma$ is surjective, we observe that 
\begin{equation}
\begin{split}
\tiny\begin{pmatrix}y_1\\ \vdots\\  y_n\\  \vdots\\ \vdots\\  y_l\end{pmatrix}\Imm \tiny{\begin{pmatrix}A_{E_H}^t-I& Z^t\\ 0& -I\end{pmatrix}}
&=\tiny{\begin{pmatrix}y_1\\ \vdots\\  y_n\\  \vdots\\ \vdots\\  y_l\end{pmatrix}} \tiny{\begin{pmatrix}A_{E_H}^t-I& Z^t\\ 0 & -I\end{pmatrix}} \tiny{\begin{pmatrix}1\\ \vdots\\  1 \\ y_{n+1}\\ \vdots \\ y_{l}\end{pmatrix}} 
\Imm {\tiny\begin{pmatrix}A_{E_H}^t-I& Z^t\\ 0& -I\end{pmatrix}} \\
&= \tiny\begin{pmatrix}z_1\\ \vdots\\  z_n\\ 1 \\ \vdots\\  1 \end{pmatrix} \Imm {\tiny\begin{pmatrix}A_{E_H}^t-I& Z^t\\ 0& -I\end{pmatrix}}.
\end{split}
\end{equation}
 
By Lemma \ref{fact}, in order to prove the induced isomorphism $K_1(L(E_H))\xrightarrow{} K_1(L(\overline{E_H}))$ by \eqref{nn}, it suffices to show that the following diagram commutes
\begin{equation}
\label{diagram-original}
\xymatrix@C+.1pc@R+1.1pc{
\Coker\left(A_{E_H}^t-I: {k^{\times}}^{H\setminus S}\xra {k^{\times}}^{H}\right) \ar[r]^>>>>>>>>{\lambda'_{E_H}} \ar[d]^{\varsigma}& K_1\big(L(E_H)\big) \ar[r]^>>>>>>>>>{\xi'_{E_H}} \ar[d]^{} & \Ker\left(A_{E_H}^t-I: \Z^{H\setminus S}\xra \Z^{H}\right)\ar[d]^{\vartheta}\\
\Coker\left({\begin{pmatrix}A_E^t-I& Z^t\\ 0& -I\end{pmatrix}}: {k^{\times}}^{\overline{E}_H^0\setminus S}\xra {k^{\times}}^{\overline{E}_H^0}\right) \ar[r]^>>>>>>>{\lambda'_{\overline{E}_H}}& K_1\big(L(\overline{E}_H)\big) \ar[r]^>>>>>>{\xi'_{\overline{E}_H}}& \Ker\left({\begin{pmatrix}A_E^t-I& Z^t\\ 0& -I\end{pmatrix}}: {\Z}^{\overline{E}_H^0\setminus S}\xra{\Z}^{\overline{E}_H^0}\right).}
\end{equation}
Obviously, the left square of the above diagram commutes. The right square of the above diagram commutes, as there is a commutative diagram with natural maps
\[\xymatrix@C+.1pc@R+1.1pc{0\ar[r]&\mathcal{K}(E_H)\ar[r] \ar[d]& C(E_H) \ar[r] \ar[d]& L(E_H) \ar[r] \ar[d]&0\\
0 \ar[r]&\mathcal{K}(\overline{E}_H) \ar[r]& C(\overline{E}_H) \ar[r]& L(\overline{E}_H) \ar[r]&0.}
\] 

Now we show that the induced morphism $\overline{K}_1\big(L(E_H)\big)\xra \overline{K}_1\big(L(\overline{E}_{H})\big)$ is an isomorphism. By \eqref{diagram-original}, we have the induced commutative diagram \begin{equation}
\label{diagram-new}
\xymatrix@C+.1pc@R+1.1pc{
\overline{\Coker}\left(A_{E_H}^t-I\right) \ar[r]^>>>>>>>>{\lambda_{E_H}} \ar[d]^{\ol{\varsigma}}& \overline{K}_1\big(L(E_H)\big) \ar[r]^>>>>>>>>>{\xi_{E_H}} \ar[d]^{} & \Ker\left(A_{E_H}^t-I: \Z^{H\setminus S}\xra \Z^{H}\right)\ar[d]^{\vartheta}\\
\overline{\Coker}\left({\begin{pmatrix}A_E^t-I& Z^t\\ 0& -I\end{pmatrix}}\right) \ar[r]^>>>>>>>{\lambda_{\overline{E}_H}}& \overline{K}_1\big(L(\overline{E}_H)\big) \ar[r]^>>>>>>{\xi_{\overline{E}_H}}& \Ker\left({\begin{pmatrix}A_E^t-I& Z^t\\ 0& -I\end{pmatrix}}: {\Z}^{\overline{E}_H^0\setminus S}\xra{\Z}^{\overline{E}_H^0}\right).}
\end{equation} It remains to show that $\ol{\varsigma}:\overline{\Coker}(A^t_{E_H}-I)\xrightarrow{} \overline{\Coker} {\begin{pmatrix}A_E^t-I& Z^t\\ 0& -I\end{pmatrix}} $ is an isomorphism. As $\overline{\Coker} (A^t_{E_H}-I)$ is a quotient of $\Coker(A^t_{E_H}-I)$ (see \eqref{diagram-fortwo}), $\ol{\varsigma}:\overline{\Coker}(A^t_{E_H}-I)\xrightarrow{} \overline{\Coker} {\begin{pmatrix}A_E^t-I& Z^t\\ 0& -I\end{pmatrix}} $ is surjective. We only need to show that  $\ol{\varsigma}:\overline{\Coker}(A^t_{E_H}-I)\xrightarrow{} \overline{\Coker} {\begin{pmatrix}A_E^t-I& Z^t\\ 0& -I\end{pmatrix}} $ is injective.

 We claim that the restriction map $$\varsigma_|: \bigg(\{-1,1\}^{H}\cdot {\rm Im} (A_E^t-I)\bigg)\bigg/{\rm Im} (A_E^t-I) \longrightarrow{} \bigg(\{-1,1\}^{\overline{E}_H^0}\cdot {\rm Im}\begin{pmatrix}A_E^t-I& Z^t\\ 0& -I\end{pmatrix}\bigg)\bigg/{\rm Im} \begin{pmatrix}A_E^t-I& Z^t\\ 0& -I\end{pmatrix}, $$ sending the isomorphism class of $(m_1,\cdots, m_n)^t$ to the isomorphism class of $(m_1,\cdots, m_n,1,\cdots, 1)^t$ with $m_1,\cdots, m_n$ (either $-1$ or $1$) the corresponding coefficients of vertices in $H$, is surjective. For any $v\in \overline{E}_H^0\setminus H$, there is only one edge $e$ starting from $v$ with $r(e)\in H$. We have that
\begin{align}
\label{equation1}
	\tiny\begin{pmatrix}1\\ \vdots\\  1\\  \vdots\\ -1\\ \vdots\\  1\end{pmatrix}\Imm \tiny{\begin{pmatrix}A_{E_H}^t-I& Z^t\\ 0& -I\end{pmatrix}}&=	\tiny\begin{pmatrix}1\\ \vdots\\  1\\  \vdots\\ -1\\ \vdots\\  1\end{pmatrix}\tiny{\begin{pmatrix}A_{E_H}^t-I& Z^t\\ 0& -I\end{pmatrix}}\tiny\begin{pmatrix}1\\ \vdots\\  1\\  \vdots\\ -1\\ \vdots\\  1\end{pmatrix}\Imm \tiny{\begin{pmatrix}A_{E_H}^t-I& Z^t\\ 0& -I\end{pmatrix}}\\
	\label{equation2}
&=	\tiny\begin{pmatrix}1\\ \vdots\\ -1\\ \vdots\\  1\\  \vdots\\  1\end{pmatrix} \Imm \tiny{\begin{pmatrix}A_{E_H}^t-I& Z^t\\ 0& -I\end{pmatrix}}\in \Imm \varsigma_|.
\end{align} Here the element $(1,\cdots, 1,\cdots, -1,\cdots,1)^t$  in \eqref{equation1} has only one entry $-1$ corresponding to $v$ and all the other entries are $1$; the element $(1,\cdots,-1,\cdots, 1,\cdots,  1)^t$ in \eqref{equation2} has only one entry $-1$ corresponding to $r(e)\in H$ and all the other entries are $1$. Hence the proof for the claim is completed.   

Suppose that $\ol{\varsigma}(\ol{x})=0$ for $\ol{x}\in \overline{\Coker}(A^t_{E_H}-I)$. This implies $\varsigma(x)$ is belonging to $$\bigg(\{-1,1\}^{\overline{E}_H^0}\cdot {\rm Im}\begin{pmatrix}A_E^t-I& Z^t\\ 0& -I\end{pmatrix}\bigg)\bigg/{\rm Im} \begin{pmatrix}A_E^t-I& Z^t\\ 0& -I\end{pmatrix},$$ which is a subgroup of $\Coker \begin{pmatrix}A_E^t-I& Z^t\\ 0& -I\end{pmatrix} $. By the above claim, $\varsigma_|$ is surjective, implying that $$x\in \bigg(\{-1,1\}^{H}\cdot {\rm Im} (A_E^t-I)\bigg)\bigg/{\rm Im} (A_E^t-I).$$
Hence $\ol{x}=0$. This completes the proof.
\end{proof}

\section{Filtered $K$-theory} 

Filtered $K$-theory of a ring is an invariant consisting of $K$-theory of certain ideals and their quotients and the long exact sequence of $K$-theory relating them (Definition~\ref{filterkk}).  In its current form, it was first formulated in the setting of $C^*$-algebras by Eilers, Restorff and Ruiz, along with a working conjecture that the filtered $K$-theory is a complete invariant for all graph $C^*$-algebras. In major works \cite{errs2,errs3}, the authors managed to prove the conjecture for unital graph $C^*$-algebras by first proving that when a pair of $C^*$-algebras associated to graphs $E$ and $F$, respectively, have the same filtered $K$-theory, then $E$ may be changed into $F$ by a sequence of geometric ``moves'' applied to the graphs in a way resembling the Reidemeister moves for knots.

In \cite{errs} the algebraic version of filtered $K$-theory was formulated and it was shown that, for two graphs, if the algebraic filtered $K$-theories of their Leavitt path algebras are isomorphic then the 
filtered $K$-theories of their graph $C^*$-algebras are also isomorphic (Theorem~\ref{lateradd}). 

Here we briefly recall the notion of algebraic filtered $K$-theory following~\cite{errs} and in \S\ref{hfgffdsret2} we relate the graded $K$-theory of a Leavitt path algebras to its algebraic filtered $K$-theory.

\subsection{Prime spectrum for a ring} \label{hgfdsa} Let $R$ be a ring. We denote by $\mathbb{I}(R)$ the lattice of ideals in $R$. Let $\mathscr{S}$ be a sublattice of $\mathbb{I}(R)$ closed under arbitrary intersections and containing the trivial ideals $\{0\}$ and $R$. 
An ideal $\mathfrak{p}\in \mathscr{S}$ is called $\mathscr{S}$-prime if $\mathfrak{p}\neq R$ and for any ideals $I, J \in \mathscr{S}$, $$IJ\subseteq \mathfrak{p}\Longrightarrow I\subseteq \mathfrak{p} \text{~or~} J\subseteq \mathfrak{p}.$$ Define 
\begin{equation}\label{hgyhfyhry5}
\Spec_{\mathscr{S}} (R):=\{ \mathfrak{p} \in \mathscr{S} \mid \mathfrak{p} \text{ is } \mathscr{S}\text{-prime} \}.
\end{equation}
 For each subset $T \sub \Spec_{\mathscr{S}} (R)$, define the kernel of $T$ as $$\Ker(T)=\bigcap_{\mathfrak{p}\in T}\mathfrak{p}$$

 Throughout this section we assume that for any $I \in \mathscr{S}$, we have 
\begin{equation}\label{intersec}
I=\bigcap_{I \subseteq \mathfrak{p}}  \mathfrak{p},  \text{~where~}  \mathfrak{p} \in \Spec_{\mathscr{S}} (R).
\end{equation}

One can then equip $\Spec_{\mathscr{S}} (R)$ with the \emph{Jacobson topology}, where the open sets are 
\begin{equation}
\label{opensets}
W(I)=\{\mathfrak{p}\in\Spec^{\gr}(R)   \mid   I \nsubseteq  \mathfrak{p}\},\end{equation} where $I \in \mathscr{S}$ (\cite[Lemma 3]{errs} and \cite[Theorem 2]{errs}).

Let $X$ be a topological space and  $\mathcal{O}(X)$ denote the set of all open subsets of $X$. Let $Y$ be a subset of $X$. We call $Y$ \emph{locally closed} if $Y =U\setminus V$, where $V \subseteq U$ are open subsets of $X$. We let $\mathbb{LC}(X)$ be the set of locally closed subsets of $X$.

By \cite[Theorem 2]{errs} there is a lattice isomorphism  
\begin{align}\label{pogfythdtr65}
\phi : \mathcal{O}(\Spec_{\mathscr{S}}(R))&\longrightarrow \mathscr{S},\\
U&\longmapsto  \Ker (U^c).\notag
\end{align}
Note that since $U=W(I)$, for some ideal $I\in \mathscr{S}$, then $\phi(U)=\Ker (U^c)=\bigcap_{I\subseteq  \mathfrak{p}}  \mathfrak{p}=I$. 
We will use this correspondence to define our filtered $K$-groups in \S\ref{jfutjf2233}.

\subsection{Filtered $K$-theory}\label{jfutjf2233}

Suppose $R$ is a ring and $\mathscr{S}$ is a sublattice of ideals as in \S\ref{hgfdsa}. Moreover, assume that every ideal $I\in \mathscr{S}$ has a countable approximate unit consisting of idempotents, i.e., there exists a sequence $\{e_n\}_{n=1}^{\infty}$ in $I$ such that

\begin{itemize}
\item $e_n$ is an idempotent for all $n\in \mathbb{N}$,
\item $e_n e_{n+1}=e_n$ for all $n\in \mathbb{N}$, and
\item for all $r\in I$, there exists $n\in \mathbb{N}$ such that $r e_n=e_n r=r$.
\end{itemize}

In keeping with the notation from $C^*$-algebras, for every open set $U\in \mathcal{O}(\Spec_{\mathscr{S}} (R))$, using the lattice isomorphism $\phi$ from (\ref{pogfythdtr65})
we define
$$R[U] :=\phi (U).$$ 
Whenever we have open sets $V \subseteq U$ we can form the quotient $R[U] / R[V]$. Then \cite[Lemma 4]{errs} shows that the quotient $R[U]/ R[V]$ only depends on the set difference $U\setminus V$ up to canonical isomorphism. For $Y =U \setminus V \in \mathbb{LC}(\Spec_{\mathscr{S}} (R))$, define
$$R[Y] := R[U] / R[V].$$

Then for any locally closed subset $Y =U\setminus V$ of $\Spec_{\mathscr{S}} (R)$, we have a collection of
abelian groups $\{K_n(R[Y])\}_{n\in\Z}$, where $K_n$ are the Quillen algebraic $K$-groups. Moreover, for all $U_1, U_2, U_3\in \mathcal{O}(\Spec_{\mathscr{S}} (R))$ with $U_1\subseteq U_2\subseteq U_3$, by \cite[Lemma 3.10]{rt}, we have a long exact sequence in algebraic $K$-theory 
\begin{align}\CD
 \label{ktheroyses}
K_n(R[U_2\setminus U_1])@>{\iota_*}>> K_n(R[U_3\setminus U_1])@>{\pi_*}>>K_n(R[U_3\setminus U_2])@>{\partial_*}>>K_{n-1}(R[U_2\setminus U_1]). 
\endCD
\end{align}

We are ready to define the filtered $K$-theory of a ring. 

\begin{defi} \label{filterkk}
The \emph{filtered $K$-theory} of the ring $R$ with respect to the sublattice of ideals  $\mathscr{S}$ as above, is the collection 
\[\big \{K_n(R[Y]) \big \}_{k\leq n\leq m, Y\in \mathbb{LC}(\Spec_{\mathscr{S}} (R))},\] 
equipped with the natural transformations $\iota_*, \pi_*, \partial_*$ given in \eqref{ktheroyses}. We denote this with  ${\FK}_{k,m}(\Spec_{\mathscr{S}} (R);R)$.
\end{defi}

Let $R, R'$ be rings with sublattices of ideals $\mathscr{S}$ and  $\mathscr{S}'$, respectively.  For $k, m\in \Z$ with $k\leq 0\leq m$, an isomorphism  
\begin{equation}\label{poploko}
{\FK}_{k,m}(\Spec_{\mathscr{S}} (R);R) \longrightarrow {\FK}_{k,m}(\Spec_{\mathscr{S}'} (R'); R')
\end{equation}
consists of a homeomorphism \[\phi: \Spec_{\mathscr{S}} (R)\xra \Spec_{\mathscr{S}'} (R')\] and an isomorphism \[\alpha_{Y, n}: K_n(R[Y]) \rightarrow K_n(R'[\phi(Y)]),\] for each $n$ with $k \leq n \leq m$ and for each $Y\in \mathbb{LC}(\Spec_{\mathscr{S}} (R))$ such that the diagrams involving the natural transformations commute. If the isomorphism (\ref{poploko}) restricts to an order isomorphism on $K_0(R[Y])$ for all
$Y\in \mathbb{LC}(\Spec_{\mathscr{S}} (R))$, we write
$${\FK}^{+}_{k,m}(\Spec_{\mathscr{S}} (R);R)\cong{\FK}^{+}_{k,m}(\Spec_{\mathscr{S}'} (R'); R').$$

Next we specialise to the case of Leavitt path algebras and graph $C^*$-algebras. For a countable graph $E$, we consider the set $\mathscr{S}$ to be $\LL^{\gr}\big(\L(E)\big)$; the lattice of graded (two-sided) ideals of $L(E)$.  Since $\Spec_{\mathscr{S}} (L(E))$ by definition (\ref{hgyhfyhry5}) becomes the set of graded prime ideals of $L(E)$, we denote it by $\Spec^{\gr}(L(E))$. 
By \cite[Lemma 6]{rt}, every graded-ideal of $L(E)$ has a countable approximate unit consisting of idempotents. Thus we can carry out the construction of the filtered $K$-theory in this setting. 






\section{Graded $K$-theory gives a certain filtered $K$-theory}\label{hfgffdsret2}

In this section we prove that the graded Grothendieck group of the Leavitt path algebra of a row-finite graph determines a certain portion of its filtered $K$-theory. Recall from \S\ref{lpabasis} that $\LL^{\gr}\big(\L(E)\big)$ denotes the lattice of graded two-sided ideals of $L(E)$ and $\Spec^{\gr}(\L(E))$ the set of graded prime ideals of $\L(E)$. 

\begin{lem}\label{topology} Let $E, F$ be row-finite graphs. Suppose that there exists an order-preserving $\mathbb Z[x,x^{-1}]$-module isomorphism 
$$\varphi: K_0^{\gr}(\L(E))\longrightarrow K_0^{\gr}(\L(F)).$$ Then 
there exists a lattice isomorphism 
\[\varphi \colon \LL^{\gr} (L(E)) \longrightarrow \LL^{\gr} (L(F)),\] which restricts to 
a homeomorphism of topological spaces 
\[\varphi: \Spec^{\gr}(\L(E))\xra \Spec^{\gr}(\L(F)),\]
such that $\VV^{\gr}(\varphi (I)) = \varphi (\VV^{\gr}(I))$ for all $I\in \LL^{\gr}\big(\L(E)\big)$. 
\end{lem}

\begin{proof} 
%
The order-preserving isomorphism $\varphi: K_0^{\gr}(L(E))\xra K_0^{\gr}(L(F))$ sends $\VV^{\gr}(L(E))$ to $\VV^{\gr}(L(F))$. In particular $\varphi $ induces a lattice isomorphism, also
denoted by $\varphi$, from $\LL (\VV^{\gr}(L(E)))$ onto $\LL (\VV^{\gr}(L(F)))$. By \eqref{latticeisosec} and \eqref{idealmon}, we have that the lattice isomorphism given in \eqref{latticncnc} sending $I$ to $\VV^{\gr}(I)$, for 
any graded ideal $I$ of $L(E)$. Therefore it follows from \eqref{latticncnc} that there exists a lattice isomorphism 
$$\varphi \colon \LL^{\gr} (L(E)) \longrightarrow \LL^{\gr} (L(F)),$$
such that $\VV^{\gr} (\varphi (I)) = \varphi (\VV^{\gr} (I))$, for any $I\in \LL^{\gr} (L(E))$. The restriction of this isomorphism to the set of prime elements gives
the desired bijection $\varphi: \Spec^{\gr}(\L(E))\xra \Spec^{\gr}(\L(F))$.

To conclude, it suffices to show that $\varphi: \Spec^{\gr}(\L(E)) \xra \Spec^{\gr}(\L(F))$ is open. By \eqref{opensets}, open sets of $\Spec^{\gr}(\L(E))$ are precisely the set of the form $W(I)=\{\mathfrak{p}\in\Spec^{\gr}(\L(E))\;|\;   I \nsubseteq  \mathfrak{p} \}$ for some proper graded ideal $I$ of $\L(E)$. Observe that for a proper graded ideal $I$ of $\L(E)$
\begin{equation*}
\begin{split}
\varphi(W(I))
&=\varphi(\{ \mathfrak{p}\in\Spec^{\gr}(\L(E))\;|\;    I \nsubseteq  \mathfrak{p} \})\\
&=\{\varphi(\mathfrak{p})\;|\; \mathfrak{p}\in\Spec^{\gr}(\L(E)),   I \nsubseteq  \mathfrak{p}\}\\
&=\{\varphi(\mathfrak{p})\in\Spec^{\gr}(\L(F))\;|\;      \varphi(I)  \nsubseteq  \varphi(\mathfrak{p})  \}\\
&=W(\varphi(I))
\end{split}
\end{equation*} and $W(\varphi(I))$ is open. Thus we prove that $\varphi: \Spec^{\gr}(\L(E)) \xra \Spec^{\gr}(\L(F))$ is an open map. Similar argument shows that the inverse of $\varphi$ is open. 
\end{proof}

From now on, let $E, F$ be row-finite graphs. Suppose that there is an order-preserving $\Z[x, x^{-1}]$-isomorphism
\begin{equation}
\label{suppose}\varphi:K_0^{\gr}(\L(E)) \longrightarrow K_0^{\gr}(\L(F)).
\end{equation} 

We record this observation on the level of $K$-groups in the following lemma. 


\begin{lem}\label{hfbcvgfbfg567} Let $E, F$ be row-finite graphs. Suppose that there exists an order-preserving $\Z[x, x^{-1}]$-isomorphism
\[\varphi: K_0^{\gr}(\L(E)) \longrightarrow K_0^{\gr}(\L(F)).\] Then for any graded ideals $I\sub J \subseteq L_k(E)$, there are 
order-preserving $\Z[x, x^{-1}]$-isomorphisms $K_0^{\gr}(I)\cong K_0^{\gr}(\varphi(I))$ and $K_0^{\gr}(J/I)\cong K_0^{\gr}(\varphi(J)/\varphi(I))$ which are induced by $\varphi$. 
\end{lem}

\begin{proof} For a graded ideal $I$ of $L(E)$, we denote by $H_I$ the corresponding hereditary saturated subset of $E^0$. We define $\varphi_{|}: K_0^{\gr}(I)\xrightarrow{}K_0^{\gr}(\varphi(I))$ as the composition of the following isomorphisms 
\[K_0^{\gr}(I)\cong K_0^{\gr}(L(E_{H_I}))\xrightarrow{}K_0^{\gr}(L(F_{H_{\varphi(I)}}))\cong K_0^{\gr}(\varphi(I)),\] where the isomorphism $K_0^{\gr}(L(E_{H_I}))\xrightarrow{} K_0^{\gr}(L(F_{H_{\varphi(I)}}))$ is given by $v(i)\mapsto \varphi(v(i))$, for $v\in H_I$ and $i\in\Z$. The isomorphism $K_0^{\gr}(I)\cong K_0^{\gr}(L(E_{H_I}))$ is given by the graded Morita context (refer to \cite[Theorem 5.7(3)]{tomforde}) and similarly we have the isomorphism $K_0^{\gr}(\varphi(I))\cong K_0^{\gr}(L(F_{H_{\varphi(I)}}))$.
 Obviously $\varphi_{|}$ is an order-preserving $\Z[x, x^{-1}]$-isomorphism.

Recall from \eqref{shortexactfor} that we have the short exact sequence 
\begin{align*}\CD
 0@>>> K_0^{\gr}(L(E_{H_I}))@>{}>>K_0^{\gr}(L(E))@>{}>>K_0^{\gr}(L(E/H_I))@>>>0,
\endCD
\end{align*}  and thus the following commutative diagram 
\begin{align}
\xymatrix{
0\ar[r] &K_0^{\gr}(L(E_{H_I})) \ar[r]^{} \ar[d]^{\varphi_{|}}& K_0^{\gr}(L(E))\ar[r]^{}\ar[d]^{\varphi} & K_0^{\gr}(L(E/H_I))\ar[d]\ar[r]&0\\
0\ar[r] &K_0^{\gr}(L(F_{H_{\varphi(I)}})) \ar[r]^{} & K_0^{\gr}(L(F))\ar[r]&K_0^{\gr}(L(F/H_{\varphi(I)}))\ar[r]&0,}
\end{align} where the map $K_0^{\gr}(L(E/H_I))\xrightarrow{} K_0^{\gr}(L(F/H_{\varphi(I)}))$ is induced by $\varphi: K_0^{\gr}(L(E))\xrightarrow {}K_0^{\gr}(L(F))$. 

 For two graded ideals $I$ and $J$ of $L(E)$ with $I\subseteq J$, we have that $J/I$ is a graded ideal of ${L(E)/I} \cong L(E/{H_I})$. Similarly we have the following order-preserving $\Z[x, x^{-1}]$-isomorphism
\[K_0^{\gr}(J/I)\cong K_0^{\gr}(L(E/H_I)_{H_{J/I}})\longrightarrow{} K_0^{\gr}(L(F_{H_{\varphi(I)}})_{H_{\varphi(J)/\varphi(I)}})\cong K_0^{\gr}(\varphi(J)/\varphi(I)).\qedhere\]  \end{proof}


Recall from \eqref{definition-quotient} that for a row-finite graph $E$ we define $\overline{K}_1(L(E))$ as the quotient $K_1(L_k(E))/G_E$ of $K_1(L_k(E))$. In the following, we define a certain quotient of $K_1(I)$ for a graded ideal $I$ of $L(E)$. 

Let $H$ be a hereditary saturated subset of $E^0$ and $I_H$ the corresponding graded ideal of $L(E)$. Recall from Lemma \ref{kone} that the homomorphism $L(E_H)\xrightarrow{}L(\overline{E}_H)$ \eqref{nn} of algebras induces isomorphisms $K_1(L(E_H))\xrightarrow{} K_1(L(\overline{E}_H))$. As we have the composition map \begin{equation}
    \label{zzz0}
    L(E_{H})\longrightarrow{} L(\overline{E}_{H})\longrightarrow{} I_H.\end{equation} It follows that it induces isomorphisms $$K_1(L(E_{H}))\longrightarrow{} K_1(L(\overline{E}_H))\longrightarrow{} K_1(I_H)$$ of K-groups. We denote the above composition isomorphism by $f$. By \eqref{definition-quotient}, there exists a subgroup $G_H$ such that $\overline{K_1}(L(E_H))=K_1(L(E_H))/G_H$. We define $\overline{K}_1(I_H)=K_1(I_H)/f(G_H)$ such that the following diagram commutes
\begin{equation}
\label{definition-idealkone}
\xymatrix@C+.1pc@R+1.1pc{
K_1(L(E_H))\ar[r]^{f}\ar[d]^{}&K_1(I_H)\ar[d]\\
\overline{K}_1(L(E_H))\ar[r]&\overline{K}_1(I_H),}  
\end{equation} where the induced morphism $\overline{K}_1(L(E_H))\xrightarrow{} \overline{K}_1(I_H)$ by $f$ is also an isomorphism.


\begin{lem}
\label{lemma-pre}
Let $E$ be a row-finite graph, $I\subseteq J$ two graded ideals of $L(E)$, and $H_I$ and $H_J$ the corresponding hereditary saturated subsets of $E^0$. Then we have the commutative diagram with exact rows: 
\begin{equation}
\label{equation-pre}
\xymatrix@C+.1pc@R+1.1pc{
\overline{K}_1(I)\ar[r]^{}
\ar[r]^{}& \overline{K}_1(J)\ar[r]^{} 
& \overline{K}_1(J/I)\ar[r]^{} 
& K_0(I)
\\
\overline{K}_1\big(L(E_{H_I})\big)\ar[u]^{{\scriptscriptstyle{\cong\ \ }}}\ar[r]^{\alpha} & \overline{K}_1\big(L(E_{H_J})\big)\ar[u]^{{\scriptscriptstyle{\cong\ \ }}}\ar[r]^>>>>{\alpha'}& \overline{K}_1\big(L(({E/H_I})_{H_{J/I}})\big)\ar[u]^{{\scriptscriptstyle{\cong\ \ }}}\ar[r] & K_0(L(E_{H_I}))\ar[u]^{{\scriptscriptstyle{\cong\ \ }}}}
\end{equation}	
\end{lem}
\begin{proof} We have the commutative diagram \eqref{hide} whose second row is given by canonical homomorphisms of algebras.

\begin{equation}
\label{hide}
\xymatrix@C-1pc@R-0.3pc{I\ar[r]
&J\ar[r]&J/I\\
L(E_{H_I})\ar[r]\ar[u]&L(E_{H_J})\ar[r]\ar[u]&L(({E/H_I)}_{H_{J/I}}).\ar[u]^{}
}
\end{equation} We emphasize that here $L(E_{H_I})\xrightarrow{} I$ in \eqref{hide} is given by  \eqref{zzz0}. By the long exact sequence of algebraic K-theory for the short exact sequence $0\xrightarrow{} I\xrightarrow{}J \xrightarrow{} J/I\xrightarrow{}0$, we have the following commutative diagram. We emphasize that the second row for K-groups of Leavitt path algebras in \eqref{koneforhide} is induced by the first row.
 \begin{equation}
\label{koneforhide}
\xymatrix@C-1pc@R-0.3pc{K_1(I)\ar[r]
&K_1(J)\ar[r]&K_1(J/I)\ar[r]&K_0(I)\\
K_1(L(E_{H_I}))\ar[r]\ar[u]^{{\scriptscriptstyle{\cong\ \ }}}&K_1(L(E_{H_J}))\ar[r]\ar[u]^{{\scriptscriptstyle{\cong\ \ }}}&K_1\big(L(({E/H_I)}_{H_{J/I}})\big)\ar[u]^{{\scriptscriptstyle{\cong\ \ }}}\ar[r]&K_0(L(E_{H_I}))\ar[u]^{{\scriptscriptstyle{\cong\ \ }}}
}
\end{equation} We observe that \eqref{equation-pre} follows as an induced commutative diagram.
\end{proof}

We can now define the corresponding version of $\FK_{0,1}(L_k(E))$.

\begin{defi}
	\label{def:FKbar} Let $E$ be a row-finite graph and $k$ be a field. Define 
	the {\it algebraic filtered $\overline{K}$-theory} $\FKbar _{0,1}(L_k(E))$ as
	the collection 
	$$\{\ol{K}_n (J/I)\}_{0\le n\le 1}$$ 
	where $(I,J)$ ranges over all the graded ideals of $L_k(E)$ such that $I\subseteq J$, $\ol{K}_0(J/I)= K_0(J/I)$ and 
	$$\ol{K}_1(J/I)=K_1(J/I)/f(G_{(E/H_I)_{H_{J/I}}})$$ as given in \eqref{definition-idealkone}. For $I\subseteq J \subseteq P$, there are 
	natural maps
	$$\xymatrix@C+.1pc@R+1.1pc
		{\overline{K}_1(J/I)\ar[r]^{}\ar[r]^{}& \overline{K}_1(P/I)\ar[r]^{}  & \overline{K}_1(P/J)\ar[r]^{}   & K_0(J/I)\ar[r]^{} & K_0(P/I)\ar[r]^{} & K_0(P/J)}$$
		as given in Lemma \ref{lemma-pre}. Note that $\FKbar _{0,1}(L_k(E))$ is a quotient of $\FK _{0,1}(L_k(E))$. 
	\end{defi}

We are in a position to state the main theorem of this paper. 

\begin{thm}\label{maintheorem}
Let $E, F$ be row-finite graphs and $k$ a field. Suppose that there exists an order-preserving $\Z[x, x^{-1}]$-isomorphism
\[\varphi: K_0^{\gr}(\L_k(E))  \longrightarrow  K_0^{\gr}(\L_k(F)).\] Then
we have an isomorphism $\FKbar _{0,1}^+(L_k(E))\cong \FKbar_{0,1}^+(L_k(F))$, that is,  there is a homeomorphism of topological spaces 
\[\varphi: {\Spec}^{\gr}({\L}_k(E))\longrightarrow {\Spec}^{\gr}({\L}_k(F)),\] and induced order-preserving $\Z[x, x^{-1}]$-isomorphisms $\varphi_|: K_0^{\gr}(I)\xra K_0^{\gr}(\varphi(I))$ and $K_0^{\gr}(J/I)\cong K_0^{\gr}(\varphi(J)/\varphi(I))$, 
for any $I,J\in \LL^{\gr}\big(\L(E)\big)$ with $I\sub J$. 
Moreover, for any $I,J, P\in \LL^{\gr}\big(\L(E)\big)$ with $I\sub J\sub P$, there are isomorphisms for $n=0,1$
\begin{align*}
\alpha_{J/I, n}: \overline{K}_n(J/I)  \longrightarrow \overline{K}_n(\widetilde{J}/\widetilde{I}),\quad
\alpha_{P/I, n}: \overline{K}_n(P/I)  \longrightarrow \overline{K}_n(\widetilde{P}/\widetilde{I}),\quad
\alpha_{P/J, n}: \overline{K}_n(P/J) \longrightarrow \overline{K}_n(\widetilde{P}/\widetilde{J}),
\end{align*} 
where $\widetilde{I}=
\varphi(I), \widetilde{J}=\varphi(J)$ and $\widetilde{P}=\varphi(P)$ such that the following diagram commutes
\begin{equation}
\label{last}
\xymatrix@C+.1pc@R+1.1pc{
\overline{K}_1(J/I)\ar[r]^{}\ar[d]^{\alpha_{J/I, 1} \hskip .2in{\txt{\tt } \hskip -.4in}} \ar[r]^{}& \overline{K}_1(P/I)\ar[r]^{} \ar[d]^{\alpha_{P/I, 1} \hskip .2in{\txt{ \tt } \hskip -.4in}}& \overline{K}_1(P/J)\ar[r]^{} \ar[d]^{\alpha_{P/J, 1} \hskip .2in{\txt{ \tt } \hskip -.4in}} & K_0(J/I)\ar[r]^{}\ar[d]^{\alpha_{J/I, 0} \hskip .2in{\txt{ \tt } \hskip -.4in}}&K_0(P/I)\ar[r]^{}\ar[d]^{\alpha_{P/I, 0} \hskip .2in{\txt{ \tt } \hskip -.4in}}&K_0(P/J)\ar[d]^{\alpha_{P/J, 0}}\\
\overline{K}_1(\widetilde{J}/\widetilde{I})\ar[r]^{}& \overline{K}_1(\widetilde{P}/\widetilde{I})\ar[r]^{} & \overline{K}_1(\widetilde{P}/\varphi(J))\ar[r]^{} & K_0(\widetilde{J}/\widetilde{I})\ar[r]^{}&K_0(\widetilde{P}/\widetilde{I})\ar[r]^{}&K_0(\widetilde{P}/\widetilde{J}).
}
\end{equation} Here the exact rows come from the long exact sequence in algebraic $K$-theory (see \cite[Theorem 2.4.1]{cortinas}). 
 \end{thm}

\begin{proof} We denote by $H_I$ the hereditary saturated subset of $E^0$ such that $I=\langle H_I\rangle$. Similarly $H_{\widetilde{I}}$ is the hereditary saturated subset of $F^0$ such that $\widetilde{I}=\langle H_{\widetilde{I}}\rangle$.  

We first prove that for any $I,J\in \LL^{\gr}\big(\L(E)\big)$ with $I\sub J$, there are isomorphisms
\begin{align*}
\alpha_{I, 0}: K_0(I) \longrightarrow K_0(\widetilde{I}),\quad
\alpha_{J, 0}: K_0(J) \longrightarrow K_0(\widetilde{J}),\quad
\alpha_{J/I, 0}: K_0(J/I) \longrightarrow K_0(\widetilde{J}/\widetilde{I}),\\
\alpha_{I, 1}: \overline{K}_1(I) \longrightarrow \overline{K}_1(\widetilde{I}),\quad
\alpha_{J, 1}: \overline{K}_1(J)  \longrightarrow \overline{K}_1(\widetilde{J}),\quad
\alpha_{J/I, 1}: \overline{K}_1(J/I) \longrightarrow \overline{K}_1(\widetilde{J}/\widetilde{I})
\end{align*} such that the following diagram commutes
\begin{equation}
\label{ktheoryp}
\xymatrix@C+.1pc@R+1.1pc{
\overline{K}_1(I)\ar[r]^{}\ar[d]^{\alpha_{I, 1} \hskip .2in{\txt{\tt $A_1$} \hskip -.4in}} \ar[r]^{}& \overline{K}_1(J)\ar[r]^{} \ar[d]^{\alpha_{J, 1} \hskip .2in{\txt{ \tt $A_2$} \hskip -.4in}}& \overline{K}_1(J/I)\ar[r]^{} \ar[d]^{\alpha_{J/I, 1} \hskip .2in{\txt{ \tt $C$} \hskip -.4in}} & K_0(I)\ar[r]^{}\ar[d]^{\alpha_{I, 0} \hskip .2in{\txt{ \tt $B_1$} \hskip -.4in}}&K_0(J)\ar[r]^{}\ar[d]^{\alpha_{J, 0} \hskip .2in{\txt{ \tt $B_2$} \hskip -.4in}}&K_0(J/I)\ar[d]^{\alpha_{J/I, 0}}\\
\overline{K}_1(\widetilde{I})\ar[r]^{}& \overline{K}_1(\widetilde{J})\ar[r]^{} & \overline{K}_1(\widetilde{J}/\widetilde{I})\ar[r]^{} & K_0(\widetilde{I})\ar[r]^{}&K_0(\widetilde{J})\ar[r]^{}&K_0(\widetilde{J}/\widetilde{I}).
}
\end{equation}

\medskip 

{\bf Step I:}  We show the existence of $\alpha_{I,0}, \alpha_{J, 0}$ and $\alpha_{J/I, 0}$ and  check that the squares $B_1$ and $B_2$ in \eqref{ktheoryp} are commutative.

Recall that we have the exact sequence given by \eqref{sesgroup}. Then we have the following commutative diagram

\[\xymatrix@C+.1pc@R+1.1pc{
K_0^{\gr}(I) \ar[d]^{{\scriptscriptstyle{\cong\ \ }}}\ar[rr]^{\phi_I}&&K_0^{\gr}(I)\ar[d]^{{\scriptscriptstyle{\cong\ \ }}}\ar[rr]&&K_0(I)\ar[r]\ar[d]^{{\scriptscriptstyle{\cong\ \ }}}&0\\
K_0^{\gr}(L(E_{H_I}))\ar[rr]^{\phi}\ar[d]^{{\scriptscriptstyle{\cong\ \ }}}&&K_0^{\gr}(L(E_{H_I}))\ar[d]^{{\scriptscriptstyle{\cong\ \ }}}\ar[rr]^{}&&K_0(L(E_{H_I}))\ar[d]^{{\scriptscriptstyle{\cong\ \ }}}&\\
K_0^{\gr}(L(F_{H_{\widetilde{I}}}))\ar[rr]^{\phi}&&K_0^{\gr}(L(F_{H_{\widetilde{I}}}))\ar[rr]^{}&&K_0(L(F_{H_{\widetilde{I}}}))&\\
K_0^{\gr}(\widetilde{I}) \ar[u]_{{\scriptscriptstyle{\cong\ \ }}}\ar[rr]^{{\phi_{\widetilde{I}}}} &&K_0^{\gr}(\widetilde{I}) \ar[u]_{{\scriptscriptstyle{\cong\ \ }}}\ar[rr]&&K_0(\widetilde{I}) \ar[u]_{{\scriptscriptstyle{\cong\ \ }}}\ar[r]&0,
}\] where the composition of the first column is the isomorphism $\varphi_{|}:K_0^{\gr}(I)\xrightarrow[]{}K_0^{\gr}(\widetilde{I})$ in Lemma \ref{hfbcvgfbfg567}.   
The map $\phi_I: K_0^{\gr}(I)\xrightarrow{} K_0^{\gr}(I)$ is defined as the composition $K_0^{\gr}(I)\xrightarrow{}K_0^{\gr}(L(E_{H_I}))\xrightarrow{\phi} K_0^{\gr}(L(E_{H_I}))\cong K_0^{\gr}(I)$, where $\phi$ is the homomorphism given by \eqref{mendhfuti6} for Leavitt path algebras. 

For two graded ideals $I$ and $J$ of $L(E)$ with $I\subseteq J$, we have that $J/I$ is a graded ideal of ${L(E)/I} \cong L(E/{H_I})$. We have the isomorphism $K_0^{\gr}(J/I)\cong K_0^{\gr}(\widetilde{J}/\widetilde{I})$ given in Lemma \ref{hfbcvgfbfg567}. We also have the map $\phi_{J/I}: K_0^{\gr}(J/I)\xrightarrow{} K_0^{\gr}(J/I)$.

For two graded ideals $I$ and $J$ of $L(E)$ with $I\subseteq J$, we have the short exact sequence 
\[\CD
 0@>>> K_0^{\gr}(I)@>{\overline{\iota}}>>K_0^{\gr}(J)@>{\overline{\pi}}>>K_0^{\gr}(J/I)@>>>0
\endCD\] given in \eqref{2020}. Similarly, we have the short exact sequence 

\[\CD
 0@>>> K_0^{\gr}(\widetilde{I})@>{\widetilde{\iota}}>>K_0^{\gr}(\widetilde{J})@>{\widetilde{\pi}}>>K_0^{\gr}(\widetilde{J}/\widetilde{I})@>>>0.
\endCD\] We now prove that the following diagram commutes
\begin{align}
\label{2020hou}
\xymatrix{
0\ar[r] &K_0^{\gr}(I) \ar[r]^{\overline{\iota}} \ar[d]^{\varphi_|}& K_0^{\gr}(J)\ar[r]^{\overline{\pi}}\ar[d]^{\varphi_|} & K_0^{\gr}(J/I)\ar[d]\ar[r]&0\\
0\ar[r] &K_0^{\gr}(\widetilde{I}) \ar[r]^{\widetilde{\iota}} & K_0^{\gr}(\widetilde{J})\ar[r]^{\widetilde{\pi}}&K_0^{\gr}(\widetilde{J}/\widetilde{I})\ar[r]&0.}
\end{align} We only show that $\varphi_|\circ \overline{\iota}=\widetilde{\iota}\circ \varphi_|$ and similarly we have that the right square commutes. To show that $\varphi_|\circ \overline{\iota}=\widetilde{\iota}\circ \varphi_|$, observe that we have the commutative diagram 

\begin{align}
\label{2020houlai}
\xymatrix{
K_0^{\gr}(I) \ar[r]^{} \ar[d]^{\overline{\iota}}& K_0^{\gr}(L(E_{H_I}))\ar[r]^{}\ar[d]^{} &K_0^{\gr}(L(F_{H_{\widetilde{I}}}))\ar[d]&K_0^{\gr}(\widetilde{I})\ar[l]\ar[d]^{\overline{\iota}}\\
K_0^{\gr}(J) \ar[r]^{} & K_0^{\gr}(L(E_{H_J}))\ar[r]^{} &K_0^{\gr}(L(F_{H_{\widetilde{J}}}))&K_0^{\gr}(\widetilde{J})\ar[l].}
\end{align} with the middle two vertical maps induced by the natural homomorphisms of algebras. By \eqref{2020} the first left and the first right squares of \eqref{2020houlai} commute. Obviously the middle square of \eqref{2020houlai} commute.

So we have the following  commutative diagram
\begin{equation}\label{3d}
\begin{split}
\xymatrix@C-1.5pc@R+0.01pc{
&    & &&   0 \ar[dd] &&0\ar[dd]&&&& \\
& &  &  0 \ar[dd]& &0\ar[dd]&&&&&\\
&    & \Ker(\phi_I)\ar[dl]_{\scriptscriptstyle{\cong\ \ }}\ar@{->}'[d][dd] \ar@{->}'[r][rr] & & K_0^{\gr}(I)\ar[dl]_{\scriptscriptstyle{\cong\ \ }}
\ar@{->}'[r][rr]^>>>>>>>>>{{\phi_I}} \ar@{->}'[d][dd] && 
K_0^{\gr}(I)\ar[dl]_{\scriptscriptstyle{\cong\ \ }}
\ar[rr] \ar@{->}'[d][dd]&& K_0(I)\ar@[blue][dl]_{\alpha_{I, 0}}\ar@[blue]@{->}'[d][dd] \ar[rrr]& & & 0  \\
&\Ker(\phi_{\widetilde{I}}) \ar[rr]\ar[dd] && 
K_0^{\gr}(\widetilde{I})
\ar[rr]^>>>>>>>>>{\phi_{\widetilde{I}}} \ar[dd]  && 
K_0^{\gr}(\widetilde{I})
\ar[rr]\ar@{->}[dd] & & K_0(\widetilde{I})\ar@[blue]@{->}[dd]\ar@{->}[rr]& &0
 \\
&   & \Ker(\phi_J)\ar@{->}'[d][dd]
\ar@{->}'[r][rr] \ar[dl]_{\scriptscriptstyle{\cong\ \ }}& & K^{\gr}_0(J)
\ar@{->}'[r][rr]^>>>>>>>>>{\phi_J}  \ar[dl]_{\scriptscriptstyle{\cong\ \ }}\ar@{->}'[d][dd]&& K^{\gr}_0(J)\ar[dl]_{\scriptscriptstyle{\cong\ \ }}\ar@{->}'[r][rr]\ar@{->}'[d][dd]&&
K_0(J) \ar@[blue][dl]_{\alpha_{J, 0}}\ar[rrr]\ar@[blue]@{->}'[d][dd] &&  &0 \\
   &  \Ker(\phi_{\widetilde{J}}) \ar[dd] 
\ar[rr] &&  K^{\gr}_0(\widetilde{J})
\ar[rr]^>>>>>>>>>{\phi_{\widetilde{J}}}  
\ar@{->}[dd]&& K^{\gr}_0(\widetilde{J})\ar[rr]\ar@{->}[dd]&&
K_0(\widetilde{J}) \ar@[blue]@{->}[dd] \ar@{->}[rr]& &0\\
&& \Ker(\phi_{J/I})\ar@{->}'[r][rr]\ar[dl]_{\scriptscriptstyle{\cong\ \ }} & & K_0^{\gr}(J/I)\ar[dl]_{\scriptscriptstyle{\cong\ \ }}\ar@{->}'[d][dd]\ar@{->}'[r][rr]^>>>>>>>{\phi_{J/I}} && K^{\gr}_0(J/I)\ar[dl]_{\scriptscriptstyle{\cong\ \ }}\ar@{->}'[r][rr] \ar@{->}'[d][dd]&& K_0(J/I)\ar[rrr]\ar@[blue][dl]_{\alpha_{J/I, 0}}&&&0
\\
 &\Ker(\phi_{\widetilde{J}/\widetilde{I}})\ar[rr]&& K^{\gr}_0(\widetilde{J}/\widetilde{I})\ar[dd]\ar[rr]^>>>>>>>{\phi_{\widetilde{J}/\widetilde{I}}} &&K^{\gr}_0(\widetilde{J}/\widetilde{I})\ar[rr]
\ar[dd] &&K_0(\widetilde{J}/\widetilde{I})\ar[rr]& &0&&
\\
&  & & &  0 &&0&&&&\\
 & & &  0 &&0&&&&&}
 \end{split}
\end{equation}   Therefore $\alpha_{I,0}, \alpha_{J,0}$ and $\alpha_{J/I, 0}$ exist and $B_1$ and $B_2$ are commutative.

\medskip

{\bf Step II:} We define the maps $\alpha_{I, 1}:\overline{K}_1(I)\xra \overline{K}_1(\widetilde{I})$, $\alpha_{J, 1}: \overline{K}_1(J)\xra \overline{K}_1(\widetilde{J})$ and $\alpha_{J/I, 1}: \overline{K}_1(J/I)\xra \overline{K}_1(\widetilde{J}/\widetilde{I})$ and check that the squares $A_1$ and $A_2$ in \eqref{ktheoryp} are commutative (recall that $\widetilde I:= \varphi(I)$ and $\widetilde J:= \varphi(J)$). 

Denote by $A_I, A_J$  and $A_{J/I}$ the transposes of the adjacency matrices for $E_{H_I}, E_{H_J}$ and $(E/H_I)_{H_{J/I}}$ respectively. 
Denote by $A_{\widetilde{I}}, A_{\widetilde{J}}$ and $A_{\widetilde{J}/\widetilde{I}}$ the transposes of the adjacency matrices for $F_{H_{\widetilde{I}}}, F_{H_{\widetilde{J}}}$ and $(F/H_{\widetilde{I}})_{H_{\widetilde{J}/\widetilde{I}}}$
respectively. 

We have the following diagram and the maps $\alpha_{I, 1}$, $\alpha_{J,1}$ and $\alpha_{J/I, 1}$ are compositions of the vertical maps in the first three columns of \eqref{kkkk}, respectively.  In order to show that the squares $A_1$ and $A_2$ in \eqref{ktheoryp} are commutative, it suffices to show that the squares $X_1, X_2, Y_1, Y_2, Z_1$ and $Z_2$ in \eqref{kkkk} are commutative.

%
%
%

\begin{equation}
\label{kkkk}
\xymatrix@C+.1pc@R+1.1pc{
\overline{K}_1(I)\ar[r]^{}
\ar[r]^{}& \overline{K}_1(J)\ar[r]^{} 
& \overline{K}_1(J/I)\ar[r]^{} 
& K_0(I)
\\
\overline{K}_1\big(L(E_{H_I})\big)\ar[u]^{{\scriptscriptstyle{\cong\ \ }}\hskip .2in{\txt{\tt $X_1$} \hskip -.4in}}\ar[r]^{\alpha}\ar[d]_{{\scriptscriptstyle{\cong\ \ }} \hskip .2in{\txt{\tt $Y_1$} \hskip -.4in}} & \overline{K}_1\big(L(E_{H_J})\big)\ar[u]^{{\scriptscriptstyle{\cong\ \ }}\hskip .2in{\txt{\tt $X_2$} \hskip -.4in}}\ar[r]^>>>>{\alpha'}\ar[d]_{{\scriptscriptstyle{\cong\ \ }} \hskip .2in{\txt{\tt $Y_2$} \hskip -.4in}} & \overline{K}_1\big(L(({E/H_I})_{H_{J/I}})\big)\ar[d]_{{\scriptscriptstyle{\cong\ \ }}\hskip .2in{\txt{\tt $Y_3$} \hskip -.4in}}\ar[u]^{{\scriptscriptstyle{\cong\ \ }}\hskip .2in{\txt{\tt $X_3$} \hskip -.4in}}\ar[r]^{\partial} & K_0(L(E_{H_I}))\ar[u]\ar[d]_{\scriptscriptstyle{\cong\ \ }\pi}\\ 
\tiny{{\begin{matrix}\overline{\Coker} \left(A_{I}-I\right)\\
\oplus \Ker \left(A_{I}-I\right)\end{matrix}}}\ar[r]^{\tiny{\begin{pmatrix}\overline{\sigma}&0\\ 0& \tau\end{pmatrix}}}\ar[d]^{\Theta_I \hskip .2in{\txt{\tt $Z_1$} \hskip -.4in}} & \tiny{{\begin{matrix}\overline{\Coker} \left(A_{J}-I\right)\\
\oplus \Ker \left(A_{J}-I\right)\end{matrix}} } \ar[r]^{\tiny{\begin{pmatrix}\overline{\sigma'}&0\\ 0& \tau' \end{pmatrix}}}\ar[d]^{\Theta_J \hskip .2in{\txt{\tt $Z_2$} \hskip -.4in}}&
\tiny{{\begin{matrix}\overline{\Coker} \left({A_{{J/I}}-I}  \right)\\
\oplus \Ker \left({A_{{J/I}}-I}\right)\end{matrix}} } \ar[r]^>>>>>{\tiny{\tiny{\begin{pmatrix}0&0\\ 0&\delta_{E_{H_I}, E_{H_J}} \end{pmatrix}}}} \ar[d]^{\Theta_{J/I} \hskip .2in{\txt{\tt $Z_3$} \hskip -.4in}} & \tiny{{\begin{matrix}\Coker \left({A_{I}-I}\right)\end{matrix}}}\ar[d]^{\Theta} \\
\tiny{{\begin{matrix}\overline{\Coker} \left({A_{\widetilde{I}}-I} \right)\\
\oplus \Ker \left({A_{\widetilde{I}}-I} \right)\end{matrix}}}\ar[r] & \tiny{{\begin{matrix}\overline{\Coker} \left({A_{\widetilde{J}}-I} \right)\\ 
\oplus \Ker \left( {A_{\widetilde{J}}-I} \right)\end{matrix}}}\ar[r] &
\tiny{{\begin{matrix}\overline{\Coker} \left({A_{\widetilde{J}/\widetilde{I}}-I}  \right)\\
\oplus \Ker \left({A_{\widetilde{J}/\widetilde{I}}-I} \right)\end{matrix}} } \ar[r] & \tiny{{\begin{matrix}\Coker \left({A_{\widetilde{I}}-I} \right)\end{matrix}}}\\
\overline{K}_1\big(L(F_{H_{\widetilde{I}}})\big)\ar[u]^{\hskip .2in{\txt{\tt $\widetilde{Y_1}$} \hskip -.4in}}\ar[r]\ar[d]_{ \hskip .2in{\txt{\tt $\widetilde{X_1}$} \hskip -.4in}} & \overline{K}_1\big(L(F_{H_{\widetilde{J}}})\big)\ar[u]^{\hskip .2in{\txt{\tt $\widetilde{Y_2}$} \hskip -.4in}}\ar[r]\ar[d]_{ \hskip .2in{\txt{\tt $\widetilde{X_2}$} \hskip -.4in}} & \overline{K}_1\big(L({(F/H_{\widetilde{I}})}_{H_{{\widetilde{J}}/{\widetilde{I}}}})\big)\ar[d]_{ \hskip .2in{\txt{\tt $\widetilde{X_3}$} \hskip -.4in}}\ar[u]^{\hskip .2in{\txt{\tt $\widetilde{Y_3}$} \hskip -.4in}}\ar[r] & K_0(L(F_{H_{\widetilde{I}}}))\ar[d]\ar[u]^{\widetilde{\pi}}\\
\overline{K}_1({\widetilde{I}})\ar[r]^{}& \overline{K}_1({\widetilde{J}})\ar[r]^{} & \overline{K}_1({\widetilde{J}}/{\widetilde{I}})\ar[r]^{} & K_0({\widetilde{I}})}
\end{equation}

By Lemma \ref{lemma-pre} the squares $X_1$ and $X_2$ in \eqref{kkkk} are commutative. Recall the maps  in the third row of the diagram \eqref{kkkk} from Lemma \ref{ssk} and Lemma \ref{ssz}. By Proposition \ref{propforquotientcom} the squares $Y_1$ and $Y_2$ in \eqref{kkkk} are commutative.



Now we define $\Theta_I$. We denote $k^{\times}/\{-1,1\}$ by $\overline{k}^{\times}$.  We have the composition of the following maps, denoted by $\Lambda_I=\begin{pmatrix}\Lambda_I^{11}&0\\
0&\Lambda_I^{22}\end{pmatrix}$,
\begin{align*}
\xymatrix@C-0.01pc@R+1pc{
\ol{\Coker}(A_I-I)\oplus \Ker(A_I-I) \ar[r]^>>>>>{\tiny{\begin{pmatrix}\Lambda &0\\ 0&\Lambda'\end{pmatrix}}}&(K_0(I)\otimes_{\Z}\ol{k}^{\times})\oplus \Ker(\phi_I)\ar[d]^{\tiny{\begin{pmatrix}\alpha_{I,0}\otimes {\ol{k}^{\times}} &0\\ 0& \Sigma\end{pmatrix}}}\\
\ol{\Coker}(A_{\widetilde{I}}-I)\oplus \Ker(A_{\widetilde{I}}-I) &(K_0({\widetilde{I}})\otimes_{\Z}{\ol{k}^{\times}})\oplus \Ker(\phi_{\widetilde{I}})\ar[l]^>>>>>{\tiny{\begin{pmatrix}\widetilde{\Lambda}^{-1} &0\\ 0&\widetilde{\Lambda}'^{-1}\end{pmatrix}}}.
}
\end{align*}  The map $\Lambda: \ol{\Coker}(A_{I}-I: {\ol{k}^{\times}}^{R}\xra {\ol{k}^{\times}}^{E_{H_I}^0})\xra K_0(I)\otimes_{\Z} {\ol{k}^{\times}}$ with $R$ the set of non-sink vertices in $E_{H_I}$ is the composition 
\[\ol{\Coker}(A_{E_{H_I}}-I: {\ol{k}^{\times}}^{R}\longrightarrow {\ol{k}^{\times}}^{{E}_{H_I}^0})\cong \Coker(A_{E_{H_I}}-I: \Z^{R}\longrightarrow\Z^{{E}_{H_I}^0})\otimes_{\Z} \ol{k}^{\times}\cong K_0(I)\otimes_{\Z}\ol{k}^{\times}.\] We emphasize here that $\ol{\Coker}(A_I-I)$ denotes by $\ol{\Coker}(A_{E_{H_I}}-I: {\ol{k}^{\times}}^{R}\longrightarrow {\ol{k}^{\times}}^{{E}_{H_I}^0})$.
 Similarly we have the map $\widetilde{\Lambda}: \ol{\Coker}(A_{F_{H_{\widetilde{I}}}}-I)\xra K_0(\widetilde{I})\otimes_{\Z} {\ol{k}^{\times}}$. The map $\Lambda':\Ker(A_{I}-I)\xra\Ker(\phi_I)$ is given by the composition of 
\[\Ker(A_{I}-I)\cong \Ker(K_0^{\gr}(L(E_{H_I}))\longrightarrow K_0^{\gr}(L(E_{H_I})))\cong \Ker(\phi_I)\] with the first isomorphism following from Proposition \ref{4itemses}. Similarly we have the map $\widetilde{\Lambda}': \Ker(A_{\widetilde{I}}-I)\xra\Ker(\phi_{\widetilde{I}})$. The map $\Sigma: \Ker(\phi_I)\xra \Ker(\phi_{\widetilde{I}})$ is given by \eqref{3d}. 

Note that $\Lambda_I^{11}$ and $\Lambda_I^{22}$ are isomorphisms. 
The homomorphism $\Theta_I$ is defined as $\begin{pmatrix}\Lambda_I^{11}&0\\
0&\Lambda_I^{22}\end{pmatrix}$. The homomorphisms $\Theta_J$ and $\Theta_{J/I}$ can be defined similarly and we omit the details here.

To show that $Z_1$ and $Z_2$ are commutative, note that by \eqref{3d} the squares in the following diagram \eqref{kkc} with blue colour are commutative. We define $l_I, l_J$ and $l_{J/I}$ as compositions of maps such that their squares in horizontal level commute. We can check directly that the two back squares involving $\tau$ and $\tau'$ are commutative. Thus the following diagram \eqref{kkc} is commutative. 
\begin{equation}\label{kkc}
\xymatrix@C-.7pc@R-.2pc{
&\Ker(A_I-I)\ar[rr]^{}\ar'[d][dd]^>>>>>>>>>{\tau}\ar[dl]_{l_I}&&\Ker(\phi_I)
\ar@[blue][dd]^{}   \ar@[blue][dl]_{}\\
\Ker(A_{\widetilde{I}}-I)\ar[rr]\ar[dd]&&\Ker(\phi_{\widetilde{I}})\ar@[blue][dd]^{}&\\
&\Ker(A_J-I)\ar'[r][rr]\ar[dl]_{l_J}\ar'[d][dd]^>>>>>>>>>{\tau'}&&\Ker(\phi_J)\ar@[blue][dd]^{}\ar@[blue][dl]_{}\\
\Ker(A_{\widetilde{J}}-I)\ar[rr]\ar[dd]&&\Ker(\phi_{\widetilde{J}})\ar@[blue][dd]^{}&\\
&\Ker(A_{J/I}-I)\ar'[r][rr]\ar[dl]_{l_{J/I}}&&\Ker(\phi_{J/I})\ar@[blue][dl]_{}\\
\Ker(A_{\widetilde{J}/\widetilde{I}}-I)\ar[rr]&&\Ker(\phi_{\widetilde{J}/\widetilde{I}})
}\end{equation} Similarly by \eqref{3d} the squares in the following diagram \eqref{ccc} with blue colour are commutative.


 
\begin{equation}\label{ccc}
\xymatrix@C-.7pc@R-0.2pc{
&\overline{\Coker}(A_I-I)\ar@[brown][dl]_{\Lambda_I}\ar@[brown]'[d][dd]\ar[rr]^{\Lambda}&&K_0(I)\otimes_{\Z}{\overline{k}^{\times}}\ar@[blue][dd]^{}\ar@[blue][dl]_{\alpha_{I,0}\otimes \overline{k}^{\times}}\\
\overline{\Coker}(A_{\widetilde{I}}-I)\ar@[brown][dd]\ar[rr]^>>>>>>>>>{\widetilde{\Lambda}}&&K_0({\widetilde{I}})\otimes_{\Z}{\overline{k}^{\times}}\ar@[blue][dd]^{}&\\
& \overline{\Coker}(A_J-I)\ar@[brown][dl]\ar@[brown]'[d][dd]\ar'[r][rr]&&K_0(J)\otimes_{\Z}{\overline{k}^{\times}}\ar@[blue][dd]^{}\ar@[blue][dl]_{}\\
\overline{\Coker}(A_{\widetilde{J}}-I)\ar@[brown][dd]\ar[rr]&&K_0(\widetilde{J})\otimes_{\Z}{\overline{k}^{\times}}\ar@[blue][dd]^{}&\\
& \overline{\Coker}(A_{J/I}-I)\ar@[brown][dl]\ar'[r][rr]&&K_0(J/I)\otimes_{\Z}{\overline{k}^{\times}}\ar@[blue][dl]_{}\\
\overline{\Coker}(A_{\widetilde{J}/\widetilde{I}}-I)\ar[rr]&&K_0(\widetilde{J}/\widetilde{I})\otimes_{\Z}{\overline{k}^{\times}}
}\end{equation} Observe that the two back squares in \eqref{ccc} are commutative  as we have the following commutative diagram
\begin{equation}
\label{adddiag}
\xymatrix{
\overline{\Coker}(A_I-I)\ar[r]^>>>>>>{\scriptscriptstyle{\cong\ \ }}\ar[d]^{\overline{\sigma}}& \Coker(A_I-I)\otimes_{\mathbb Z}\overline{k}^{\times}\ar[r]^>>>>>>{\scriptscriptstyle{\cong\ \ }}\ar[d]^{\overline{\tau}\otimes {\rm id}} &K_0(I)\otimes_{\mathbb Z} \overline{k}^{\times}\ar[d]^{\alpha_{I,0}\otimes {\rm id}}\\
\overline{\Coker}(A_J-I)\ar[r]^>>>>>>{\scriptscriptstyle{\cong\ \ }}\ar[d]^{\overline{\sigma'}} & \Coker(A_J-I)\otimes_{\mathbb Z}\overline{k}^{\times}\ar[r]^>>>>>>{\scriptscriptstyle{\cong\ \ }}\ar[d]^{\overline{\tau'}\otimes {\rm id}} &K_0(J)\otimes_{\mathbb Z} \overline{k}^{\times}\ar[d]^{\alpha_{J,0}\otimes {\rm id}}\\
\overline{\Coker}(A_{J/I}-I)\ar[r]^>>>>>>{\scriptscriptstyle{\cong\ \ }}& \Coker(A_{J/I}-I)\otimes_{\mathbb Z}\overline{k}^{\times}\ar[r]^>>>>>>{\scriptscriptstyle{\cong\ \ }}&K_0(J/I)\otimes_{\mathbb Z} \overline{k}^{\times}.}
\end{equation} Here the two maps $\sigma$ and $\sigma'$ in \eqref{adddiag} are naturally induced by $\sigma: {k^{\times}}^{H_I}\xrightarrow{} {k^{\times}}^{H_J} $ and by $\sigma': {k^{\times}}^{H_J}\xrightarrow{} {k^{\times}}^{H_J\setminus H_I}$ given in Lemma \ref{ssk}, respectively. The maps $\tau$ and $\tau'$ in \eqref{adddiag} are given in Lemma \ref{ssz}.  Thus the squares with brown colour in the diagram \eqref{ccc} are commutative. By \eqref{kkc} and \eqref{ccc} the squares $Z_1$ and $Z_2$ in \eqref{kkkk} are commutative.

\medskip

 {\bf Step III:}  We  define the maps in the fourth column of \eqref{kkkk} and  show that $\alpha_{I, 0}$ is the composition of them. We check that  the square $C$ in \eqref{ktheoryp} is commutative.
 
 We have the isomorphism $\pi: K_0(L(E_{H_I}))\xra \Coker(A_I-I)$, 
 sending $v$ to ${\bf e}_v$. Similarly we have the isomorphism
 $\widetilde{\pi}: K_0(LF_{H_{\widetilde{I}}})\xra \Coker(A_{\widetilde{I}}-I)$. The isomorphism $\Theta: \Coker(A_I-I)\xra \Coker(A_{\widetilde{I}}-I)$ is defined to be $\widetilde{\pi}\circ \alpha_{I,0}\circ \pi^{-1}$. Obviously $\alpha_{I, 0}$ is the composition of the maps in the fourth column of \eqref{kkkk}.

 We check that the square $C$ in \eqref{ktheoryp} is commutative. It suffices to show that the squares $X_3, Y_3$ and $Z_3$ in \eqref{kkkk} are commutative. By Lemma \ref{lemma-pre}, we have that the square $X_3$ in \eqref{kkkk} is commutative.

 In order to define the map $\delta_{E_{H_I}, E_{H_J}}: \Ker(A_{J/I}-I)\xra \Coker(A_{I}-I)$, note that $H_J\setminus H_I$ is a hereditary saturated subset of $(E/H_I)^0$ and $J/I=\langle H_J\setminus H_I\rangle$ is a graded ideal of the Leavitt path algebra $L(E)/I\cong L(E/H_I)$. Then we have $H_{J/I}=H_J\setminus {H_I}$ and the restriction graph $$(E/H_{I})_{H_{J/I}}:=(E/H_{I})_{H_J\setminus {H_I}}.$$ Observe that we have $(E/H_{I})_{H_{J/I}}=E_{H_J}/{H_I}.$ By Lemma \ref{ssz} the 
 connecting map $\delta_{E_{H_I}, E_{H_J}}$ is defined.

 To show that the square $Y_3$ in \eqref{kkkk} is commutative, by Lemma \ref{fact} we need to show that the following diagram is commutative with $\partial \circ \chi_1=\pi^{-1}\circ \delta_{E_{H_I}, E_{H_J}}$:
 \begin{equation}
\xymatrix@C+.1pc@R+1.1pc{0\ar[r]&\ol{\Coker}{(A_{J/I}-I)}\ar[r]^{\lambda}\ar[d]^{}&\overline{K}_1(L(E/H_I)_{H_{J/I}})\ar[r]^{\xi}\ar[d]^{\partial}&\Ker{(A_{J/I}-I)}\ar@{.>}@/^/[l]^>>>{\chi_1}\ar[r]\ar[d]^{\delta_{E_{H_I}, E_{H_J}}}&0\\
0\ar[r]&0\ar[r]^{}& K_0(L(E_{H_I})) \ar[r]^{\pi}&\Coker(A_{I}-I)\ar[r]&0.
}
\end{equation} Here $\partial \colon \overline{K}_1(L(E/H_I)_{H_{J/I}})\to K_0(L(E_{H_I}))$ is the map appearing in diagram \eqref{kkkk}.
We can check directly that $\partial \circ \lambda=0$. We only need to show that $\delta_{E_{H_I}, E_{H_J}}\circ\xi=\pi\circ \partial$, as it implies $\pi^{-1}\circ \delta_{E_{H_I}, E_{H_J}} =\pi^{-1}\circ\delta_{E_{H_I}, E_{H_J}}\circ\xi\circ\chi_1= \partial\circ\chi_1$ (see Proposition \ref{newprop}(2)). Observe that we have the following commutative diagram 
\begin{equation}\label{ll}
\xymatrix{0\ar[r]&I\ar[r]&J\ar[r]&J/I\ar[r]&0\\
&L(E_{H_I})\ar[u]^{}&L(E_{H_J})\ar[u]^{}&L({(E/H_I)}_{H_{J/I}})\ar[u]^{}&\\
0\ar[r]& \mathcal{K}(({E/{H_I}})_{H_{J/I}})\ar[r]\ar[u]^{\wp}&C(({E/{H_I}})_{H_{J/I}})\ar[r]\ar[u]^{\wp}&L(({E/H_I)}_{H_{J/I}})\ar[u]\ar[r]&0,
}
\end{equation} 
where the map $\wp$ is given by $\wp(v)=v, \wp(e)=e$ and $\wp(e^*)=e^*$ for $v\in ({E/{H_I}})_{H_{J/I}}^0$ and $e\in ({E/{H_I}})_{H_{J/I}}^1$. 
Note that we have the following equality in $J$, for each $v\in (H_J\setminus H_I)\cap R'$ with $R'$ the set of non-sink vertices in $E_{H_J}$,
$$ v - \sum_{e\in s^{-1}(v),  r(e)\notin H_I} ee^* = \sum _{f\in s^{-1}(v), r(f)\in H_I} ff^*.$$
We denote by $\mu \colon \Ker (A_{J/I}-I) \to K_0(\mathcal{K}(({E/{H_I}})_{H_{J/I}}))$ the natural inclusion map.

By \eqref{ll} and the naturality of the connecting map in $K$-theory, we have the following commutative diagrams:

\begin{equation}
\xymatrix@C+.1pc@R+1.1pc{K_1(L(({E/{H_I}})_{H_{J/I}}))\ar[d]\ar[r]^{\partial'}&K_0(L(E_{H_I}))\ar[d]\\
\overline{K}_1(L(({E/{H_I}})_{H_{J/I}}))\ar[r]^{\partial}&K_0(L(E_{H_I}))
}	\end{equation}
and
\begin{equation}\label{sl}
\xymatrix@C+.1pc@R+1.1pc{
K_1(L(({E/{H_I}})_{H_{J/I}}))\ar[r]^{\xi'}\ar[d]^{\partial'}& \Ker(A_{J/I}-I)\ar[r]^{\mu}&K_0(\mathcal{K}(({E/{H_I}})_{H_{J/I}}))\ar[r]^>>>>>{K_0(\wp)}\ar[d]^{}\ar[d]^{K_0(\wp)}& K_0(L(E_{H_I}))\ar[r]^>>>>>>>{\pi}& \Coker(A_I-I)\ar[d]^{\text{id}}\\
K_0(L(E_{H_I}))\ar[rr]^{\text{id}}&&K_0(L(E_{H_I}))\ar[rr]^{\pi}&&\Coker(A_I-I).
}
\end{equation} 
We claim that $\delta_{E_{H_I}, E_{H_J}}\circ\xi'=\pi\circ \partial'$. By \eqref{sl} it suffices to prove that $\delta_{E_{H_I}, E_{H_J}} = \pi\circ K_0(\wp)\circ \mu $. 
Recall from \eqref{ktheorymaps} that $[v-\sum_{e\in s^{-1}(v), r(e)\notin H_I}ee^*]_0$, for $v\in (H_J\setminus H_I)\cap R'$, are the canonical generators of the free abelian group 
$K_0(\mathcal{K}(({E/{H_I}})_{H_{J/I}}))$. 
Now for $v\in (H_J\setminus H_I)\cap R'$ we have
\begin{equation}
\begin{split}
(\pi\circ K_0(\wp))([v-\sum_{\substack{e\in s^{-1}(v)\\ r(e)\notin H_I}}ee^*]_0)
=\sum_{\substack{f\in s^{-1}(v)\\ r(f)\in H_I}} [{\bf{e}}_{r(f)}].
\end{split}
\end{equation}
On the other hand, it follows from Lemma \ref{ssz} that  
$$\delta_{E_{H_I}, E_{H_J}}(x) = [X^tx]$$
for every $x\in \Ker (A_{J/I}-I)$, where $X$ is the matrix such that $X(v,w)=A_E(v,w)$ for each $v\in (H_J\setminus H_I)\cap R'$ and $w\in H_I$. It follows that
$\delta_{E_{H_I}, E_{H_J}} = \pi\circ K_0(\wp)\circ \mu $. So the proof for the claim is completed. By \eqref{thenewxi}, we have $\delta_{E_{H_I}, E_{H_J}}\circ\xi=\pi\circ \partial$ as desired. 

To show that $Z_3$ in \eqref{kkkk} is commutative, by \eqref{3d} and the Snake Lemma, we have connecting maps 
$\rho:\Ker(\phi_{J/I})\rightarrow K_0(I)$ and $\widetilde{\rho}: \Ker (\phi_{\widetilde{J}/\widetilde{I}})\rightarrow K_0(\widetilde{I})$ 
such that the blue square in the following diagram is commutative (similarly as in Lemma \ref{2dsnakelemma}). 

\begin{equation}\label{zzz}
\xymatrix@C-1pc@R-0.3pc{&\Ker(\phi_{J/I})\ar@[blue][rr]^{\rho}\ar@[blue]'[d][dd]&&K_0(I)\ar@[blue][dd]^{}\\
\Ker(A_{J/I}-I)\ar[rr]^{\delta_{E_{H_I}, E_{H_J}}}\ar[ur]^{\psi}\ar[dd]^{\Theta_{J/I}}&&\Coker(A_I-I)\ar[ur]^{}\ar[dd]_<<<<<<<<{\Theta}&\\
&\Ker(\phi_{\widetilde{J}/\widetilde{I}})\ar@[blue]'[r][rr]^>>>>>>>>>>{\widetilde{\rho}}&&K_0(\widetilde{I})\\
\Ker(A_{\widetilde{J}/\widetilde{I}}-I)\ar[ur]\ar[rr]_{\delta_{F_{H_{\widetilde{I}}}, F_{H_{\widetilde{J}}}}}&&\Coker(A_{\widetilde{I}}-I)\ar[ur]^{}&}\end{equation} We need to check that the top and bottom faces in the above diagram \eqref{zzz} are commutative.  Hence we obtain that $Z_3$ in \eqref{kkkk} which is the front face of the above diagram is commutative. 
By Lemma \ref{2dsnakelemma} we have $\rho\circ \psi=\pi^{-1}\circ \delta_{E_{H_I}, E_{H_J}}$, implying that the top face in the above diagram \eqref{zzz} is commutative as well. This completes the proof to show that \eqref{ktheoryp} is commutative.

In order to have the commutative diagram given in \eqref{last}, 
we take graded ideals $I\sub J\sub P$ of $L(E)$. Then we have $J/I\sub P/I$ which are graded ideals of $L(E)/I\cong L(E/H_I)$. By Lemma \ref{hfbcvgfbfg567}, we have an order-preserving $\Z[x, x^{-1}]$-isomorphism $K_0^{\gr}(L(E/H_I))\cong K_0^{\gr}(L(F/H_{\widetilde{I}}))$. Similarly as the argument to obtain the commutative diagram \eqref{ktheoryp}, we have the commutative diagram \eqref{last}. 
\end{proof}

\begin{rmk}
	\label{remark:invariantFKbar}
	We have shown in Theorem \ref{maintheorem} that $\FKbar_{0,1}^+$ is a {\it graded invariant} for Leavitt path algebras of row-finite graphs. Indeed if $g\colon L_k(E)\to L_k(F)$ is a graded isomorphism of $k$-algebras, then $g$ induces a $\Z[x,x^{-1}]$-module order isomorphism between the $K_0^{\mathrm{gr}}$-groups, and thus we get an isomorphism $\FKbar _{0,1}^+(L_k(E))\cong \FKbar _{0,1}^+(L_k(F))$ by Theorem \ref{maintheorem}. It is easily checked that $\FKbar_{0,1}$ is indeed an algebra  invariant. If $h\colon L_k(E)\to L_k(F)$ is an algebra homomorphism, $v\in E^0$ and $\lambda \in k^{\times}$, then we can find $w_1,\dots , w_n\in F^0$ such that $\text{diag}(h(v),0,\dots ,0)\sim \text{diag}(w_1, \cdots ,w_n)$ in $M_n(L_k(F))$ for some $n\ge 1$. We can therefore find an invertible matrix  $U\in M_{2n}(\widetilde{L_k(F)})$ such that   $\text{diag}(h(v),0,\dots ,0) = U \text{diag}(w_1, \cdots , w_n,0,\dots, 0)U^{-1}$. Now we have
	$$  \text{diag}(\lambda h(v)+(1-h(v)),1,\dots ,1) = U \text{diag}(\lambda w_1+(1-w_1), \cdots , \lambda w_n+(1-w_n),1,\dots, 1)U^{-1},$$
	which shows that $h_*([\lambda v +(1-v)]_1)= \sum_{i=1}^n [\lambda w_i+(1-w_i)]_1$. In particular we have that $h_*(G_E)\subseteq G_F$, and so if $h$ is an algebra isomorphism, then  
	$h_*$ induces a natural isomorphism $\ol{K}_1(L_k(E))\cong \ol{K}_1(L_k(F))$. 
		\end{rmk}

Recall that, in the analytic case, filtered $K$-theory of graph $C^*$-algebras is built from the space of all prime gauge invariant ideals of $C^*(E)$, denoted by $\Prime_{\gamma}(C^*(E))$,  and the $K$-groups involved are topological $K$-theory groups as developed in \cite{errs3}.

The following theorem which relates the graded $K$-theory of Leavitt path algebras to the filtered $K$-theory of their corresponding graph $C^*$-algebras will be used in \S\ref{gdtbryr777}. 

\begin{thm}\label{lateradd}Let $E, F$ be row-finite graphs with no sinks and $\mathbb C$ the field of complex numbers. Suppose that there exists an order-preserving $\Z[x, x^{-1}]$-isomorphism
\[\varphi: K_0^{\gr}(\L_{\mathbb C}(E))  \longrightarrow  K_0^{\gr}(\L_{\mathbb C}(F)).\] Then there is an isomorphism   \[{\FK}_{0,1}({\Prime_{\gamma}}(C^*(E)); C^*(E))\cong {\FK}_{0,1}({\Prime_{\gamma}}(C^*(F)); C^*(F)).\] 
\end{thm}
\begin{proof} Let $I\subseteq J\subseteq P$ be graded ideals of $L_{\mathbb C}(E)$. Whenever we have a $*$-homomorphism $\iota_{\mathscr{U}}\colon A\to \mathscr{U}$ from a $*$-algebra $A$ to a $C^*$-algebra $\mathscr{U}$, we denote the composition $$
K_n(A)\stackrel{K_n(\iota_{\mathscr{U}})}{\longrightarrow} K_n(\mathscr{U})\stackrel{}{\longrightarrow} K_n^{\rm top}(\mathscr{U})$$ by $\gamma_{n,\mathscr{U}}$, where $K_n^{\rm top}(\mathscr{U})$ is the topological K-theory of the $C^*$-algebra $\mathscr{U}$.
By \cite[Theorem 4]{errs},  excluding the brown arrows, we have the commutative diagram (with black arrows).  \begin{equation}
\label{forsymbolic}
\xymatrix@C+.1pc@R+1.1pc{
K_1(J/I)\ar[r]^{}\ar[d]^{p} \ar[r]^{}\ar@{.>}@/_2pc/@[blue][dd]_<<<<<<<<<{\gamma_{1,J/I}}& K_1(P/I)\ar[r]^{} \ar[d]^{p}\ar[r]^{}\ar@{.>}@/_2pc/@[blue][dd]_<<<<<<<<<{\gamma_{1,P/I}}& K_1(P/J)\ar[r]^{} \ar[d]^p{}\ar[r]^{}\ar@{.>}@/_2pc/@[blue][dd]_<<<<<<<<<{\gamma_{1,P/J}} & K_0(J/I)\ar[r]^{}\ar[d]^{\rm id}\ar[r]^{}
&K_0(P/I)\ar[r]^{}\ar[d]^{\rm id}
&K_0(P/J)\ar[d]^{\rm id}\\
\overline{K}_1(J/I)\ar@[brown][r]^{}\ar[d]^{\beta_{1,J/I}} \ar@[brown][r]^{}
& \overline{K}_1(P/I)\ar@[brown][r]^{} \ar[d]^{\beta_{1,P/I}}
& \overline{K}_1(P/J)\ar@[brown][r]^{} \ar[d]^{\beta_{1,P/J}} & K_0(J/I)\ar@[brown][r]^{}\ar[d]^{\gamma_{0,J/I}}&K_0(P/I)\ar@[brown][r]^{}\ar[d]^{\gamma_{0,P/I}}&K_0(P/J)\ar[d]^{\gamma_{0,P/J}}\\
K^{\rm top}_1(\widetilde{J}/\widetilde{I})\ar[r]^{}& K^{\rm top}_1(\widetilde{P}/\widetilde{I})\ar[r]^{} & K^{\rm top}_1(\widetilde{P}/\widetilde{J})\ar[r]^{} & K^{\rm top}_0(\widetilde{J}/\widetilde{I})\ar[r]^{}&K^{\rm top}_0(\widetilde{P}/\widetilde{I})\ar[r]^{}&K^{\rm top}_0(\widetilde{P}/\widetilde{J}).
}
\end{equation}We claim that the three blue (dotted) maps $\gamma_{1,*}$ factor through the natural projection map $p$ (to show $\gamma_{1,*}=\beta_{1,*}\circ p$).
But this is obvious because the elements $[-v]_1\in K_1(L_{\mathbb C}(H))$, for $v\in H$, are sent to $0$ through the canonical map $K_1(L_{\mathbb C}(H))\to K_1^{\rm top}(C^*(H))$ for every graph $H$. (Just consider any continuous path $(\lambda_t)$ in $\mathbb T$ connecting $-1$ with $1$ and the corresponding path of unitaries $(\lambda_tv+(1-v))_t$ in $\widetilde{C^*(H)}$.) 
 We observe that the squares in the first level of \eqref{forsymbolic} commutes. Hence by the claim, so do the squares in the second level of \eqref{forsymbolic}. 

Similarly as the proof of \cite[Theorem 8.4(2)]{abramstom}, we have the statement: if $\overline{K}_1(L_{\mathbb C}(E))\cong \overline{K}_1(L_{\mathbb C}(F))$, then $K^{\rm top}_1(C^*(E))\cong K^{\rm top}_1(C^*(F))$. Combining with Theorem \ref{maintheorem}, we have the following commutative diagram with the vertical maps isomorphisms (the $K_0$ parts are omitted).
\begin{equation}\label{comparefortop}
\xymatrix@C1.1pc@R1.1pc{&\overline{K}_1(J/I)\ar[dl]\ar[rr]\ar'[d][dd]&&\overline{K}_1(P/I)\ar[rr]^{}\ar'[d][dd]^>>>>>>>>>*{}\ar[dl]&&\overline{K}_1(P/J)\ar[dd]^{}\ar[dl]_{}\\
K^{\rm top}_1(J^{\rm top}/I^{\rm top})\ar[dd]\ar[rr]&&K^{\rm top}_1(P^{\rm top}/I^{\rm top})\ar[rr]\ar[dd]&&K^{\rm top}_1(P^{\rm top}/J^{\rm top})\ar[dd]^{}&\\
& \overline{K}_1(\widetilde{J}/\widetilde{I})\ar'[r][rr]\ar[dl]&&\overline{K}_1(\widetilde{P}/\widetilde{I})\ar'[r][rr]\ar[dl]&&\overline{K}_1(\widetilde{P}/\widetilde{J})\ar[dl]_{}\\
K^{\rm top}_1(\widetilde{J}^{\rm top}/\widetilde{I}^{\rm top})\ar[rr]&&K^{\rm top}_1(\widetilde{P}^{\rm top}/\widetilde{I}^{\rm top})\ar[rr]&&K^{\rm top}_1(\widetilde{P}^{\rm top}/\widetilde{J}^{\rm top})&
}\end{equation} Then the proof of the theorem is completed.
\end{proof}


%
%

\section{Shift equivalence and the stability of corresponding graph $C^*$-algebras}\label{gdtbryr777}

Let $A$ be an integral $n\times n$ matrix with nonnegative entries. Consider the following directed system of abelian groups with $A$ acting as an order-preserving group homomorphism 
\[\mathbb Z^n \stackrel{A}{\longrightarrow} \mathbb Z^n \stackrel{A}{\longrightarrow}  \mathbb Z^n \stackrel{A}{\longrightarrow} \cdots,
\]
where the ordering in $\mathbb Z^n$ is defined point-wise.  The direct limit of this system, $\Delta_A:= \varinjlim_{A} \mathbb Z^n$,  along with its positive cone, $\Delta^+= \varinjlim_{A} \mathbb N^n$, and the automorphism which is induced by multiplication by $A$ on the direct limit, 
$\delta_A:\Delta_A \rightarrow \Delta_A$, is called \emph{Krieger's dimension group}. Following~\cite{lindmarcus}, we denote this triple by $(\Delta_A, \Delta_A^+, \delta_A)$.  It can be shown that two matrices $A$ and $B$ are shift equivalent if and only if their associated Krieger's dimension groups are isomorphic (\cite[Theorem~4.2]{krieger1}, and~\cite[Theorem~7.5.8]{lindmarcus}, see also~\cite[\S7.5]{lindmarcus} for a detailed algebraic treatment). We say two finite graphs $E$ and $F$ are  \emph{shift equivalent} if their adjacency matrices $A_E$ and $A_F$ are shift equivalent. If the matrices $A_E$ and $A_F$ are shift equivalent then by~\cite[Theorem~7.4.17]{lindmarcus}, their Bowen-Franks groups are isomorphic, namely, 
$\BF(A_{E})\cong \BF(A_{F})$. On the other hand, by~\cite[Exercise~7.4.4, for $p(t)=1-t$]{lindmarcus},
$\det(1-A_{E})=\det(1-A_{F})$. Recall that a graph $E$ is \emph{irreducible} if given any two vertices $v$ and $w$ in $E$, there is a path from  $v$ to $w$. Now if the graphs $E$ and $F$ are irreducible, the main theorem of Franks~\cite{franks} gives that $A_{E}$ is flow equivalent to $A_{F}$. Thus $F$ can be obtained from $E$ by a finite sequence of 
in/out-splitting and expansion of graphs (see~\cite{parrysullivan,batespask}). 
Each of these transformation preserve the Morita equivalence (\cite{batespask}) and thus $C^*(E)$ is Morita equivalent to $C^*(F)$. 

Our result now allows us to extend this fact from finite irreducible graphs to finite graphs without sinks (see the Figure~\ref{tgftgrtgete3}).  

\begin{prop}\label{bfg1998d}
Let $E$ and $F$ be finite graphs with no sinks. If $E$ and $F$ are shift equivalent, then the $C^*$-algebras $C^*(E)$ and $C^*(F)$ are Morita equivalent. 
\end{prop}
\begin{proof}
Since $E$ and $F$ are shift equivalent, we have an isomorphism of Krieger's dimension groups 
\[ (\Delta_E, \Delta_E^+, \delta_E) \cong (\Delta_F, \Delta_F^+, \delta_F).\]
But since Krieger's dimension group for the graph $E$ coincides with the graded Grothendieck group $K_0^{\gr}(L(E))$ (~\cite[Lemma~11]{haz3}), we obtain an order-preserving $\Z[x,x^{-1}]$-module isomorphism $K_0^{\gr}(L(E))\cong K_0^{\gr}(L(F))$. By Theorem~\ref{lateradd}, the filtered $K$-theory of the corresponding graph $C^*$-algebras are  isomorphic. Now the main theorem of \cite{errs3}
gives that the $C^*$-algebras $C^*(E)$ and $C^*(F)$ are Morita equivalent. 
\end{proof}

\section{Acknowledgements} The authors would like to acknowledge Australian Research Council grant DP160101481. They would like to thank Guillermo  Corti\~nas, S\o ren Eilers and Petter Nyland for discussions and valuable comments on this work.


\begin{thebibliography}{9999}

\bibitem{ap} G. Abrams, G. Aranda Pino, {\it The Leavitt path algebra of a graph}, J.
    Algebra $\mathbf{293}$ (2) (2005), 319--334.
   
\bibitem{abramstom} G. Abrams, M. Tomforde, {\it Isomorphism and Morita equivalence of graph algebras}, Trans.
Amer. Math. Soc. $\mathbf{363}$ (2011), 3733--3767.   
    
 
 \bibitem{lpabook} G. Abrams, P. Ara, M. Siles Molina, Leavitt path algebras,  Lecture Notes in Mathematics, vol. 2191, Springer Verlag, 2017.


%

\bibitem{abc} P. Ara, M. Brustenga, G. Corti\~{n}as, {\it K-theory of Leavitt path algebras}, M\"unster J. Math. $\mathbf{2}$ (2009), 5--33.

%
%

\bibitem{ag}P. Ara, K.R. Goodearl, {\it Leavitt path algebras of separated graphs,}
    J. Reine Angew. Math. $\mathbf{669}$ (2012), 165--224.


\bibitem{amp} P. Ara, M.A. Moreno, E. Pardo, {\it Nonstable K-theory for graph algebras,}
Algebr. Represent. Theory $\mathbf{10}$ (2) (2007), 157--178.

\bibitem{arapardo} P. Ara, E. Pardo, {\it Towards a K-theoretic characterization of graded isomorphisms between Leavitt path algebras}, J. K-Theory,  {\bf 14} (2014), no. 2, 203--245.

\bibitem{ahls} P. Ara, R. Hazrat, H. Li, A. Sims, {\it Graded Steinberg algebras and their representations},  Algebra \& Number theory, {\bf 12-1} (2018), 131--172.




\bibitem{batespask} T. Bates, D. Pask, \emph{ Flow equivalence of graph algebras}, Ergodic Theory Dynam. Systems {\bf 24} (2004), no. 2, 
367--382.

\bibitem{bort}M. Brustenga i Bort, {\it $\grave{A}$lgebres associades a un buirac}, Ph.D. Thesis, Universitat Auto\'noma de Barcelona, 2007. 


\bibitem{cet} T. M. Carlsen, S. Eilers, M. Tomforde, {\it Index maps in the K-theory of graph algebras}, J. K-theory $\mathbf{9}$ (2012), 385--406.

\bibitem{cortinas} G. Corti\~nas, Algebraic v. topological K-theory: a friendly match. In Topics in algebraic and topological K-theory, volume 2008 of Lecture Notes in Math., pages 103--165. Springer, Berlin, 2011.


\bibitem{cortinasmontero} G. Corti\~nas, D. Montero, {\it Algebraic bivariant K-theory and Leavitt path algebras}, to appear in J. Noncommutative Geometry; arXiv:1806.09204v2.

\bibitem{cortinasmontero2}, G. Cortiñas, D. Montero, {\it Homotopy classification of Leavitt path algebras}, Adv. Math. $\mathbf{362}$ (2020), 106961, 26 pp.


\bibitem{errs4} S. Eilers, G. Restorff, E. Ruiz,  {\it On graph $C^*$-algebras with a linear ideal lattice,} Bull. Malays. Math. Sci. Soc. $\mathbf{33}$ (2) (2010), no. 2, 233--241.


\bibitem{errs2} S. Eilers, G. Restorff, E. Ruiz, A.P.W. S\o rensen, {\it The complete classification of unital graph $C^*$-algebras: Geometric and strong,} arXiv:1611.07120.

\bibitem{errs3} S. Eilers, G. Restorff, E. Ruiz, A.P.W. S\o rensen, {\it  Geometric classification of graph $C^*$-algebras over finite graphs,} 
Canadian Journal of Mathematics, $\mathbf{70}$ (2018), 294--353.

\bibitem{errs} S. Eilers, G. Restorff, E. Ruiz,  A.P.W. S\o rensen, {\it Filtered K-theory for graph algebras}, 2016 MATRIX Annals, MATRIX Book Series 1,  229--249.

\bibitem{franks} J. Franks, \emph{Flow equivalence of subshifts of finite type}, Ergodic Theory Dynam. Systems {\bf 4} (1984) 53--66.



\bibitem{grtw} J. Gabe, E. Ruiz, M. Tomforde,  T. Whalen, {\it K-theory for Leavitt path algebras: Computation and classification}, J. Algebra $\mathbf{433}$ (2015), 35--72.


\bibitem{goodearlbook} K.R. Goodearl, von Neumann regular rings, 2nd ed., Krieger Publishing Co., Malabar, FL, 1991.



\bibitem{roozbehhazrat2013} R. Hazrat, {\it The graded Grothendieck group and the classification of Leavitt path algebras}, Math. Annalen $\mathbf{355}$ (2013), 273--325.

\bibitem{haz2013} R. Hazrat, {\it A note on the isomorphism conjectures for Leavitt path algebras,}
J. Algebra $\mathbf{375}$ (2013), 33--40. 

\bibitem{haz3} R. Hazrat, {\it The dynamics of Leavitt path algebras}, J. Algebra $\mathbf{384}$ (2013), 242--266.


\bibitem{haz} R. Hazrat, Graded rings and graded Grothendieck groups, {\it London Math.
    Society Lecture Note Series,} Cambridge University Press, 2016.
  
  
\bibitem{hl2} R. Hazrat, H. Li, {\it  Homology of \'etale groupoids, a graded approach}, arXiv:1806.03398. 
    
   
   
  \bibitem{krieger1} W. Krieger, \emph{On dimension functions and topological Markov chains}, Invent. Math. {\bf 56} (1980), 239--250.    
   
%
\bibitem{ls} T. Y. Lam, M.K. Siu, {\it $K_0$
 and $K_1$--an introduction to algebraic K-theory}, Amer. Math. Monthly $\mathbf{82}$ (1975), 329--364. 
 
 
 \bibitem{lindmarcus} D. Lind, B. Marcus. An introduction to symbolic dynamics and coding, Cambridge University
Press, 1995.


\bibitem{parrysullivan} W. Parry, D. Sullivan, \emph{A topological invariant of flows on 1-dimensional spaces}, Topology {\bf 14} 
(1975), 297--299.

 
%
%
%
 
 
 
 
  \bibitem{rs}I. Raeburn, W. Szymanski, {\it Cuntz--Krieger algebras of infinite graphs and matrices}, Trans.
Amer. Math. Soc. $\mathbf{356}$ (2004), 39--59. 

\bibitem{rangaswamy} K.M. Rangaswamy, {\it The theory of prime ideals of Leavitt path algebras over arbitrary graphs}, J. Algebra $\mathbf{375}$ (2013), 73--96.

\bibitem{restorff} G. Restorff, {\it Classification of Cuntz-Krieger algebras up to stable isomorphism}, J. Reine Angew. Math. {\bf 598}  (2006), 185--210.


\bibitem{ror} M. R\o rdam, {\it Classification of extensions of certain $C^*$-algebras by their six term
 exact sequences in $K$-theory}, Math. Annalen {\bf 308} (1997), no. 1, 93--117. 


\bibitem{ror2} M. R\o rdam, Classification of nuclear, simple $C^*$-algebras, {\it Encyclopaedia of Mathematical
Sciences}, vol. 126, Springer, Berlin, 2001.


\bibitem{ror3} M. R\o rdam,  {\it Structure and classification of $C^*$-algebras}, International Congress of Mathematicians. Vol. II, 1581--1598, Eur. Math. Soc., Z\"urich, 2006.


\bibitem{rosenberg} J. Rosenberg, Algebraic K-theory and its applications,  {\it Graduate Text in Math.} 147, Springer-Verlag, New York, 1994.

\bibitem{rt} E. Ruiz, M. Tomforde, {\it Ideal-related K-theory for Leavitt path algebras and graph $C^*$-algebras}, Indiana Univ. Math. J. $\mathbf{62}$ (5) (2013), 1587--1620.

\bibitem{rt2014} E. Ruiz, M. Tomforde, {\it Ideals in graph algebras}, Algebr. Represent. Theor. $\mathbf{17}$ (2014), 849--861.


\bibitem{jackspielberg} J. Spielberg, {\it Semiprojectivity for certain purely infinite $C^*$-algebras}, Trans. Amer. Math. Soc. $\mathbf{361}$ (6) (2009), 2805--2830.

\bibitem{tomforde} M. Tomforde, {\it Uniqueness theorems and ideal structure for Leavitt
    path algebras}, J. Algebra $\mathbf{318}$ (2007), 270--299.
    
\bibitem{tomforde2} M. Tomforde, {\it  Classification of graph algebras: a selective survey}, Operator algebras and applications--the Abel Symposium 2015, 303--325, Abel Symp., 12, Springer, 2017. 


\end{thebibliography}
\end{document}